\documentclass[a4paper,11pt]{amsart}

\usepackage{amssymb}
\usepackage{amsmath}
\usepackage{amsthm,color}

\allowdisplaybreaks

\usepackage{amsfonts}
\usepackage{amssymb}
\usepackage{amsmath}
\usepackage{amsthm}
\usepackage{bbm}
\usepackage{verbatim}
\usepackage{url}
\usepackage[latin1]{inputenc} 
\usepackage[T1]{fontenc}

\usepackage{amsfonts}
\usepackage{amsmath,amssymb,amsthm}

\usepackage{mathrsfs}
\usepackage{verbatim}

\usepackage{graphicx}
\usepackage{ifpdf}

\newtheorem{thm}{Theorem}[section]
\newtheorem{prop}[thm]{Proposition}
\newtheorem{lem}[thm]{Lemma}

\theoremstyle{definition} 

\newtheorem{defn}[thm]{Definition}

\newtheorem{remark}[thm]{Remark}

\newcommand{\pa}{\varphi^t}

\newcommand{\R}{\mathbb{R}}
\newcommand{\Z}{\mathbb{Z}}
\newcommand{\N}{\mathbb{N}}

\newcommand{\G}{\mathcal{G}}

\newcommand{\bmo}{{\rm BMO}}

\newcommand{\cmo}{{\rm CMO}}

\def\cg{{\mathcal G}}

\newcommand{\GG}{ {\mathop \cg \limits^{    \circ}}_1}
\newcommand{\GGone}{ {\mathop \cg \limits^{    \circ}}_1}
\newcommand{\GGp}{{\mathop \cg\limits^{\circ}}_{1,\,1}}

\newcommand{\GGpone}{{\mathop \cg\limits^{\circ}}_{1,\,1}}

\newcommand{\GGpp}{{\mathop \cg\limits^{\circ}}_{1,\,1}}

\newcommand{\X}{\mathbb{X}}

\numberwithin{equation}{section}

\def\rr{{\mathbb R}}

\def\rrp{{\rr_+}}

\def\zz{{\mathbb Z}}

\def\cm{{\mathcal M}}
\def\cn{{\mathcal N}}
\def\crz{{\mathcal R}}

\def\fz{\infty}
\def\az{\alpha}
\def\supp{{\mathop\mathrm{\,supp\,}}}

\def\lz{\lambda}
\def\dz{\delta}

\def\ez{\epsilon}

\def\gz{{\gamma}}

\def\bgz{{\Gamma}}

\def\vz{\varphi}

\def\pa{\partial}
\def\wz{\widetilde}

\def\gfz{\genfrac{}{}{0pt}{}}
\def\dlz{{\Delta_\lz}}
\def\cms{{\mathcal M_S}}
\def\prz{{\partial}}
\def\gratx{{\nabla_{t,\,x}}}
\def\gratxo{{\nabla_{t_1,\,x_1}}}
\def\gratxt{{\nabla_{t_2,\,x_2}}}

\def\deltx{{\triangle_{t,\,x}}}
\def\rizo{{R_{\Delta_\lz,\,1}}}
\def\rizt{{R_{\Delta_\lz,\,2}}}
\def\slz{{\sharp_\lz}}
\def\qlz{{Q^{[\lz]}_t}}
\def\qlzo{{Q^{[\lz]}_{t_1}}}
\def\qlzt{{Q^{[\lz]}_{t_2}}}
\def\plz{{P^{[\lz]}_t}}
\def\plzo{{P^{[\lz]}_{t_1}}}
\def\plzzo{{P^{[\lz]}_{\dz_1}}}
\def\plzt{{P^{[\lz]}_{t_2}}}
\def\plzzt{{P^{[\lz]}_{\dz_2}}}

\def\wlzo{{W^{[\lz]}_{t_1}}}
\def\wlzt{{W^{[\lz]}_{t_2}}}
\def\prz{{\partial}}
\def\dmzo{{\,dm_\lz(x_1)}}
\def\dmzt{{\,dm_\lz(x_2)}}
\def\dmzy{{\,dm_\lz(y)}}

\def\dmzd{{\,d\mu_\lz(x_1,x_2)}}

\def\dmzdt{{\,d\mu_\lz(x_1,x_2)\,dt_1\,dt_2}}
\def\supd{{\sup_{\gfz{t_1>0}{t_2>0}}}}

\def\inzf{{\int_0^\fz}}
\def\dinzf{{\int_0^\fz\int_0^\fz}}
\def\inrp{{\int_{\R_+}}}
\def\dinrp{{\iint_{\R_+\times \R_+}}}
\def\dlmxt{{\frac{\dmz(x)\,dt}{t}}}
\def\dlmxto{{\frac{\dmz(x_1)\,dt_1}{t_1}}}
\def\dlmxtt{{\frac{\dmz(x_2)\,dt_2}{t_2}}}

\def\lozd{{L^1(\rlz)}}
\def\ltzd{{L^2(\rlz)}}
\def\lpzd{{L^p(\rlz)}}

\def\lrzd{{L^r(\rlz)}}
\def\lszd{{L^s(\rlz)}}
\def\horiz{{H^1_{Riesz}(\rlz)}}

\def\hpmz{{H^p_{\cm}(\rlz)}}
\def\hap{{H^p_{\Delta_\lz}(\rlz)}}
\def\hsu{{H^p_{S_u}(\rlz)}}

\def\hnz{{H^p_{\cn_{h}}(\rlz)}}

\def\hrz{{H^p_{\crz_{h}}(\rlz)}}

\def\hrp{{H^p_{\crz_{P}}(\rlz)}}

\def\hnp{{H^p_{\cn_{P}}(\rlz)}}

\def\slz{{\sharp_\lz}}

\def\pplz{t\partial_tP_t^{[\lambda]}}

\def\ls{\lesssim}
\def\gs{\gtrsim}

\def\tbz{{\triangle_\lz}}
\def\dmz{{dm_\lz}}

\def\tlz{{\tau^{[\lz]}_x}}

\def\riz{{R_{\Delta_\lz}}}

\def\slz{{\sharp_\lz}}
\def\qlz{{Q^{[\lz]}_t}}
\def\plz{{P^{[\lz]}_t}}

\def\ltz{{L^2(\rr_+,\, dm_\lz)}}
\def\lrz{{L^r(\rr_+,\, dm_\lz)}}

\def\loz{{L^1(\rr_+,\, dm_\lz)}}
\def\lpz{{L^p(\rr_+,\, dm_\lz)}}

\def\hpz{{H^p(\rr_+,\, dm_\lz)}}

\def\dint{\displaystyle\int}

\def\dlimsup{\displaystyle\limsup}
\def\dfrac{\displaystyle\frac}
\def\dsup{\displaystyle\sup}
\def\dlim{\displaystyle\lim}

\def\r{\right}
\def\lf{\left}

\def\noz{\nonumber}

\def\rlz{\R_\lz}

\def\XXint#1#2#3{{\setbox0=\hbox{$#1{#2#3}{\int}$}
     \vcenter{\hbox{$#2#3$}}\kern-.5\wd0}}

\allowdisplaybreaks
\begin{document}

\title[Product Hardy spaces in the Bessel setting]{Characterizations of product Hardy spaces \\ in Bessel setting}

\author{Xuan Thinh Duong}
\address{Xuan Thinh Duong, Department of Mathematics\\
         Macquarie University\\
         NSW 2019\\
         Australia
         }
\email{xuan.duong@mq.edu.au}

\author{Ji Li}
\address{Ji Li, Department of Mathematics\\
         Macquarie University\\
         NSW 2019\\
         Australia
         }
\email{ji.li@mq.edu.au}

%
%

\author{Brett D. Wick}
\address{Brett D. Wick, Department of Mathematics\\
         Washington University -- St. Louis\\
         St. Louis, MO 63130-4899 USA
         }
\email{wick@math.wustl.edu}

\author{Dongyong Yang}
\address{Dongyong Yang, School of Mathematical Sciences\\
 Xiamen University\\
  Xiamen 361005,  China
  }
\email{dyyang@xmu.edu.cn }



\subjclass[2010]{42B35, 42B25, 42B30, 30L99}

\date{\today}


\keywords{Bessel operator, maximal function, Littlewood--Paley theory, Riesz transform, Cauchy--Riemann type equations, product Hardy space, product BMO space}

\begin{abstract}
In this paper, we work in the setting of Bessel operators and Bessel Laplace equations studied by Weinstein, Huber, and the harmonic function theory in this setting introduced by Muckenhoupt--Stein, especially the generalised Cauchy--Riemann equations and the conjugate harmonic functions.  We provide the equivalent characterizations of product Hardy spaces associated with Bessel operators in terms of the Bessel Riesz transforms, non-tangential and radial maximal functions defined via Poisson and heat semigroups, based on the atomic decomposition, the extension of Merryfield's result which connects the product non-tangential maximal function and area function, and on the grand maximal function technique which connects the product non-tangential and radial maximal function.   We then obtain directly the decomposition of the product BMO space associated with Bessel operators. These results are a first extension for product Hardy and BMO associated to a differential operator other than the Laplacian and are a major step beyond the Chang--Fefferman setting.
\end{abstract}

\maketitle




\section{Introduction and statement of main results}
\label{sec:introduction}
\setcounter{equation}{0}


\subsection{Background and main results}

Multiparameter harmonic analysis was introduced in the '70s and studied extensively in the '80s, led by S.-Y. A. Chang, R. Fefferman, R. Gundy, J. Journ\'e, J. Pipher, E. Stein and others (see for example \cite{Cha,GS,CF1,Fef,CF2,FSt,CF3,Jo,KM,F2,F1,J,P,F3,F4}).  The theory of multiparameter harmonic analysis is largely influenced by the corresponding theory of one-parameter (classical) harmonic analysis, but is strongly motivated by two different geometric phenomena.  First, one naturally encounters families of rectangles and operators that are invariant under different scalings than the standard one (e.g, the operator is invariant under a scaling in each variable separately and not just a uniform scaling of all the variables).  Second, the boundary behavior of analytic functions in several complex variables necessitated an understanding of approach regions that behaved differently in each variable separately.  Both these naturally lead to the theory of harmonic analysis allowing for a decomposition of functions admitting different, independent behavior in each variable separately.

As in classical harmonic analysis a key ingredient in the theory is the development of the Hardy and BMO spaces, their duality and the connections to atomic decompositions.  In the multiparameter setting the geometry alluded to above leads to a more complicated description of product BMO.  As demonstrated by Carleson the natural BMO condition on rectangles is not sufficient to characterize the dual of the Hardy space.  This necessitates a BMO theory based on arbitrary open sets and leads to numerous geometric challenges in the theory.  An important result in the area is Journ\'e's covering lemma which provides a tool by which the general open sets can be replaced by certain families of rectangles with controlled geometry.  After this important ground work was established, the development of multiparameter harmonic analysis followed the lines of obtaining $T1$ theorems, characterizations of Hardy spaces via non-tangential and radial maximal functions, and Hilbert (Riesz) transforms.  More recent developments in the area include the dyadic structures, characterization of product BMO via commutators (see for example \cite{fs,fl,lppw,lt,lppw2,T,DO,LPW,ops,KLPW}), and the developments of $T1$ and $Tb$ theorems (see for example \cite{HyM,Ou}).  As is well known, the product Hardy space $H^1 $ has a variety of equivalent norms, in terms of square functions, maximal functions and Hilbert transforms (Riesz transforms in higher dimension), see for example \cite[p.\,19]{l}. See also the product Hardy spaces and boundedness of product singular integrals in different versions studied in \cite{hy,cyz,blyz,lby,lbyz}.



Spaces of homogeneous type were first introduced by R. Coifman and G. Weiss \cite{CW} in the 70's in order to extend the theory of Calder\'on--Zygmund operators to a more general setting. There are no translations or dilations, and no analogue of the Fourier transform or convolution operation on such spaces. Using Coifman's idea on the decomposition of the identity operator, G. David, J. Journ\'e, and S. Semmes \cite{DJS} developed Littlewood-Paley analysis on spaces of homogeneous type and used it to give a proof of the $T1$ theorem on this general setting. Recently, based on this Littlewood-Paley analysis, Han, the second author, and Lu \cite{HLL2} developed the product Hardy spaces $H^p(X_1\times X_2)$ for $p\leq1$ and close to 1 on product spaces of homogeneous type $X_1\times X_2$ via Littlewood--Paley square function, and proved the duality of $H^p$ with the Carleson measure type spaces $CMO^p$, see also the related results in \cite{HLL} and \cite{HLW}. Later, the boundedness of singular integrals,
the product $T1 $ theorem, and atomic decomposition of $H^p$ were also studied in \cite{HLLW,LW,HLLin,HLPW}.



The theory of the classical Hardy space is intimately connected to the Laplacian; changing the differential operator introduces new challenges and directions to explore.
In the past 10 years, a theory of Hardy spaces associated to operators was introduced and developed by P. Auscher, the first author, S. Hofmann, A. McIntosh, L. Yan and many others (we refer to \cite{DY1,DY2,HM,HLMMY} and the references therein). In \cite{DSTY} they first introduced the product Hardy space on the Euclidean setting associated with operators via area functions, and the product BMO space via Carleson measures and proved the duality. We also refer to \cite{DLY,STY} for the product Hardy spaces on the Euclidean setting associated with operators for the atomic decomposition.
%
Recently, P. Chen, L. Ward,  L. Yan and the first and second authors \cite{CDLWY} developed the product Hardy spaces $H^1_{L_1,\,L_2}(X_1\times
X_2)$ on $X_1\times X_2$ (the product spaces of homogeneous
type) associated with operators via Littlewood--Paley area functions and atomic decompositions, and studied the boundedness of product singular integrals with non-smooth kernels,  the Calder\'on--Zygmund decomposition and interpolations of  $H^p$,  as well as the boundedness of Marcinkiewicz type multipliers. Here $L_1$ and $L_2$ are two non-negative
self-adjoint operators acting on $L^2(X_1)$ and $L^2(X_2)$,
respectively, and satisfying Davies--Gaffney
estimates. 
However, the weak conditions on $L_1$ and $L_2$ seem not strong enough for obtaining the characterizations of product space $H^1_{L_1,\,L_2}(X_1\times X_2)$ via maximal function or  via the ``Riesz transforms'' and  the decomposition of product BMO space in this setting is not known either.

In 1965, B. Muckenhoupt and E. Stein in \cite{ms} introduced the harmonic function theory associated with Bessel operator $\tbz$, defined by
setting for suitable functions $f$,
\begin{equation*}
\tbz f(x):=\frac{d^2}{dx^2}f(x)+\frac{2\lz}{x}\frac{d}{dx}f(x),\quad \lz>0,\quad x\in \R_+:=(0,\fz).
\end{equation*}
The related elliptic partial differential equation is the following ``singular Laplace equation''
\begin{equation}\label{bessel laplace equation}
\triangle_{t,\,x} (u) :=\pa_{t}^2u + \pa_{x}^2u+\frac{2\lz}{x}\pa_{x}u=0
\end{equation}
studied by A. Weinstein \cite{w}, and A. Huber \cite{Hu} in higher dimension, where they considered the generalised axially symmetric potentials, and obtained the properties of the solutions of this equation, such as the extension, the uniqueness theorem, and the boundary value problem for certain domains.

If $u$ is a solution of \eqref{bessel laplace equation} then $u$ is said to be $\lambda$-harmonic. The function $u$ and the conjugate of $u$ (denoted by $v$) satisfy the following Cauchy--Riemann type equations
\begin{align}\label{CR}
\pa_{x}u=-\pa_{t}v\ \ {\rm and\ \ }
\pa_{t}u =\pa_{x}v + {2\lambda\over x} v\ \ {\rm in\ \ }\mathbb{R}_+\times \R_+.
\end{align}

In \cite{ms} they developed a theory in the setting of
$\tbz$ which parallels the classical one associated to the standard Laplacian, where results on $\lpz$-boundedness of conjugate
functions and fractional integrals associated with $\tbz$ were
obtained for $p\in[1, \fz)$ and $\dmz(x):= x^{2\lz}\,dx$.

We also point out that Haimo \cite{h} studied the Hankel convolution transforms $\vz \sharp_\lz f$ associated with the Hankel transform in the Bessel setting systematically,
which provides a parallel theory to the classical convolution and Fourier transforms. It is well-known that
the Poisson integral of $f$ studied in \cite{ms} is the Hankel convolution of Poisson kernel with $f$, see \cite{bdt}.

Since then, many problems based on the Bessel context were studied, such as the boundedness of Bessel Riesz transform,
Littlewood--Paley functions, Hardy and BMO spaces associated with Bessel operators, $A_p$ weights associated with Bessel operators
(see, for example, \cite{k78,ak,bfbmt,v08,bfs,bhnv,bcfr,yy,dlwy,DLMWY} and the references therein).


\vskip.15cm

The aim of this paper is to focus on this specific Bessel setting, to study the equivalent characterizations of product Hardy spaces and the decomposition of product BMO spaces associated with Bessel operator $\tbz$.

We note that the measure $\dmz$ related to $\tbz$ is a doubling measure, and hence the standard product Hardy spaces via Littlewood--Paley area functions and via atoms fall into the line of \cite{HLL2,HLPW}, see Section \ref{s2}.  Also, the kernels of the Poisson and heat semigroups of $\tbz$ satisfy the size, smoothness and conservation property, and hence the product Hardy spaces associated with $\tbz$ via Littlewood--Paley area functions and via atoms fall into the line of \cite{CDLWY}. The first part of this paper is  to prove that these two versions of Hardy spaces coincide in this Bessel setting, denoted by $H^p_{\Delta_\lambda}$, $p\in ((2\lz+1)/(2\lz+2), 1]$. We also provide the equivalent characterization via Littlewood--Paley $g$-functions. Based on this, we obtain that the dual of $H^p_{\Delta_\lambda}$ is the standard Carleson measure spaces studied in \cite{HLL2}, denoted by ${\rm CMO}^p_{\Delta_\lambda}$ (by ${\rm BMO}_{\Delta_\lambda}$ when $p=1$).

 Then the second part, which is the main contribution of this paper, is to provide the equivalent characterizations of  $H^p_{\Delta_\lambda}$ in terms of non-tangential and radial maximal function defined via the  Poisson and heat semigroups, as well as the characterization via Bessel Riesz transforms. To obtain this, we build up a variant of the technical lemma of K. Merryfield \cite{KM}, which connects the product non-tangential maximal function and the area function, and we make good use of the generalised Cauchy--Riemann type equations \eqref{CR} and  establish the grand maximal function which connects the non-tangential and radial maximal functions. Then, as a direct consequence, we obtain the decomposition of ${\rm BMO}_{\Delta_\lambda}$ via the Bessel Riesz transforms. We note that these results in the second part are  first extensions for product Hardy and BMO spaces beyond the Chang--Fefferman setting on Euclidean spaces.



\subsection{Statement of main results}
Throughout the paper, for every interval $I\subset \R_+$, we denote it by $I:=I(x,t):= (x-t,x+t)\cap \R_+$. The measure of $I$ is defined as
$m_\lz(I(x,t)):=\int_{I(x,\,t)} x^{2\lz} dx$. In the product setting $\mathbb{R}_+\times \mathbb{R}_+$, we define $d\mu_\lz(x_1, x_2):=dm_\lz(x_1)\times dm_\lz(x_2)$ and
$ \rlz:= (\mathbb{R}_+\times \mathbb{R}_+,d\mu_\lz(x_1, x_2)).$ We work with the domain $( \mathbb{R}_+\times \mathbb{R}_+)
\times ( \mathbb{R}_+\times \mathbb{R}_+)$ and its distinguished
boundary ${ \mathbb{R}_+\times  \mathbb{R}_+}$. For $x := (x_1,x_2)\in { \mathbb{R}_+\times
 \mathbb{R}_+}$, denote by $\Gamma(x)$ the product cone $\Gamma(x) :=
\Gamma_1(x_1)\times\Gamma_2(x_2)$, where $\Gamma_i(x_i) :=
\{(y_i,t_i)\in \mathbb{R}_+\times \mathbb{R}_+: |x_i-y_i| < t_i\}$
for $i := 1$, 2.

%
%
%
%
%
%
We now provide several definitions of
$H^p_{\Delta_\lambda}$, $p\in ((2\lz+1)/(2\lz+2), 1]$.  These spaces all end up being the same, which is one of the main results in this paper.   This requires some additional notation, but the careful reader will notice that the spaces are distinguished notationally by a subscript to remind how they are defined.




Following \cite{CDLWY}, we define the product Hardy spaces associated with the Bessel operator $\tbz$ using the Littlewood--Paley area functions and square functions via the semigroups $\{T_t\}_{t>0}$, where $\{T_t\}_{t>0}$  can be the Poisson semigroup $\{e^{-t\sqrt{\tbz}}\}_{t>0}$ or the heat semigroup $\{e^{-t\tbz}\}_{t>0}$.

Given a function $f$ on $L^2(\rlz)$, the Littlewood--Paley {area
function}~$Sf(x)$, $x:=(x_1,x_2)\in \R_+\times\R_+$, associated with the operator $\Delta_\lambda$ is
defined as
\begin{align}\label{esf}
     Sf(x)
    := \bigg(\iint_{\Gamma(x) }\Big|
        t_1\pa_{t_1}T_{t_1}\,  t_2\pa_{t_2}T_{t_2}  f(y_1,y_2)\Big|^2\
        {d\mu_\lambda(y_1,y_2) dt_1  dt_2 \over t_1m_\lz(I(x_1,t_1)) t_2 m_\lz(I(x_2,t_2))}\bigg)^{1\over2}.
   \end{align}
The square
function~$g(f)(x)$, $x:=(x_1,x_2)\in \R_+\times\R_+$, associated with the operator $\Delta_\lambda$ is
defined as
\begin{align}\label{egf}
     g(f)(x)
    := \bigg(\int_{0}^\infty\int_{0}^\infty\Big|
        t_1\pa_{t_1}T_{t_1}\,  t_2\pa_{t_2}T_{t_2}  f(x_1,x_2)\Big|^2\
        { dt_1  dt_2 \over t_1 t_2 }\bigg)^{1\over2}.
   \end{align}





We now define the product Hardy space $H^p_{\Delta_\lz}$ by using \eqref{esf} via Poisson semigroup as follows.
\begin{defn}\label{def of Hardy space via S function}
For $p\in ((2\lz+1)/(2\lz+2), 1]$, the
    {Hardy space $H^p_{\Delta_\lambda}( \rlz )$} is defined as the completion of
    \[
        \{f\in L^2(\rlz) :
        \|Sf\|_{L^p(\rlz)} < \infty\}
    \]
    with respect to the norm (quasi-norm)
    $
        \|f\|_{H^{p}_{\Delta_\lambda}(\rlz ) }
        := \|Sf \|_{L^p( \rlz)},
    $
    where $Sf$ is defined by \eqref{esf} with $T_t:=e^{-t\sqrt{\tbz}}$.

\end{defn}

Our first main result is to show that $H^{p}_{\Delta_\lambda}(\rlz) $ coincide with $H^p(\rlz)$ as in \cite{HLL2}; see Definition \ref{def-of-Hardy-space on M new metric} below.

\begin{thm}\label{thm H1 and classical H1}
 Let $p\in ((2\lz+1)/(2\lz+2), 1]$. The space
$H^{p}_{\Delta_\lambda}(\rlz) $ coincides with the classical product Hardy space $H^{p}(\rlz)$ and they have equivalent norms (or quasi-norms).
\end{thm}
As a direct consequence, we have that
the dual of $H^{p}_{\Delta_\lambda}(\rlz)$, denoted by ${\rm CMO}^p_{\Delta_\lambda}(\rlz)$, is the classical product Carleson measure space ${\rm CMO}^p(\rlz)$, which is introduced in \cite{HLL2} (see the precise definition in Section 2). Especially, for $p=1$, we denote the dual of $H^{1}_{\Delta_\lambda}(\rlz)$ by ${\rm BMO}_{\Delta_\lambda}(\rlz)$.

\begin{remark}\label{remark}
We note that we can also define the product Hardy spaces $H^p_{\Delta_\lz}( \rlz )$ as in Definition \ref{def of Hardy space via S function} using the area function $Sf$ via $T_t:=e^{-t\tbz}$, as well as using the square function
$g(f)$ as in \eqref{egf} via $T_t:=e^{-t\sqrt{\tbz}}$ or $T_t:=e^{-t\tbz}$, denoted by $H^p_{\Delta_\lz,\,1}( \rlz )$, $H^p_{\Delta_\lz,\,2}( \rlz )$ and $H^p_{\Delta_\lz,\,3}( \rlz )$, respectively. These three versions of product Hardy spaces coincide with $H^{p}(\rlz)$ and they have equivalent norms (or quasi-norms). See Proposition~\ref{prop eee} below.
%
\end{remark}

We now define another version of the Littlewood--Paley area function.
Let
$$\nabla_{t_1,\,y_1}:=(\partial_{t_1}, \partial_{y_1}),\,\, \nabla_{t_2,\,y_2}:=(\partial_{t_2}, \partial_{y_2}).$$
%
Then  the Littlewood--Paley area function $S_uf(x)$ for $f \in L^2(\rlz)$,
$x:=(x_1,x_2)\in \R_+\times\R_+$  is defined as
\begin{align}\label{esfu}
     S_uf(x)
    := \bigg(\iint_{\Gamma(x) }\big| \nabla_{t_1,\,y_1} e^{-t_1\sqrt{\tbz}}\nabla_{t_2,\,y_2}e^{-t_2\sqrt{\tbz}}(f)(y_1,y_2) \big|^2 \
        \frac{t_1t_2\ d\mu_\lambda(y_1,y_2) dt_1dt_2}{m_\lz(I(x_1,t_1)) m_\lz(I(x_2,t_2))}\bigg)^{1\over2}.
\end{align}
Then naturally we have the following definition of the product Hardy space via the Littlewood--Paley area function $S_uf$.
\begin{defn} \label{def of Hardy space via Su function}
%
For $p\in ((2\lz+1)/(2\lz+2), 1]$, the
    {Hardy space $\hsu $ } is defined as the completion of
    \[
        \{f\in L^2(\rlz ) :
        \|S_uf\|_{L^p(\rlz)} < \infty\}
    \]
    with respect to the norm (quasi-norm)
    $
        \|f\|_{H^{p}_{S_u}( \rlz ) }
        := \|S_uf \|_{L^p(\rlz )}.
    $
\end{defn}

Next we define the product non-tangential and radial maximal functions via heat semigroup and Poisson semigroup  associated to $\Delta_\lz$, respectively.
For all $\az\in (0, \fz)$, $p\in[1,\fz)$,
$f\in\lpzd$ and $x_1,x_2\in \R_+$, let
\begin{align*}
&\cn^\az_{h}f(x_1, x_2):=\!\!\sup_{\gfz{|y_1-x_1|<\az t_1}{|y_2-x_2|<\az t_2}}\!\lf|e^{-t_1\tbz}e^{-t_2\tbz} f(y_1, y_2)\r|,\\
&\cn^\az_{P}f(x_1, x_2):=\!\!\sup_{\gfz{|y_1-x_1|<\az t_1}{|y_2-x_2|<\az t_2}}\!\lf|e^{-t_1\sqrt{\tbz}}e^{-t_2\sqrt{\tbz}} f(y_1, y_2)\r|
\end{align*}
be the product non-tangential maximal functions with aperture $\alpha$ via the heat semigroup and Poisson semigroup  associated to $\Delta_\lz$, respectively. Denote $\cn^1_{h}f$ by $\cn_{h}f$ and $\cn^1_{P}f$ by $\cn_{P}f$.  Moreover let
\begin{align*}
&\crz_{h}f(x_1, x_2):=\sup_{t_1>0,\,t_2>0}\lf|e^{-t_1\tbz}e^{-t_2\tbz} f(x_1, x_2)\r|,\\
&\crz_{P}f(x_1, x_2):=\sup_{t_1>0,\,t_2>0}\lf|e^{-t_1\sqrt{\tbz}}e^{-t_2\sqrt{\tbz}} f(x_1,x_2)\r|
\end{align*}
be the product radial maximal functions via the heat semigroup and Poisson semigroup  associated to $\Delta_\lz$, respectively.
%
%

\begin{defn}\label{defn:maximal funct}
%
The
    {Hardy space $\hpmz$}, $p\in ((2\lz+1)/(2\lz+2), 1]$, associated to the maximal function ${ \cm}f$ is defined as the completion of the set
    \[
        \{f\in L^2(\rlz ) :
        \|{ \cm}f\|_{L^p(\rlz)} < \infty\}
    \]
    with the norm (quasi-norm)
    $
        \|f\|_{\hpmz }
        := \|{ \cm}f \|_{L^p(\rlz )}.
    $
Here ${ \cm}f$ is one of the following maximal functions: 
$\cn_{h}f$, $\cn_{P}f$, $\crz_{h}f$ and $\crz_{P}f$.
\end{defn}

Based on our first main result Theorem \ref{thm H1 and classical H1},  the second main result of this paper is as follows.

\begin{thm}\label{thm: char of Hardy spacs by max func}
 Let $p\in ((2\lz+1)/(2\lz+2), 1]$. The product Hardy spaces
$\hap$, $\hsu $, $\hnz,$
$\hrz$, $\hrp$
 and $\hnp$
coincide and have equivalent norms (or quasi-norms).
\end{thm}

Next we consider the definition of product Hardy space via the Bessel Riesz transforms $\rizo(f)$ and $\rizt(f)$ on the first and second variable, respectively. For the definition of Bessel Riesz transforms, we refer to \eqref{riz} in Section 2.2.
%

\begin{defn} \label{def of Hardy space via riesz}
 The {product Hardy space $\horiz$ } is defined as the completion of
\begin{eqnarray}\label{Hardy space riesz}
&&\lf\{f\in\lozd\cap L^2(\rlz): \,\, \rizo f,\ \rizt f,\ \rizo\rizt f\in\lozd\r\}\noz
\end{eqnarray}
endowed with the norm
\begin{eqnarray*}
\|f\|_{\horiz}:=\|f\|_\lozd+\|\rizo f\|_\lozd +\|\rizt f\|_\lozd +\|\rizo\rizt f\|_\lozd.
\end{eqnarray*}
\end{defn}

Then based on our result in Theorem \ref{thm: char of Hardy spacs by max func}, the third main result of the paper is the following characterization of $H^{1}_{\Delta_\lambda}( \rlz ) $.

\begin{thm}\label{thm H1 riesz characterization}
The  product Hardy spaces
$H^{1}_{\Delta_\lambda}(\rlz)$ and $H^1_{Riesz}(\rlz)$
 are equivalent.
\end{thm}

Moreover, we also characterize $H^{p}_{\Delta_\lambda}( \rlz ) $ for $ p\in((2\lz+1)/(2\lz+2), 1)$ via Bessel Riesz transform in a slightly different form.
A distribution $f\in \big(\GGp(\beta_{1},\beta_{2};\gamma_{1},\gamma_{2})\big)^{'}$ is said to be {restricted at
infinity}, if for any
 $r>0$ large enough, $e^{-t_1\sqrt{\tbz}}e^{-t_2\sqrt{\tbz}} f\in \lrzd$ (for the notation and details of this distribution space, we refer to
 Definition \ref{def-of-test-func-on-M times M} below).
 By Theorem \ref{thm: char of Hardy spacs by max func} and an
argument as in \cite[pp.\,100-101]{S}, we see that for any
$f\in \hpz$ with $p\in((2\lz+1)/(2\lz+2), 1]$ , $e^{-t_1\sqrt{\tbz}}e^{-t_2\sqrt{\tbz}} f\in \lrzd$ for all $r\in[p, \fz]$.


\begin{thm}\label{thm Hp Riesz characterization}
Let $p\in ((2\lz+1)/(2\lz+2), 1)$ and $f\in \big(\GGp(\beta_{1},\beta_{2};\gamma_{1},\gamma_{2})\big)^{'}$  be restricted at
infinity.  Then
$f\in \hap$
 if and only if
there exists a positive constant $C$ such that for all $t_1,\,t_2\in(0, \fz)$, 
\begin{eqnarray}\label{eqn:restr infty lp norm; Riesz trans char Hp}
&&\hspace{.9cm}\big\|e^{-t_1\sqrt{\tbz}}e^{-t_2\sqrt{\tbz}}(f)\big\|_\lpzd+\big\|\rizo\big(e^{-t_1\sqrt{\tbz}}e^{-t_2\sqrt{\tbz}}(f)\big)\big\|_\lpzd
\\
&&\hspace{.9cm} +\big\|\rizt\big(e^{-t_1\sqrt{\tbz}}e^{-t_2\sqrt{\tbz}}(f)\big)\big\|_\lpzd+\big\|\rizo\rizt\big( e^{-t_1\sqrt{\tbz}}e^{-t_2\sqrt{\tbz}}(f)  \big)\big\|_\lpzd\le C.\noz
\end{eqnarray}

\end{thm}

Based on the characterization of product Hardy space $H^{1}_{\Delta_\lambda}( \rlz ) $ via Bessel Riesz transforms and the duality of $H^{1}_{\Delta_\lambda}( \rlz ) $ with ${\rm BMO}_{\Delta_\lambda}( \rlz ) $, we directly have the fourth main result:  the decomposition of ${\rm BMO}_{\Delta_\lambda}( \rlz ) $, whose proof  is similar to the classical setting.

\begin{thm}\label{thm BMO decomposition}
The following two statements are equivalent.

${\rm (i)}$ $\varphi \in {\rm BMO}_{\Delta_\lambda}( \rlz ) $;

${\rm(ii)}$ There exist $g_i\in L^\infty( \rlz )$, $i=1,2,3,4$, such that
$$ \varphi= g_1 +  R_{\Delta_\lambda,\,1}(g_2) + R_{\Delta_\lambda,\,2}(g_3) + R_{\Delta_\lambda,\,1}R_{\Delta_\lambda,\,2}(g_4).$$
\end{thm}

\subsection{Structure and main methods of this paper}

In Section \ref{s2}, we first recall the known facts on product spaces of homogeneous type (\cite{HLL2}) and then apply these to our setting $\rlz$, including the Littlewood--Paley theory, Hardy and BMO spaces and atomic decompositions. We then  provide the $L^p$-boundedness ($1<p<\infty$) of the product Littlewood--Paley area functions $Sf$ and $S_uf$ as defined in \eqref{esf} and \eqref{esfu}, respectively. In fact, we will prove this result for a more general Littlewood--Paley  area functions and $g$-functions with the kernels of the operators inside satisfying certain size, smoothness and cancellation conditions which covers both the Bessel Poisson kernel and Bessel heat kernel. The main approach we use here is Calder\'on's reproducing formula, almost orthogonality estimates and the Plancherel--P\'olya type inequalities in the product setting.

In Section \ref{sec:first theorem}, we prove Theorem \ref{thm H1 and classical H1}, our first main result. We note that the standard product Hardy spaces $H^p(\rlz)$ (Definition \ref{def-of-Hardy-space on M new metric}) is a subset of a larger distribution space while  our $H^p_{\Delta_\lz}(\rlz)$ is a completion of the initial subspace in $L^2(\rlz)$. Hence, to show that $H^p(\rlz)$ and $H^p_{\Delta_\lz}(\rlz)$ coincide, the main approach used here is via atomic decompositions. Following an idea in \cite{CDLWY},
we show that by choosing some particular function for Calder\'on's reproducing formula (Proposition \ref{p-kernel cancel-compact supp}), we obtain the atomic decomposition for $H^{p}_{\Delta_\lambda}( \rlz )$, which leads to $H^p(\rlz)\cap L^2(\rlz)=H^p_{\Delta_\lz}(\rlz)\cap L^2(\rlz)$ with equivalent norms (quasi-norms).

In Section \ref{sec:second thm}, we present the proof of Theorem \ref{thm: char of Hardy spacs by max func} by
showing the following inequalities
\begin{eqnarray}\label{loop}
\|f\|_{H^p_{\Delta_\lz}(\rlz)}
&&\leq
\|f\|_{\hsu } \ls
\|f\|_{H^p_{\cn_{P}}(\rlz)}\ls
\|f\|_{H^p_{R_{P}}(\rlz)}\\
&&\ls \|f\|_\hrz\ls \|f\|_\hnz\ls \|f\|_\hap\noz,
\end{eqnarray}
i.e., all the norms above are equivalent.

Here the first inequality follows directly by definition. The fourth inequality follows from the well-known subordination formula which connects the Poisson kernel to the heat kernel. The fifth inequality follows directly from definition. The last inequality follows from the result that $\hap$ coincides with the classical $H^p(\rlz)$ (Theorem \ref{thm H1 and classical H1}) and the fact that $H^p(\rlz)$ has atomic decomposition.
The main difficulties here are in the proofs of the second and third inequalities.

To prove the second inequality, we first point out that, to the best of our knowledge, the only one way up to now, to pass from the Littlewood--Paley area function to the non-tangential maximal function in the classical product setting is due to K. Merryfield \cite{KM}.
The main technique in \cite{KM} relies on the construction of the function $\psi$ in $C_c^\infty(\R)$ according to any given $\phi\in C_c^\infty(\R)$ with certain conditions,  satisfying that for any $f\in L^2(\rr)$,
$$  \pa_t (f\ast \phi_t(x)) = -\pa_x (f\ast \psi_t(x)) $$
which is one of the Cauchy--Riemann equations in the classical setting. Here $\phi_t(x):=t^{-1}\phi({x\over t})$ and similar for $\psi_t$.


Based on the idea above, suppose $\phi\in C^\fz_c(\R_+)$ such that $\phi\ge 0$, $\supp(\phi)\subset (0, 1)$, and $\inzf \phi(x)\,\dmz(x)=1$, we construct a function $\psi(t, x, y)$ defined on $\R_+\times \R_+ \times \R_+$ by solving the following equation (which is one of the Cauchy--Riemann equations adapted to the Bessel setting):
$$\prz_t (\phi_t\slz f)(x)= \prz_x[\psi(f)(t,x)]+\frac{2\lz}x\psi(f)(t,x),$$
where $\phi_t$ is the dilation of $\phi$ in the Bessel setting and $\psi(f)(t,x):=\int_{\R_+} \psi(t, x, y) f(y)\dmz(y)$.
Moreover, we show that $\psi(t, x, y)$ satisfies the required
size, smoothness and cancellation conditions, and especially the support condition: supp $\psi(t, x, y)\subset \{t,x,y\in\R_+:\, |x-y|<t\}$. Note that  $\psi(f)$ here is no longer a Hankel convolution (for notation and details, we refer to Lemma \ref{lem: constru of func psi}).

To prove the third inequality, we borrow an idea from \cite{YZ} in the one-parameter setting (see also \cite{gly1,gly2}), to establish a product grand maximal function, which controls the non-tangential maximal function. Then, by using the reproducing formula and almost orthogonality estimates, we obtain that
the $L^p$ norm of the grand maximal function is bounded by that of the radial maximal function for ${2\lz+1\over 2\lz+2}<p\leq1$.

In Section \ref{sec:third theorem}, we prove our third main result, the Bessel Riesz transform characterizations of $H^p_{\Delta_\lz}(\rlz)$  (Theorems \ref{thm H1 riesz characterization} and  \ref{thm Hp Riesz characterization}). The main ideas in the one-parameter setting are from Fefferman--Stein \cite{FSt} (see also \cite[Chapter III, Section 4.2]{S} and \cite{bdt}), where they obtained the result by studying the boundary value of the corresponding harmonic function and its conjugate. In our product setting, we consider the Bessel bi-harmonic function $u(t_1,t_2,x_1,x_2)$ and its three conjugate bi-harmonic functions $v,w,z$ such that $(u,v)$ and $(w,z)$ satisfy the generalised Cauchy--Riemann equations \eqref{CR} in the first group of variables $(t_1,x_1)$, and that $(u,w)$ and $(v,z)$ in the second group of variables $(t_2,x_2)$. Then, using Lemma 11 in \cite{ms}, we obtain the harmonic majorants of the following four functions $\{u^2+v^2\}^{p\over2}$, $\{w^2+z^2\}^{p\over2}$,  $\{u^2+w^2\}^{p\over2}$,  $\{v^2+z^2\}^{p\over2}$ corresponding to the four groups of Cauchy--Riemann equations above respectively. By iteration, we obtain the harmonic majorant of the bi-harmonic function $\{u^2+v^2+w^2+z^2\}^{p\over2}$.
Then, our main result follows from the properties of the Poisson semigroup $\{e^{-t\sqrt{\Delta_\lz}}\}_{t>0}$ and the standard approach (\cite[Chapter III, Section 4.2]{S}). To the best of our knowledge, the Hilbert transform characterizations not been addressed before for the classical Chang--Fefferman product Hardy space $H^p(\R\times \R)$ when $p<1$. We note that when $\lz=0$, our result and proof go back
to $H^p(\R\times \R)$ with minor modifications, and hence provide the characterizations of $H^p(\R\times \R)$ via Hilbert transforms.

Throughout the whole paper,
we denote by $C$ and $\widetilde{C}$ {  positive constants} which
are independent of the main parameters, but they may vary from line to
line. Constants with subscripts, such as $C_0$ and
$A_1$, do not change in different occurrences.
For every $p\in(1, \fz)$, we denote by $p'$ the conjugate of $p$, i.e., $\frac{1}{p'}+\frac{1}{p}=1$.  If $f\le Cg$, we then write $f\ls g$ or $g\gs f$;
and if $f \ls g\ls f$, we  write $f\sim g.$
For any $k\in \mathbb{R}_+$ and $I:= I(x, r)$ for some $x$, $r\in (0, \fz)$,
$kI:=I(x, kr)$.

\section{Preliminaries}
\label{s2}

In this section, we first apply the known results on product spaces of homogeneous type developed in \cite{HLL2,HLPW} to our setting on $\rlz$. We then recall the properties of the Poisson kernels and conjugate Poisson kernels in the Bessel setting. Finally, we provide a continuous version of Littlewood--Paley theory, which will be used in the subsequent sections.

\subsection{Product Hardy and BMO spaces on spaces of homogeneous type}


To begin with, we point out that in \cite{HLL2,HLPW}, they considered the general setting of product spaces of homogeneous type
$\X:= (X_1,d_1,\mu_1)\times (X_2,d_2,\mu_2)$,
and developed the test function spaces and distribution spaces, Calder\'on's reproducing formula, Littlewood--Paley theory,
product Hardy and BMO spaces and atomic decompositions. For notational simplicity,  we now apply all these results to our setting, i.e.,
$$\X:=\rlz:=(\R_+,|\cdot|,\dmz)\times (\R_+,|\cdot|,\dmz).$$
We first observe that for
any interval $I:=I(x,r)\subset \R_+$, $m_\lz(I)\sim x^{2\lz}r+r^{2\lz+1}$; moreover, from  \cite{DLMWY} we have that
for any $I\subset \rrp$,
\begin{equation*}
\min\big(2, 2^{2\lz}\big)m_\lz(I)\le m_\lz(2I)\le 2^{2\lz+1}m_\lz(I).
\end{equation*}

We now recall the definition of approximation to the identity.
\begin{defn}\label{def-ati}
We say that $\{S_k\}_{k\in\mathbb Z}$ is an approximation to the identity if $\lim_{k\to\infty} S_k=Id$, $\lim_{k\to -\infty} S_k=0$
and moreover, the kernel $S_k(x,y)$ of $S_k$ satisfies the following condition:  for $\beta,\gamma\in (0,1]$,
\begin{itemize}
\item[${\rm (A_i)}$] for any $x,\,y,\,t \in\R_+$,
\begin{equation*}
|S_k(x, y)|\ls \frac1{m_\lz(I(x, 2^{-k}))+m_\lz(I(y,2^{-k}))+m_\lz(I(x, |x-y|))}\bigg(\frac {2^{-k}}{|x-y|+2^{-k}}\bigg)^\gamma;
\end{equation*}
  \item [${\rm (A_{ii})}$] for any $x,y,\wz y,t \in\R_+$ with $|y-\wz y|\le (t+|x-y|)/2$,
\begin{align*}
&|S_k(x, y)-S_k(x,\wz y)|+ |S_k(y,x)-S_k(\wz y,x)|\\
&\ls \frac1{m_\lz(I(x, 2^{-k}))+m_\lz(I(y,2^{-k}))+m_\lz(I(x, |x-y|))}\bigg(\frac {|y-\wz y|}{|x-y|+2^{-k}}\bigg)^{\beta}\bigg(\frac {2^{-k}}{|x-y|+2^{-k}}\bigg)^\gamma;
\end{align*}
  \item [${\rm (A_{iii})}$] for any $x,y,\wz x,\wz y,t \in\R_+$ with $|y-\wz y|\le (t+|x-y|)/2$,
\begin{align*}
&|S_k(x, y)-S_k(x,\wz y)-S_k(\wz x,y)+S_k(\wz x,\wz y)|\\
&\ls \frac1{m_\lz(I(x, 2^{-k}))+m_\lz(I(y,2^{-k}))+m_\lz(I(x, |x-y|))}\\
&\quad\times \bigg(\frac {|x-\wz x|}{|x-y|+2^{-k}}\bigg)^{\beta}\bigg(\frac {|y-\wz y|}{|x-y|+2^{-k}}\bigg)^{\beta}\bigg(\frac {2^{-k}}{|x-y|+2^{-k}}\bigg)^\gamma;
\end{align*}
  \item [${\rm (A_{iv})}$] for any $t,\, x\in\R_+$,
\begin{equation*}
\inzf S_k(x, y)\dmzy=\inzf S_k(y,x)\dmzy=1.
\end{equation*}
\end{itemize}
\end{defn}
One of the constructions of an approximation to the identity is due to
Coifman, see \cite{DJS}. We set $D_k:=S_k-S_{k-1}$, and it is obvious that
$D_k$ satisfies ${\rm (A_i)}$, ${\rm (A_{ii}) }$ and ${\rm (A_{iii})}$ and with
\begin{equation*}
\inzf D_k(x, y)\dmzy=\inzf D_k(y,x)\dmzy=0
\end{equation*}
for any $t,\, x\in\R_+$.

We now
recall the results on Hardy spaces and Carleson measure spaces
and related results developed in
 \cite{HLL2}. We begin with the test function spaces and distribution spaces, and the one-parameter version of which was defined by Han, M\"uller, and Yang \cite{hmy06,hmy08}, and then the product version by Han, Li, and Lu \cite{HLL2}.

%
%
\begin{defn}[\cite{hmy06}]\label{def-of-test-func-space}
Consider the space $(\R_+,|\cdot|,\dmz)$. Let $0<\gamma, \beta\leq 1$ and $r>0.$
A function $f$ defined on $\R_+$ is said to be a test function of type
$(x_0,r,\beta,\gamma)$ centered at $x_0\in \R_+$ if $f$ satisfies the
following conditions:
\begin{itemize}
\item[(i)]
$|f(x)|\leq C \frac{\displaystyle 1}{\displaystyle
V_r(x_0)+V(x,x_0)} \Big(\frac{\displaystyle r}{\displaystyle
r+d(x,x_0)}\Big)^{\gamma}$;
\item[(ii)]
$|f(x)-f(y)|\leq C \Big(\frac{\displaystyle
d(x,y)}{\displaystyle r+d(x,x_0)}\Big)^{\beta} \frac{\displaystyle
1}{\displaystyle V_r(x_0)+V(x,x_0)} \Big(\frac{\displaystyle
r}{\displaystyle r+d(x,x_0)}\Big)^{\gamma}$ for all $x,y\in \R_+$ with
$d(x,y)\le{\frac{1}{2}}(r+d(x,x_0)).$
\end{itemize}
Here $V(x,x_0):=m_\lz(I(x, |x-x_0|))$.
If $f$ is a test function of type $(x_0,r,\beta,\gamma)$, we write
$f\in \cg(x_0,r,\beta,\gamma)$ and the norm of $f\in
\cg(x_0,r,\beta,\gamma)$ is defined by
$\|f\|_{\cg(x_0,\,r,\,\beta,\,\gamma)}:=\inf\{C>0: {\rm (i)\ and\ (ii)\ hold}\}.$
\end{defn}

Now for any fixed $x_0\in \R_+$, we denote
$\cg(\beta,\gamma):=\cg(x_0,1,\beta,\gamma)$ and by $\cg_0(\beta,\gamma)$ the collection of all test functions in $\cg(\beta,\gamma)$ with $\int_{\R_+} f(x) \dmz(x)=0.$ Note that
$\cg(x_1,r,\beta,\gamma)=\cg(\beta,\gamma)$ with equivalent norms for all
$x_1\in \R_+$ and $r>0$ and that
$\cg(\beta,\gamma)$ is a Banach space with respect to the norm in
$\cg(\beta,\gamma)$.

Let $\GG(\beta,\gamma)$ be the
completion of the space $\cg_0(1,1)$ in
the norm of $\cg(\beta,\gamma)$ when $0<\beta,\gamma<1$. If $f\in \GG(\beta,\gamma)$, we then define
$\|f\|_{\GG(\beta,\gamma)}:=\|f\|_{\cg(\beta,\gamma)}$. $(\GG(\beta,\gamma))'$, the distribution space, is defined to be the set of all
linear functionals $L$ from $\GG(\beta,\gamma)$ to $\mathbb{C}$ with
the property that there exists $C\geq0$ such that for all $f\in
\GG(\beta,\gamma)$,
$|L(f)|\leq C\|f\|_{\GG(\beta,\gamma)}.$

Now we return to the product setting and recall the space of test functions and distributions on the product space $\rlz$.
\begin{defn}[\cite{HLL2}]\label{def-of-test-func-on-M times M}
Let $(x_0,y_0)\in \R_+\times\R_+$, $0<\gamma_1,\gamma_2,\beta_1,\beta_2\leq 1$ and $r_1, r_2>0.$ A function $f(x,y)$ defined on $\rlz$ is said to be a test function of type
$(x_0,y_0;r_1,r_2;\beta_1,\beta_2;\gamma_1,\gamma_2)$ if for any fixed $y\in \R_+,$ $f(x,y),$ as a function of
the variable of $x,$ is a test function in $\cg(x_0,r_1,\beta_1,\gamma_1)$ on $\R_+.$  Moreover, the following conditions are satisfied:
\begin{itemize}
\item[(i)]
$\Vert f(\cdot,y)\Vert_{\cg(x_0,\,r_1,\,\beta_1,\,\gamma_1)}\leq
\frac{\displaystyle C }{\displaystyle
V_{r_2}(y_0)+V(y_0,y)}\Big(\frac{\displaystyle r_2}{\displaystyle
r_2+d(y,y_0)}\Big)^{\gamma_2}$;
\item[(ii)]
$\Vert f(\cdot,y)-f(\cdot,\wz y)\Vert_{\cg(x_0,\,r_1,\,\beta_1,\,\gamma_1)}\leq
\frac{\displaystyle C}{\displaystyle
V_{r_2}(y_0)+V(y_0,y)}\Big(\frac{\displaystyle d(y,\wz y)}{\displaystyle
r_2+d(y,y_0)}\Big)^{\beta_2}\Big(\frac{\displaystyle r_2}{\displaystyle
r_2+d(y,y_0)}\Big)^{\gamma_2}$ \ \ \ \ \ \ for all $y,\wz y\in \R_+$
with $d(y,\wz y)\leq (r_2+d(y,y_0))/2$.
\end{itemize}
Similarly, for any fixed $x\in \R_+,$ $f(x,y),$ as a function of the variable of $y,$ is a test function in $\cg(y_0,r_2,\beta_2,\gamma_2)$ on $\R_+$, and both properties (i) and (ii) also hold with $x,y$ interchanged.

If $f$ is a test function of type
$(x_0,y_0;r_1,r_2;\beta_1,\beta_2;\gamma_1,\gamma_2)$, we write
$f\in \cg(x_{0},y_{0};r_{1},r_{2};\beta_{1},\beta_{2};$
$\gamma_{1},\gamma_{2})$ and the norm of $f$ is defined by
$$\|f\|_{\cg(x_{0},\,y_{0};\,r_{1},\,r_{2};\,\beta_{1},\,\beta_{2};\,\gamma_{1},\,\gamma_{2})}:=\inf\{C:\  {\rm (i),\, (ii)\ and\ (iii)\ hold}\}.$$
\end{defn}

Similarly, we denote by
$\cg(\beta_{1},\beta_{2};\gamma_{1},\gamma_{2})$ the class
$\cg(x_{0},y_{0};1,1;\beta_{1},\beta_{2};\gamma_{1},\gamma_{2})$ for
any fixed $(x_{0},y_{0})\in \R_+\times\R_+.$  Say that $f(x,y)\in \cg_0(\beta_{1},\beta_{2};\gamma_{1},\gamma_{2})$ if $$\int_{X_1}f(x,y)d\mu_1(x)=\int_{X_2}f(x,y)d\mu_2(y)=0.$$
Note that
$\cg(x_{0},y_{0};r_{1},r_{2};\beta_{1},\beta_{2};\gamma_{1},\gamma_{2})=
\cg(\beta_{1},\beta_{2};\gamma_{1},\gamma_{2})$ with equivalent
norms for all $(x_{0},y_{0})\in \R_+\times \R_+$ and $r_1,r_2>0$ and that
$\cg(\beta_{1},\beta_{2};\gamma_{1},\gamma_{2})$ is a Banach space
with respect to the norm in
$\cg(\beta_{1},\beta_{2};\gamma_{1},\gamma_{2})$.

Let $\GGp(\beta_1,\beta_2;\gamma_1,\gamma_2)$ be the completion of
the space $\cg_0(1,1,1,1)$ in
$\cg(\beta_1,\beta_2;\gamma_1,\gamma_2)$ with
$0<\beta_i,\gamma_i<1$.  If
$f\in\GGp(\beta_{1},\beta_{2};\gamma_{1},\gamma_{2}) $, we then
define
$\|f\|_{\GGp(\beta_{1},\,\beta_{2};\,\gamma_{1},\,\gamma_{2})}:=\|f\|_{\cg(\beta_{1},\,\beta_{2};\,\gamma_{1},\,\gamma_{2})}$.

We define the distribution space
$\big(\GGp(\beta_{1},\beta_{2};\gamma_{1},\gamma_{2})\big)^{'}$ by
all linear functionals $L$ from the space
$\GGp(\beta_{1},\beta_{2};\gamma_{1},\gamma_{2})$ to $\mathbb{C}$
with the property that there exists $C\geq 0$ such that for all
$f\in \GGp(\beta_{1},\beta_{2};\gamma_{1},\gamma_{2})$,
$|L(f)|\leq C
\|f\|_{\GGp(\beta_{1},\,\beta_{2};\,\gamma_{1},\,\gamma_{2})}.$

Now we recall the Hardy space $H^p(\rlz)$ in \cite{HLL2} defined in terms of
discrete Littlewood--Paley--Stein square function via a system of ``dyadic cubes" in
spaces of homogeneous type. We mention that in our current setting, we take the classical dyadic
intervals as our dyadic system. That is, for each $k\in \Z$,
$\mathscr X^k:=\{\tau:\,\,I^k_\tau:=(\tau 2^k, (\tau+1)2^k]\}_{\tau\in \Z_+}$, where  $\Z_+:=\N\cup\{0\}$.

\begin{defn}[\cite{HLL2}]\label{def-of-squa-func new metric}
    For $i = 1$, $2$, let $\{S_{k_i}\}_{k_i\in\mathbb{Z}}$ be approximations to
    the identity on $\R_+$, and let $D^{(i)}_{k_i} :=
    S_{k_i}-S_{k_i-1}$. For $f\in
    \big(\GGp(\beta_1,\beta_2;\gamma_1,\gamma_2)\big)'$ with
    $0 < \beta_i, \gamma_i < \theta_i$ for $i
    = 1$, 2, the {discrete Littlewood--Paley--Stein square
    function} $ S_d(f)$ of $f$ is defined by
    \[
        S_d(f)(x,y)
        :=\bigg\{  \sum_{k_1}\sum_{I_1 \in\mathscr{X}^{k_1+N_1}} \sum_{k_2}\sum_{I_2 \in\mathscr{X}^{k_2+N_2}} \big|D^{(1)}_{k_1}D^{(2)}_{k_2}(g)(x_{I_1},x_{I_2})\big|^2 \chi_{I_1}(x_1)\chi_{I_2}(x_2)\bigg\}^{1\over2},
    \]
    where $x_{I_i}$ is the center of of the dyadic interval $I_i$ for $i=1,2$, and $N_1$ and $N_2$ are two large fixed positive numbers.
\end{defn}

\begin{defn}[\cite{HLL2}]\label{def-of-Hardy-space on M new metric}
    Suppose $\frac{ 2\lz+1}{ 2\lz+2} <p\leq1$. Suppose $0 < \beta_i, \gamma_i <1$ for $i = 1$, $2$. The {Hardy space}
    $H^p(\rlz)$ is defined by
    \[
        H^p(\rlz)
        :=\big\lbrace f \in
            \big(\GGpp(\beta_1,\beta_2;\gamma_1,\gamma_2)\big)':
            S_d(f)\in L^p(\rlz)\big\rbrace,
    \]
    with the norm (quasi-norm) $\Vert
    f\Vert_{H^p(\rlz)} := \Vert S_d(f)\Vert_{L^p(\rlz)}$.
\end{defn}

\begin{defn}[\cite{HLL2,HLW}]\label{def-of-product-CMO-space new metric}
    Suppose $\frac{ 2\lz+1}{ 2\lz+2} <p\leq1$. Suppose $0 < \beta_i, \gamma_i < 1$ for $i = 1$, $2$.
    The generalized Carleson measure space $\cmo^p(\rlz)$
    is defined by the set of all
    $f\in\big(\GGp(\beta_1,\beta_2;\gamma_1,\gamma_2)\big)'$ such
    that
    \begin{align*}
        &\|f\|_{\cmo^p(\rlz)}:
        =\sup_{\Omega}
            \bigg\{ \frac{\displaystyle1}{\displaystyle \mu_\lz(\Omega)^{{2\over p}-1}}
            \int_{\Omega}\sum_{         k_1}\sum_{\alpha_1\in \mathscr{X}^{k_1}}\sum_{k_2}\sum_{\alpha_2\in \mathscr{X}^{k_2}}  \Big|D^{(1)}_{k_1}D^{(2)}_{k_2}(f)(x,y)\Big|^2 \\
         &\hskip3cm  \chi_{\{k_1,\,\alpha_1,\,k_2,\,\alpha_2:\ I_{\alpha_1}^{k_1}\times I_{\alpha_2}^{k_2}\subset \Omega\}}(k_1,\alpha_1,k_2,\alpha_2) \chi_{I_{\alpha_1}^{k_1}\times I_{\alpha_2}^{k_2}}(x,y) d\mu_\lz(x,y)
            \bigg\}^{1\over 2}<\infty,\nonumber
    \end{align*}
    where $\Omega$ ranges over all open sets in $\R_+\times\R_+$ with
    finite measure and $\chi_A$ is the characteristic function of a given set $A$.
\end{defn}

Next we recall the atomic decomposition for $H^p(\rlz)$ in \cite{HLPW}.
We call $R:=I_{\tau_1}^{k_1}\times
I_{\tau_2}^{k_2}$ a dyadic rectangle in $\R_+\times\R_+$.   Let
$\Omega\subset \R_+\times\R_+$ be an open set of finite measure and
$m_i(\Omega)$ denote the family of dyadic rectangles
$R\subset\Omega$ which are maximal in the $i$th ``direction'', $i=1,2$. Also we denote by $m(\Omega)$ the set of all maximal
dyadic rectangles contained in $\Omega$.


\begin{defn}[\cite{HLPW}]\label{def-of-p q atom}
Suppose $\frac{ 2\lz+1}{ 2\lz+2} <p\leq1$. A function $a(x_1,x_2)$ defined on $ \R_+\times\R_+ $ is called an atom of $H^p( \rlz )$ if $a(x_1,x_2)$ satisfies:
\begin{itemize}
\item[(1)] supp $a\subset\Omega$, where $\Omega$ is an open set of $ \R_+\times\R_+ $ with finite measure;
\item[(2)] $\|a\|_{L^2(\rlz)}\leq \mu_\lz(\Omega)^{1/2-1/p}$;
\item[(3)]  $a$ can be further decomposed into rectangular atoms $a_R$ associated to dyadic rectangle $R:=I_1\times I_2$, satisfying the following

\begin{enumerate}
  \item [(i)] there exist two constants $\bar C_1$ and $\bar C_2$ such that supp $a_R\subset \bar C_1I_1\times \bar C_2I_2$;
  \item [(ii)] $\int_{\R_+}a_R(x_1,x_2)\dmz(x_1)=0$ for a.e. $x_2\in \R_+$ and\\
   \hskip.7cm $\int_{\R_+}a_R(x_1,x_2)\dmz(x_2)=0$ for a.e.
$x_1\in \R_+$;
  \item [(iii)] $a=\sum\limits_{R\in m(\Omega)}a_R$ and $ \Big(\sum\limits_{R\in m(\Omega)}\|a_R\|_{L^2(\rlz)}^2\Big)^{1/2} \leq \mu_\lz(\Omega)^{1/2-1/p}$.

\end{enumerate}
\end{itemize}
\end{defn}


\begin{thm}[\cite{HLPW}]\label{theorem-Hp atom decomp}
Suppose $\frac{ 2\lz+1}{ 2\lz+2} <p\leq1$. Then $f\in L^2( \rlz )\cap H^p( \rlz )$ if and only if $f$ has an atomic decomposition; that is,
\begin{eqnarray*}
    f=\sum_{i=-\infty}^\infty\lambda_ia_i,
\end{eqnarray*}
in the sense of both $H^p(\rlz )$ and
$L^2( \rlz )$, where $a_i$ are  atoms and $\sum_i|\lambda_i|^p<\infty.$
Moreover,
\begin{eqnarray*}
    \|f\|_{H^p( \rlz )}
    \approx \inf  \left\{\sum_{i=-\infty}^{\infty}|\lambda_i |^p \right\}^{\frac{1}{p}},
\end{eqnarray*}
where the infimum is taken over all decompositions as above and
the implicit constants are independent of the $L^2( \rlz)$ and $H^p( \rlz )$ norms of $f.$
\end{thm}

\subsection{Poisson kernel and conjugate Poisson kernel in the Bessel setting $(\R_+,|\cdot|,\dmz)$}
Recall that $\plz(f):= e^{-t\sqrt\tbz}f=P^{[\lz]}_t\sharp_\lz f$
and
$W^{[\lz]}_t(f):= e^{-t\tbz}f= W^{[\lz]}_{\sqrt{2t}}\sharp_\lz f$,
where \begin{equation*}
P^{[\lz]}(x):=\frac{2\lz\bgz(\lz)}{\bgz(\lz+1/2)\sqrt{\pi}}
\frac1{(1+x^2)^{\lz+1}}
\end{equation*}
and $W^{[\lz]}(x):=2^{(1-2\lz)/2}
\exp\lf(-x^2/2\r)/\bgz(\lz+1/2)$ and $f\in\loz$.
For $f$ and $\varphi\in \loz$, their {Hankel convolution} is defined by
setting, for all $x,\,t\in (0,\fz)$,
\begin{equation}\label{Hankel convo}
\Phi_{t,\,\lz}f(x):=\varphi_t\sharp_\lz f(x):=\dint_0^\fz f(y)\tlz \varphi_t(y)\dmz(y),
\end{equation}
where for $t,\,x\in (0, \fz)$, $\varphi_t(y):=t^{-2\lz-1}\varphi(y/t)$ and $\tlz \varphi_t(y)$
denotes the {Hankel translation} of $\varphi_t(y)$, that is,
\begin{equation}\label{Hankel trans}
\tlz \varphi_t(y):=c_\lz
\dint_0^\pi \varphi_t\lf(\sqrt{x^2+y^2-2xy\cos \theta}\r)(\sin\theta)^{2\lz-1}\,d\theta
\end{equation}
with $c_\lz:=\frac{\bgz(\lz+1/2)}{\bgz(\lz)\sqrt{\pi}}$, see \cite[pp.\,200-201]{bdt} or \cite{h}.


Moreover, we recall that  $\{e^{-t\Delta_\lz}\}_{t>0}$ or $\{e^{-t\sqrt{\Delta_\lz}}\}_{t>0}$ have the following properties; see
\cite{bdt,yy,wyz}.


\begin{lem}\label{l-seimgroup prop}
Let $\{T_t\}_{t>0}$ be one of $\{e^{-t\Delta_\lz}\}_{t>0}$ or $\{e^{-t\sqrt{\Delta_\lz}}\}_{t>0}$. Then $\{T_t\}_{t>0}$ is
a symmetric diffusion semigroup satisfying that $T_tT_s=T_sT_t$ for any $t,\,s\in(0, \fz)$,
$T_0=Id$, the identity operator, $\lim_{t\to0}T_tf=f$ in $\ltz$ and
\begin{itemize}
  \item [${\rm (S_i)}$]$\|T_tf\|_{\lpz}\le \|f\|_{\lpz}$ for all $p\in[1, \fz]$ and $t\in(0,\,\fz)$;
  \item[${\rm (S_{ii})}$] $T_t f\ge0$ for all $f\ge0$ and $t\in(0, \fz)$;
  \item [${\rm (S_{iii})}$] $T_t(1)=1$ for all $t\in (0,\,\fz)$.
\end{itemize}
\end{lem}

Next we recall the definitions of the Poisson kernel and conjugate Poisson kernel.
For any $t,\,x,\,y\in(0, \fz)$,
\begin{equation*}
P^{[\lz]}_tf(x):=\inzf P^{[\lz]}_t(x,y)f(y)y^{2\lz}\,dy,
\end{equation*}
where
\begin{equation*}
P^{[\lz]}_t(x,y)=\frac{2\lz t}{\pi}\int_0^\pi\frac{(\sin\theta)^{2\lz-1}}{(x^2+y^2+t^2-2xy\cos\theta)^{\lz+1}}\,d\theta.
\end{equation*}
See \cite{w,bdt}.

If $f\in\lpz$, $p\in[1, \fz)$, the $\Delta_\lz$-conjugate of $f$ is defined by setting,
for any $t,\,x,\,y\in(0, \fz)$,
\begin{equation}\label{conj poi defn}
\qlz(f)(x):=\int_0^\fz\qlz(x, y)f(y)\, dm_\lz(y),
\end{equation}
where
\begin{equation*}
\begin{array}[b]{cl}
\qlz(x, y)&:
=-\dfrac{2\lz}{\pi}\dint_0^\pi\dfrac{(x-y\cos\theta)(\sin\theta)^{2\lz-1}}
{(x^2+y^2+t^2-2xy\cos\theta)^{\lz+1}}\,d\theta;
\end{array}
\end{equation*}
see \cite[p.\,84]{ms}. We point out that
there exists the boundary value function $\lim_{t\to0}\qlz(f)(x)$
for almost every $x\in (0, \fz)$ (see \cite[p.\,84]{ms}),
which is defined to be the Riesz transform $\riz(f)$, .i.e,
\begin{align}\label{riz}
\riz(f)(x):=\lim_{t\to0}\qlz(f)(x) = \int_{\R_+}  -\dfrac{2\lz}{\pi}\dint_0^\pi\dfrac{(x-y\cos\theta)(\sin\theta)^{2\lz-1}}
{(x^2+y^2-2xy\cos\theta)^{\lz+1}}\,d\theta \ f(y) \dmz(y).
\end{align}
Moreover, we note that $u(t,x):=\plz(f)(x)$ satisfies \eqref{bessel laplace equation} and
 that $u(t,x):=\plz(f)(x)$ and $v(t,x):=\qlz(f)(x)$ satisfy the Cauchy--Riemann equations \eqref{CR}.

\begin{prop}[\cite{yy,wyz}]\label{p-aoti}
For any fixed $t$ and $x\in\R_+$,
$ \plz(x, \cdot)$, $ \qlz(x, \cdot)$, $ \pplz(x, \cdot)$ and $ t\pa_y\plz(x, \cdot)$
as  functions of $x$ are in $\GGone(\beta,\gamma)$ for all $\beta,\gamma\in (0,1]$;  symmetrically,  for any
fixed $t$ and $y\in\R_+$, $\pplz(\cdot, y)$  and $ t\pa_y\plz(\cdot, y)$ are in $\GGone(\beta,\gamma)$ for all $\beta,\gamma\in (0,1]$.
\end{prop}

Based on Definition \ref{def-of-test-func-on-M times M} and Proposition \ref{p-aoti}, we further point out that
\begin{prop}\label{pp-aoti}
For any fixed $t_1$, $t_2$, $x_1$ and $x_2\in\R_+$,
$$\plzo(x_1, \cdot)\plzt(x_2,\cdot),\, \plzo(x_1, \cdot)\qlzt(x_2,\cdot),\,\qlzo(x_1, \cdot)\plzt(x_2,\cdot),\,\qlzo(x_1, \cdot)\qlzt(x_2,\cdot)$$
as functions of $(x_1,x_2)$ is in $\GGpone(\beta_1,\beta_2;\gamma_1,\gamma_2)$ for all $\beta_1,\gamma_1,\beta_2,\gamma_2\in (0,1]$.

\end{prop}

\subsection{$L^p$ boundedness of the product Littlewood--Paley area functions and square functions}

In this subsection, we provide the $L^p$ boundedness of a general version of the product Littlewood--Paley area functions and square functions for $1<p<\infty$, which covers  $S(f)$ and $S_u(f)$ defined in \eqref{esf} and \eqref{esfu}, respectively.


To begin with,
for $i=1,2$, we consider the integral operators $\left\{Q^{(i)}_{t_i}\right\}_{t_i>0}$ associated with the kernels $K^{(i)}_{t_i}(x_i,y_i)$. 
Assume that $K^{(i)}_{t_i}(x_i,y_i)$ satisfies the following properties (for the sake of simplicity, when we state these properties we drop the superscript $i$):
\begin{itemize}
\item[${\rm (K_i)}$] for any $x,\,y,\,t \in\R_+$,
\begin{equation*}
|K_t(x, y)|\ls \frac1{m_\lz(I(x, t))+m_\lz(I(y,t))+m_\lz(I(x, |x-y|))}\frac t{|x-y|+t};
\end{equation*}
  \item [${\rm (K_{ii})}$] for any $x,y,\wz y,t \in\R_+$ with $|y-\wz y|\le (t+|x-y|)/2$,
\begin{equation*}
|K_t(x, y)-K_t(x,\wz y)|\ls \frac1{m_\lz(I(x, t))+m_\lz(I(y,t))+m_\lz(I(x, |x-y|))}\frac {t|y-\wz y|}{(|x-y|+t)^2};
\end{equation*}
  \item [${\rm (K_{iii})}$] for any $t,\, x\in\R_+$,
\begin{equation*}
K_t(1)(x):=\inzf K_t(x, y)\dmzy=0.
\end{equation*}
\end{itemize}

%
%

We now provide the $L^p$ (for $1<p<\infty$) boundedness of the product Littlewood--Paley square functions associated with the operators $Q^{(1)}_{t_1}$ and $Q^{(2)}_{t_2}$, which will be needed in Section 4.
 To be precise, we have


\begin{thm}\label{thm product Littlewood--Paley}
Let $p\in(1, \fz)$, $Q^{(1)}_{t_1}$ and $Q^{(2)}_{t_2}$ be the same as above. Then for every $g\in L^p(\rlz)$, and almost all $x_2$, we have that
\begin{eqnarray}\label{Littlewood--Paley}
\bigg\| \left\{  \int_0^\fz \Big|Q^{(1)}_{t_1}(g(\cdot,x_2))(x_1)\Big|^2 {dt_1\over t_1} \right\}^{1/2}\bigg\|_{L^p(\R_+,\,\dmz(x_1))}\ls \|g(x_1,x_2)\|_{L^p(\R_+,\,\dmz(x_1))},
\end{eqnarray}
and similar result holds for $Q^{(2)}_{t_2}$,
and that
\begin{eqnarray}\label{product Littlewood--Paley}
\bigg\| \left\{  \int_0^\fz\int_0^\fz \Big|Q^{(1)}_{t_1}Q^{(2)}_{t_2}(g)(\cdot,\cdot)\Big|^2 {dt_1\over t_1}{dt_2\over t_2} \right\}^{1\over2}\bigg\|_{L^p(\rlz)}\ls \|g\|_{L^p(\rlz)}.
\end{eqnarray}

\end{thm}

\begin{proof}
To begin with, we need to apply the general product discrete Littlewood--Paley theory on spaces of homogeneous type to our Bessel setting, by considering
$(X_i,d_i,\mu_i):=(\R_+,|\cdot|,dm_\lz)$ for $i=1,2$, i.e., $\X:=\R_\lz$.
We note that in this product setting $\X$, we already have the discrete Littlewood--Paley theory (we refer to Theorem 2.14 and Proposition 2.16 in \cite{HLL2}), stated as follows: for $1<p<\infty$,
\begin{align}\label{discrete Littlewood--Paley}
\big\| S_d(g) \big\|_{L^p(\rlz)}\ls \|g\|_{L^p(\rlz)}.
\end{align}

We first prove \eqref{product Littlewood--Paley}. To this end, it suffices to prove that
\begin{eqnarray}\label{product Littlewood--Paley eee}
\bigg\| \left\{  \int_0^\fz\int_0^\fz \Big|Q^{(1)}_{t_1}Q^{(2)}_{t_2}(g)(\cdot,\cdot)\Big|^2 {dt_1\over t_1}{dt_2\over t_2} \right\}^{1\over2}\bigg\|_{L^p(\rlz)}\ls \big\| S_d(g) \big\|_{L^p(\rlz)}.
\end{eqnarray}
To see this, we now recall Calder\'on's reproducing formula (Theorem 2.9 in \cite{HLL2})
\begin{align}\label{reproducing}
g(x_1,x_2)&=   \sum_{k_1}\sum_{I_1 \in\mathscr{X}^{k_1+N_1}} \sum_{k_2}\sum_{I_2 \in\mathscr{X}^{k_2+N_2}}  m_\lambda(I_1)m_\lambda(I_2) \\
&\quad\quad\times\tilde{D}^{(1)}_{k_1}(x_1,x_{I_1})\tilde{D}^{(2)}_{k_2}(x_2,x_{I_2})D^{(1)}_{k_1}D^{(2)}_{k_2}(g)(x_{I_1},x_{I_2}),\noz
\end{align}
where the series converges in the sense of $L^p(\rlz)$, for $i:=1,2$, $\tilde{D}^{(i)}_{k_i}$  satisfies the same size, smoothness and cancellation conditions as $D^{(i)}_{k_i}$ does.

Observe that $Q^{(1)}_{t_1}Q^{(2)}_{t_2}$ is bounded on $\lpzd$.
Applying Calder\'on's reproducing formula to the left-hand side of \eqref{product Littlewood--Paley eee}, we obtain that
\begin{align}\label{eee1}
 \mathbb L&:=\int_0^\fz\int_0^\fz \big|Q^{(1)}_{t_1}Q^{(2)}_{t_2}(g)(x_1,x_2)\big|^2 {dt_1\over t_1}{dt_2\over t_2}\\
 &=\sum_{\wz k_1\in\mathbb Z}\sum_{\wz k_2\in\mathbb Z}  \int_{2^{-\wz k_1-1}}^{2^{-\wz k_1}}\int_{2^{-\wz k_2-1}}^{2^{-\wz k_2}} \bigg| \sum_{k_1}\sum_{I_1 \in\mathscr{X}^{k_1+N_1}} \sum_{k_2}\sum_{I_2 \in\mathscr{X}^{k_2+N_2}}  m_\lambda(I_1)m_\lambda(I_2)\nonumber\\
 &\quad \quad\times Q^{(1)}_{t_1}(\tilde{D}^{(1)}_{k_1}(\cdot,x_{I_1}))(x_1)Q^{(2)}_{t_2}(\tilde{D}^{(2)}_{k_2}(\cdot,x_{I_2}))(x_2) D^{(1)}_{k_1}D^{(2)}_{k_2}(g)(x_{I_1},x_{I_2})\bigg|^2 {dt_1\over t_1}{dt_2\over t_2}.\nonumber
\end{align}
Note that in this case, $t_1\sim 2^{-\wz k_1}$ and $t_2\sim 2^{-\wz k_2}$, hence,  $$Q^{(1)}_{t_1}(\tilde{D}^{(1)}_{k_1}(\cdot,x_{I_1}))(x_1)Q^{(2)}_{t_2}(\tilde{D}^{(2)}_{k_2}(\cdot,x_{I_2}))(x_2) $$
 satisfies the following almost orthogonality estimate (see Lemma 2.11 in \cite{HLL2}):  for $\epsilon \in(0, 1)$,
\begin{align}\label{eee2}
&\big|Q^{(1)}_{t_1}(\tilde{D}^{(1)}_{k_1}(\cdot,x_{I_1}))(x_1) Q^{(2)}_{t_2}(\tilde{D}^{(2)}_{k_2}(\cdot,x_{I_2}))(x_2)\big|\\
&\quad\ls \prod_{i=1}^2  2^{-|k_i-\wz k_i|\epsilon}\bigg(\frac {2^{-k_i}+2^{-\wz k_i}}{|x_i-x_{I_i}|+2^{-k_i}+2^{-\wz k_i}}\bigg)^\epsilon  \nonumber\\
&\quad\quad\times   \frac1{m_\lz(I(x_i, 2^{-k_i}+2^{-\wz k_i} ))+m_\lz(I(x_{I_i}, 2^{-k_i}+2^{-\wz k_i}))+m_\lz(I(x_i, |x_i-x_{I_i}|))}.\nonumber
\end{align}
Note that the right-hand side of the above inequality is independent of $t_1$ and $t_2$. By substituting \eqref{eee2} back to the right-hand side of \eqref{eee1}, we have that
\begin{align*}
 \mathbb L&\ls\sum_{\wz k_1}\sum_{\wz k_2}   \bigg| \sum_{k_1}\sum_{I_1 \in\mathscr{X}^{k_1+N_1}} \sum_{k_2}\sum_{I_2 \in\mathscr{X}^{k_2+N_2}}   \prod_{i=1}^2m_\lambda(I_i)  2^{-|k_i-\wz k_i|\epsilon}\bigg(\frac {2^{-k_i}+2^{-\wz k_i}}{|x_i-x_{I_i}|+2^{-k_i}+2^{-\wz k_i}}\bigg)^\epsilon \\
 &\quad \quad\times\frac1{m_\lz(I(x_i, 2^{-k_i}+2^{-\wz k_i} ))+m_\lz(I(x_{I_i}, 2^{-k_i}+2^{-\wz k_i}))+m_\lz(I(x_i, |x_i-x_{I_i}|))}\\ &\quad \quad\times \lf|D^{(1)}_{k_1}D^{(2)}_{k_2}(g)(x_{I_1},x_{I_2})\r| \bigg|^2.
\end{align*}
Then, based on the estimates in the proof of Theorem 2.10 in \cite[pp.\, 335--336]{HLL2}, we have the following estimate:
\begin{align*}
 \mathbb L
& \ls \sum_{k_1} \sum_{k_2} 2^{-|k_1-\wz k_1|\epsilon}2^{-|k_2-\wz k_2|\epsilon} 2^{[(k_1\wedge \wz k_1)-k_1](2\lz+1)(1-{1\over r})}2^{[(k_2\wedge \wz k_2)-k_2](2\lz+1)(1-{1\over r})}\\
 &\quad\times\Bigg[ \mathcal{M}_1\bigg( \sum_{I_1 \in\mathscr{X}^{k_1+N_1}}  \mathcal{M}_2\Big( \sum_{I_2 \in\mathscr{X}^{k_2+N_2}}   \inf_{\gfz{y_1\in I_1}{y_2\in I_2}}\lf|D^{(1)}_{k_1}D^{(2)}_{k_2}(g)(y_1,y_2)\r|^r  \chi_{I_2}(\cdot)\Big)(x_2) \chi_{I_1}(\cdot) \bigg)(x_1) \Bigg]^{2\over r},
\end{align*}
where $r<1$, $a\wedge b:=\min\{a,b\}$, and
\begin{equation*}
\mathcal{M}_1f(x_1,x_2):=\dsup_{I \ni x_1}\dfrac1{m_\lz(I)}\dint_I |f(y_1, x_2)|\,\dmz(y_1),
\end{equation*}
and
\begin{equation*}
\mathcal{M}_2f(x_1,x_2):=\dsup_{J \ni x_2}\dfrac1{m_\lz(J)}\dint_J |f(x_1, y_2)|\,\dmz(y_2).
\end{equation*}
By taking the square root and then the $L^p$ norm on both sides of the above inequality and then using Fefferman--Stein's vector-valued maximal function inequality (\cite{HLL2}), we obtain that
\begin{eqnarray}\label{eeeee3}
&&\bigg\| \left\{  \int_0^\fz\int_0^\fz \lf|Q^{(1)}_{t_1}Q^{(2)}_{t_2}(g)(\cdot,\cdot)\r|^2 {dt_1\over t_1}{dt_2\over t_2} \right\}^{1/2}\bigg\|_{L^p(\rlz)}\\
&&\quad\ls \bigg\| \bigg\{  \sum_{k_1}\sum_{I_1 \in\mathscr{X}^{k_1+N_1}} \sum_{k_2}\sum_{I_2 \in\mathscr{X}^{k_2+N_2}} \lf|D^{(1)}_{k_1}D^{(2)}_{k_2}(g)(x_{I_1},x_{I_2})\r|^2 \chi_{I_1}(\cdot)\chi_{I_2}(\cdot) \bigg\}^{1/2}\bigg\|_{L^p(\rlz)}\noz\\
&&\quad= \big\| S_d(g) \big\|_{L^p(\rlz)},\noz
\end{eqnarray}
which implies that \eqref{product Littlewood--Paley eee} holds. Hence, we have that \eqref{product Littlewood--Paley} holds.

Following the same steps above, we now sketch the proof of \eqref{Littlewood--Paley}.  Applying the following version of Calder\'on's reproducing formula
\begin{align}\label{reproducing one para}
g(x_1,x_2)&=   \sum_{k_1}\sum_{I_1 \in\mathscr{X}^{k_1+N_1}}  m_\lambda(I_1)
\tilde{D}^{(1)}_{k_1}(x_1,x_{I_1})D^{(1)}_{k_1}(g)(x_{I_1},x_{2})
\end{align}
to the left-hand side of \eqref{Littlewood--Paley},
%
we obtain that
\begin{align}\label{eee111}
 \wz {\mathbb L}&:=\int_0^\fz \lf|Q^{(1)}_{t_1}(g)(x_1,x_2)\r|^2 {dt_1\over t_1}\\
 &=\sum_{\wz k_1\in\mathbb Z}\sum_{\wz k_2\in\mathbb Z}  \int_{2^{-\wz k_1-1}}^{2^{-\wz k_1}}\bigg| \sum_{k_1}\sum_{I_1 \in\mathscr{X}^{k_1+N_1}}  m_\lambda(I_1) Q^{(1)}_{t_1}(\tilde{D}^{(1)}_{k_1}(\cdot,x_{I_1}))(x_1)D^{(1)}_{k_1}(g)(x_{I_1},x_{2})\bigg|^2 {dt_1\over t_1}.\nonumber
\end{align}
Then using the almost orthogonality estimate for $Q^{(1)}_{t_1}(\tilde{D}^{(1)}_{k_1}(x_1,x_{I_1}))$, and based on the estimates in the proof of Theorem 2.10 in \cite[pp.\, 335--336]{HLL2}, we have the following estimate:
\begin{align*}
 \wz{\mathbb L}
 \ls \sum_{k_1} 2^{-|k_1-\wz k_1|\epsilon}2^{[(k_1\wedge \wz k_1)-k_1](2\lz+1)(1-{1\over r})}
 \Bigg[ \mathcal{M}_1\bigg( \sum_{I_1 \in\mathscr{X}^{k_1+N_1}}    \inf_{y_1\in I_1}\lf|D^{(1)}_{k_1}(g)(y_1,x_2)\r|^r   \chi_{I_1}(\cdot) \bigg)(x_1) \Bigg]^{2\over r},
\end{align*}
where $r<1$
and the Fefferman--Stein vector-valued maximal function inequality, we obtain that \eqref{Littlewood--Paley} holds.
Similarly, we can obtain that \eqref{Littlewood--Paley} holds for $Q^{(2)}_{t_2}$.
\end{proof}

Based on the proof of Theorem \ref{thm product Littlewood--Paley} above, we have the following estimates related to the Littlewood--Paley $g$-function.
\begin{prop}\label{cor product Littlewood--Paley}
Let $p\in \left({2\lz+1\over 2\lz+2}, 1\right]$, $Q^{(1)}_{t_1}$ and $Q^{(2)}_{t_2}$ be the same as above. Then for every $g\in H^p(\rlz)\cap L^2(\rlz)$, we have that
\begin{eqnarray}\label{product Littlewood--Paley p leq 1}
\bigg\| \left\{  \int_0^\fz\int_0^\fz \Big|Q^{(1)}_{t_1}Q^{(2)}_{t_2}(g)(\cdot,\cdot)\Big|^2 {dt_1\over t_1}{dt_2\over t_2} \right\}^{1\over2}\bigg\|_{L^p(\rlz)}\ls \|S_d(g)\|_{L^p(\rlz)}.
\end{eqnarray}
\end{prop}
\begin{proof}
Following the same proof of Theorem \ref{thm product Littlewood--Paley} above, and noting that
based on the estimates in the proof of Theorem 2.10 in \cite[pp.\, 335--336]{HLL2}, we have the following estimate:
\begin{align*}
 \mathbb L
 &\ls \sum_{k_1} \sum_{k_2}2^{-|k_1-\wz k_1|\epsilon}2^{-|k_2-\wz k_2|\epsilon} 2^{[(k_1\wedge \wz k_1)-k_1](2\lz+1)(1-{1\over r})}2^{[(k_2\wedge \wz k_2)-k_2](2\lz+1)(1-{1\over r})}\\
 &\quad\ \ \ \times
 \Bigg[ \mathcal{M}_1\bigg( \sum_{I_1 \in\mathscr{X}^{k_1+N_1}}  \mathcal{M}_2\Big( \sum_{I_2 \in\mathscr{X}^{k_2+N_2}}   \inf_{\gfz{y_1\in I_1}{y_2\in I_2}}\lf|D^{(1)}_{k_1}D^{(2)}_{k_2}(g)(y_1,y_2)\r|^r  \chi_{I_2}(\cdot)\Big)(x_2) \chi_{I_1}(\cdot) \bigg)(x_1) \Bigg]^{2\over r}
\end{align*}
for ${2\lz+1\over 2\lz+2}<r<p$, where $\mathbb L$ is defined as in \eqref{eee1}.  Thus, we obtain that \eqref{eeeee3} holds  for ${2\lz+1\over 2\lz+2}<p\leq 1$, which implies \eqref{product Littlewood--Paley p leq 1}.
\end{proof}

Next we provide the $L^p$ (for $1<p<\infty$) boundedness of the product Littlewood--Paley area functions associated with the operators $Q^{(1)}_{t_1}$ and $Q^{(2)}_{t_2}$.
\begin{thm}\label{thm product Littlewood--Paley area}
Let $Q^{(1)}_{t_1}$ and $Q^{(2)}_{t_2}$ be the same as above. Then for $1<p<\infty$ and  for every $g\in L^p(\rlz)$, we have that
\begin{align}\label{product Littlewood--Paley area}
\bigg\| \left\{  \iint_{\Gamma(x) } \Big|Q^{(1)}_{t_1}Q^{(2)}_{t_2}(g)(y_1,y_2)\Big|^2 {d\mu_\lambda(y_1,y_2) dt_1  dt_2 \over t_1 m_\lz(I(x_1, t_1)) t_2 m_\lz(I(x_2, t_2))} \right\}^{1\over2}\bigg\|_{L^p(\rlz)}\ls \|g\|_{L^p(\rlz)},
\end{align}
where $\Gamma(x):=\Gamma(x_1)\times \Gamma(x_2)$.
\end{thm}

\begin{proof}
The proof of this theorem is similar to that of Theorem \ref{thm product Littlewood--Paley} and so we briefly sketch the proof.  From Calder\'on's reproducing formula \eqref{reproducing} and the almost orthogonality estimates \eqref{eee2}, we have
\begin{align*}
&  \iint_{\Gamma(x) } \Big|Q^{(1)}_{t_1}Q^{(2)}_{t_2}(g)(y_1,y_2)\Big|^2 {d\mu_\lambda(y_1,y_2) dt_1  dt_2 \over t_1 m_\lz(I(x_1, t_1)) t_2 m_\lz(I(x_2, t_2))} \\
& \,\,\,\, \ls  \sum_{k_1,\,k_2}\!
 \Bigg[ \mathcal{M}_1\bigg( \sum_{I_1 \in\mathscr{X}^{k_1+N_1}}  \mathcal{M}_2\Big( \sum_{I_2 \in\mathscr{X}^{k_2+N_2}}   \inf_{\gfz{y_1\in I_1}{y_2\in I_2}}|D^{(1)}_{k_1}D^{(2)}_{k_2}(g)(y_1,y_2)|^r  \chi_{I_2}(\cdot)\Big)(x_2) \chi_{I_1}(\cdot) \bigg)(x_1) \Bigg]^{2\over r}
\end{align*}
which implies that the left-hand side of \eqref{product Littlewood--Paley area} is bounded by $\big\| S_d(g) \big\|_{L^p(\rlz)}$, which together with \eqref{discrete Littlewood--Paley}, finishes the proof of Theorem \ref{thm product Littlewood--Paley area}.
\end{proof}


Note that from Proposition \ref{p-aoti} and Lemma \ref{l-seimgroup prop} ${\rm (S_{iii})}$, we
see that the kernels of $S(f)$ and $S_u(f)$ satisfy the conditions ${\rm (K_i)}, {\rm (K_{ii})}$ and ${\rm (K_{iii})}$ listed above.  As a direct consequence of Theorem \ref{thm product Littlewood--Paley area}, we have
\begin{thm}\label{thm Littlewood--Paley area functions}
The product Littlewood--Paley area functions $S(f)$ and $S_u(f)$ are bounded operators on
$L^p(\rlz)$, $1<p<\infty$.
\end{thm}

\section{Proof of Theorem \ref{thm H1 and classical H1}}
\label{sec:first theorem}


In this section, we prove Theorem \ref{thm H1 and classical H1}.
%
%
%
The main approach here is to  show that
$H^p_{\Delta_\lambda}(\rlz)$ has a particular atomic decomposition as in Theorem \ref{theorem-Hp atom decomp}

To begin with, we recall the following construction of $\psi$ in \cite{dy03}.  Let $\varphi:=-\pi i\chi_{\frac12<|x|<1}$ and
$\psi$ the Fourier transform of $\varphi$. That is,
$$\psi(s):=s^{-1}(2\sin (s/2)-\sin s).$$

Consider the operator
\begin{align}\label{eee psi}
\psi\big(t\sqrt{\Delta_\lz}\big):=
\big(t\sqrt{\Delta_\lz}\big)^{-1}\lf[2\sin\big(t\sqrt{\Delta_\lz}/2\big)-\sin\big(t\sqrt{\Delta_\lz}\big)\r].
\end{align}

\begin{prop}[\cite{dy03}]\label{p-kernel cancel-compact supp}
For all $t\in(0, \fz)$, $\psi\lf(t\sqrt{\Delta_\lz}\r)(1)=0$
and the kernel $K_{\psi(t\sqrt{\Delta_\lz})}$ of $\psi(t\sqrt{\Delta_\lz})$ has support contained in
$\{(x, y)\in\mathbb R_+\times \mathbb R_+:\, |x-y|\le t\}.$

\end{prop}

%

We now turn to the proof of our first main result, Theorem \ref{thm H1 and classical H1}.
\begin{proof}[\bf Proof of Theorem \ref{thm H1 and classical H1}]
Note that from the  definition of $H^p_{{\Delta_\lz}}( \rlz)$, we have $L^2(\rlz)\cap H^p_{{\Delta_\lz}}( \rlz)$ is dense in $H^p_{{\Delta_\lz}}( \rlz)$. Moreover, from  Proposition 2.19 in \cite{HLL2}, we have that
 $L^2(\rlz)\cap H^p(\rlz)$ is dense in $H^p(\rlz)$.
Thus, by a density argument, it suffices to show that
\begin{equation}\label{equiv: atomic H1}
L^2(\rlz)\cap H^p_{{\Delta_\lz}}( \rlz)=L^2(\rlz)\cap H^p(\rlz)
\end{equation}
with equivalent norms.

We first prove that for every $f\in L^2(\rlz)\cap H^p_{{\Delta_\lz}}( \rlz)$,
$f$ belongs to $H^p(\rlz)$ and
\begin{equation}\label{equiv: atomic H1 e1}
\|f\|_{H^p(\rlz)}\ls \|f\|_{H^p_{{\Delta_\lz}}( \rlz)}.
\end{equation}
To prove this argument, it suffices to show that for every $f\in L^2(\rlz)\cap H^p_{{\Delta_\lz}}( \rlz)$, $f$ has an atomic decomposition, with atoms satisfying the properties as in Definition \ref{def-of-p q atom}.

To see this, we adapt the proof of the atomic decomposition as in Proposition 3.4 in \cite{CDLWY} to our current setting of Bessel operators. We point out that
in \cite{CDLWY} they considered only the atomic decompositions for Hardy space $H^1_{L_1,\,L_2}(X_1\times X_2)$ associated with operators $L_1$ and $L_2$, but their methods also work for $p<1$.

Let $f\in L^2(\rlz)\cap H^p_{{\Delta_\lz}}( \rlz)$. For each
$j\in\zz$, define
$$\Omega_j:=\lf\{(x_1, x_2)\in \rlz:\, Sf(x_1,x_2)>2^j\r\},$$
\begin{equation*}
\begin{array}[b]{cl}
B_j&:=\lf\{R:=I^{k_1}_{\alpha_1}\times I^{k_2}_{\alpha_2}:\,
\mu_\lz(R\cap \Omega_j)>\mu_\lz(R)/2,\,\mu_\lz((R\cap \Omega_{j+1})\le \mu_\lz(R)/2\r\},
\end{array}
\end{equation*}
and
$$\widetilde{\Omega}_j:=\lf\{(x_1, x_2)\in \rlz:\, \cms(\chi_{\Omega-j})(x_1,x_2)>1/2\r\}.$$
Here we use $I^{k_i}_{\alpha_i}$ ($i=1,2$) to denote the dyadic intervals as stated in Section 2 and
$\cms$ is the maximal function defined by
\begin{equation}\label{stro max func-defn}
{\cms}(f)(x_1, x_2):=\dsup_{\gfz{I\ni x_1,\, J\ni x_2}{R:=I\times J}}\dfrac1{\mu_\lz(R)}\iint_R|f(y_1,y_2)|\,d\mu_\lz(y_1,y_2),
\end{equation}
where the supremum is taken over all rectangles $R:=I\times J$ with intervals $I,\,J\subset\R_+$.
For each dyadic rectangle $R:= I^{k_1}_{\alpha_1}\times I^{k_2}_{\alpha_2}$ in $\rlz$, the tent $T(R)$ is defined as
$$T(R):= \lf\{(y_1, y_2, t_1, t_2): (y_1, y_2)\in R,\, t_1\in\lf(2^{-k_1}, 2^{-k_1+1}\r], t_2\in\lf(2^{-k_2}, 2^{-k_2+1}\r]\r\}.$$

We now consider the following reproducing formula
\begin{eqnarray*}
f(x_1, x_2)&&=C_\psi\dint_0^\infty\dint_0^\infty\psi(t_1\sqrt{\Delta_\lz})\psi(t_2\sqrt{\Delta_\lz})
\big( t_1\sqrt{\Delta_\lz}e^{-t_1\sqrt{\Delta_\lz}} \otimes t_2\sqrt{\Delta_\lz}e^{-t_2\sqrt{\Delta_\lz}}\big) f(x_1, x_2)\,\frac{dt_1}{t_1}\frac{dt_2}{t_2},
\end{eqnarray*}
in the sense of $L^2(\rlz)$, where $\psi(t_1\sqrt{\Delta_\lz})$ and $\psi(t_2\sqrt{\Delta_\lz})$ are defined as in \eqref{eee psi} and $C_\psi$ is a constant depending on $\psi$ (see (3.13) in \cite{CDLWY}). Then we have
\begin{eqnarray}\label{eeee1}
f(x_1, x_2)&&=C_\psi\sum_{j\in\zz}\sum_{R\in B_j}\iiiint_{T(R)}\psi(t_1\sqrt{\Delta_\lz})(x_1,y_1)\psi(t_2\sqrt{\Delta_\lz})(x_2,y_2)\\
 &&\quad\quad \big(t_1\sqrt{\Delta_\lz}e^{-t_1\sqrt{\Delta_\lz}} \otimes t_2\sqrt{\Delta_\lz}e^{-t_2\sqrt{\Delta_\lz}}\big)(f)(y_1, y_2)\,d\mu_\lz(y_1,y_2)\frac{dt_1}{t_1}\frac{dt_2}{t_2}\nonumber\\
&&=:\sum_{j\in\zz}\az_ja_j(x_1, x_2)\nonumber
\end{eqnarray}
with
\begin{eqnarray*}
\az_j&&:=C_\psi\bigg\|\bigg(\sum_{R\in B_j}\dint_0^\infty\dint_0^\infty
\lf|\big(t_1\sqrt{\Delta_\lz}e^{-t_1\sqrt{\Delta_\lz}} \otimes t_2\sqrt{\Delta_\lz}
e^{-t_2\sqrt{\Delta_\lz}}\big)f(\cdot, \cdot)\r|^2\chi_{T(R)}\frac{dt_1}{t_1}\frac{dt_2}{t_2}\bigg)^{\frac12}\bigg\|_{L^2(\rlz)}\\
&&\quad\quad\times\mu_\lz(\widetilde \Omega_j)^{1/p-1/2},
\end{eqnarray*}
and
$$a_j(x_1, x_2)=\sum_{\bar R\in m(\widetilde \Omega_j)}a_{j,\,\bar R}(x_1, x_2),$$
where
\begin{eqnarray*}
a_{j,\,\bar R}&&:=\sum_{R\in B_j,\,R\subset \bar R}\frac1{\az_j}\iiiint_{T(R)}\psi(t_1\sqrt{\Delta_\lz})(x_1,y_1)\psi(t_2\sqrt{\Delta_\lz})(x_2,y_2)
 \\
 &&\quad\quad \big(t_1\sqrt{\Delta_\lz}e^{-t_1\sqrt{\Delta_\lz}} \otimes t_2\sqrt{\Delta_\lz}e^{-t_2\sqrt{\Delta_\lz}}\big)(f)(y_1, y_2)\,d\mu_\lz(y_1,y_2)\frac{dt_1}{t_1}\frac{dt_2}{t_2}.
\end{eqnarray*}

Following the proof of Proposition 3.4 in \cite{CDLWY}, we deduce that
\begin{align}\label{eeeee2}
\sum_j|\alpha_j|^p \leq C\|f\|_{H^p_{{\Delta_\lz}}( \rlz)}^p
\end{align}
and that for each $j$, $a_j$ is supported in $\widetilde \Omega_j$, and that $\|a_j\|_{L^2(\rlz)}\leq \mu_\lz(\widetilde \Omega_j)^{1/2-1/p},$
which implies that $a_j$ satisfies (1) and (2)
in Definition \ref{def-of-p q atom}.
%
%
%
Moreover,  each $a_{j,\,\bar R}(x_1, x_2)$ is supported in $CR$, where $C$ is a fixed positive constant, and
$$ \sum_{\bar R\in m(\widetilde \Omega_j)} \|a_{j,\,\bar{R}}\|^2_{L^2(\rlz)}\leq \mu_\lz(\widetilde \Omega_j)^{1-2/p}.$$
By Proposition \ref{p-kernel cancel-compact supp}, we also see that
$$\dint a_{j,\,\bar R}(x_1, x_2)\,dm_\lz(x_1)=\dint a_{j,\,\bar R}(x_1, x_2)\,dm_\lz(x_2)=0.$$
This shows that $a_j$ satisfies (3) in Definition \ref{def-of-p q atom}.  Combining these results, we get that for each $j$, $a_j$ is an atom as in Definition \ref{def-of-p q atom}. Hence, applying the result of Theorem \ref{theorem-Hp atom decomp} and \eqref{eeee1} and \eqref{eeeee2}, we obtain that \eqref{equiv: atomic H1 e1} holds.

Next we prove that for every $f\in L^2(\rlz)\cap H^p(\rlz) $,
$f$ belongs to $H^p_{{\Delta_\lz}}( \rlz)$ and
\begin{equation}\label{equiv: atomic H1 e2}
\|f\|_{H^p_{{\Delta_\lz}}( \rlz)}\ls\|f\|_{H^p(\rlz)} .
\end{equation}
To see this, note that for every $f\in L^2(\rlz)\cap H^p(\rlz) $, from Theorem \ref{theorem-Hp atom decomp}, we obtain that
$ f=\sum_k \lambda_k  a_k,$
where each $a_k$ is an atom in Definition \ref{def-of-p q atom} and $\sum_k|\lambda_k|^p\ls\|f\|_{H^p(\rlz)}^p$.

As pointed out in Section 2.3, $S(f)$ is bounded on $L^2(\rlz)$ (Theorem \ref{thm Littlewood--Paley area functions}), and the kernels of $S(f)$ satisfy the conditions ${\rm (K_i)}, {\rm (K_{ii})}$ and ${\rm (K_{iii})}$. Then, we have
\begin{align*}
\|S(a_k)\|_{L^p(\rlz)}\ls 1,
\end{align*}
where  the implicit constant is independent of $a_k$.
As a consequence, we have
\begin{align*}
\|S(f)\|_{L^p(\rlz)}^p\leq \sum_k|\lambda_k|^p\|S(a_k)\|_{L^p(\rlz)}^p \ls \|f\|_{H^p(\rlz)}^{p},
\end{align*}
which implies that \eqref{equiv: atomic H1 e2} holds.

Combining the results of \eqref{equiv: atomic H1 e1} and \eqref{equiv: atomic H1 e2}, we get that \eqref{equiv: atomic H1} holds, with equivalent norms. This completes the proof of Theorem \ref{thm H1 and classical H1}.
\end{proof}

%
%


Based on the proof of Theorem \ref{thm H1 and classical H1} above, we now prove that the three versions of Hardy spaces,
i.e.,  $H^p_{\Delta_\lz,\,1}( \rlz )$, $H^p_{\Delta_\lz,\,2}( \rlz )$ and $H^p_{\Delta_\lz,\,3}( \rlz )$, as define in Remark \ref{remark}, coincide with $H^{p}(\rlz)$.
\begin{prop}\label{prop eee}
For $p\in \left({2\lz+1\over2\lz+2},1\right]$, the Hardy spaces $H^p_{\Delta_\lz,\,1}( \rlz )$, $H^p_{\Delta_\lz,\,2}( \rlz )$ and $H^p_{\Delta_\lz,\,3}( \rlz )$ coincide with $H^{p}(\rlz)$ and they have equivalent norms (or quasi-norms).
\end{prop}
\begin{proof}
We first consider $H^p_{\Delta_\lz,\,1}( \rlz)$ defined by the Littlewood--Paley area function via the heat semigroup $\{e^{-t\tbz}\}_{t>0}$. Since the kernel of $\{e^{-t\tbz}\}_{t>0}$ satisfies conditions ${\rm (K_i)}, {\rm (K_{ii})}$ and ${\rm (K_{iii})}$,
following the proof of Theorem \ref{thm H1 and classical H1}, we obtain that  $H^p_{\Delta_\lz,\,1}( \rlz )$ coincides with $H^{p}(\rlz)$ and they have equivalent norms (or quasi-norms).

Next we consider  $H^p_{\Delta_\lz,\,2}( \rlz)$ defined by the Littlewood--Paley $g$-function via the Poisson semigroup $\{e^{-t\sqrt{\tbz}}\}_{t>0}$. Suppose $f\in H^p(\rlz)$. Then from Proposition \ref{cor product Littlewood--Paley}, we obtain that $\|g(f)\|_{L^p(\rlz)}\ls\|S_d(f)\|_{L^p(\rlz)}$, which shows that
$H^p_{\Delta_\lz,2}( \rlz)\supseteq H^p(\rlz)$. Conversely, suppose $f\in H^p_{\Delta_\lz,\,2}( \rlz)$, following the proof of Proposition \ref{cor product Littlewood--Paley}, we obtain that $\|S_d(f)\|_{L^p(\rlz)} \ls\|g(f)\|_{L^p(\rlz)}$ and hence $H^p_{\Delta_\lz,\,2}( \rlz)\subseteq H^p(\rlz)$. As a consequence, we get that  $H^p_{\Delta_\lz,\,2}( \rlz)$ coincides with $H^p(\rlz)$ and they have equivalent norms. Similar argument holds for $H^p_{\Delta_\lz,\,3}( \rlz)$.
\end{proof}

\section{Proof of Theorem \ref{thm: char of Hardy spacs by max func}}
\label{sec:second thm}

This section is devoted to the proof of Theorem \ref{thm: char of Hardy spacs by max func}.
To this end, we will prove the chain of six inequalities as in \eqref{loop} by the following six steps, respectively.


\medskip
{\bf Step 1:\ $\|f\|_{H^p_{\Delta_\lz}( \rlz)} \leq
\|f\|_{H^p_{S_u}( \rlz)}$} for $f\in H^p_{S_u}( \rlz)\cap L^2(\rlz)$.



Note that from the definitions of the area functions $Sf$ and $S_uf$ in \eqref{esf} and \eqref{esfu} respectively, we have
for $f\in L^2(\rlz)$, $S(f)(x) \leq S_u(f)(x)$, which implies that $\|f\|_{H^p_{\Delta_\lz}( \rlz)} \leq
\|f\|_{H^p_{S_u}( \rlz)}$.
%
%

\bigskip

{\bf Step 2:\ $\|f\|_{H^p_{S_u}( \rlz)} \ls
\|f\|_{H^p_{\cn_{P}}(\rlz)}$} for $f\in H^p_{\cn_{P}}( \rlz)\cap L^2(\rlz)$.

\medskip
Recall again
\begin{align*}
\gratx u(t,x):=(\prz_t u, \prz_x u),\, \,\deltx u(t, x):=\dlz+\pa_t^2=-\mathcal D^\ast \mathcal D+\pa_t^2,
\end{align*}
where $\mathcal D^\ast:=-\pa_x - {2\lz\over x}\pa_x$ is the formal adjoint operator of $\mathcal D:=\pa_x$ in $\ltz$.


Next we introduce the following lemma about finding the ``conjugate pair'' of functions $(\phi,\psi)$, which plays a key role in this step.
\begin{lem}\label{lem: constru of func psi}
Let $\phi\in C^\fz_c(\R_+)$ such that $\phi\ge 0$, $\supp(\phi)\subset (0, 1)$ and
\begin{equation*}
\inzf \phi(x)\,\dmz(x)=1.
\end{equation*}
Then there exists a function $\psi(t, x, y)$ on $\R_+\times\R_+\times\R_+$, such that
\begin{itemize}
\item[(i)]for any function $f\in\ltz$ and $t,\,x\in \R_+$,
\begin{equation}\label{CR equa:psi}
\prz_t (\phi_t\slz f)(x)= -\mathcal D^*[\psi(f)(t,x)]= \prz_x[\psi(f)(t,x)]+\frac{2\lz}x\psi(f)(t,x),
\end{equation}
where
$$\psi(f)(t,x):=\inzf \psi(t, x, y)f(y)\dmzy;$$

  \item [(ii)] for any $t,\, x,\, y$ with $|x-y|\ge t,$
  \begin{equation}\label{comp supp psi}
\psi(t, x, y)\equiv0;
\end{equation}
  \item [(iii)] $\psi(t, x, y)$ satisfies the conditions ${\rm (K_i)}$, ${\rm (K_{ii})}$ and ${\rm (K_{iii})}$ as in Section 2.3.
%
\end{itemize}
\end{lem}
\begin{proof}
First, by 
\eqref{Hankel trans}, we observe that
\begin{eqnarray*}
\prz_t\tlz\phi_t(y)&&=-c_\lz t^{-2\lz-2}\dint_0^\pi
 (\sin\theta)^{2\lz-1}\lf[(2\lz+1)\phi\lf(\frac{\sqrt{x^2+y^2-2xy\cos\theta}}t\r)\r.\\
&&\quad\lf.+\frac{\sqrt{x^2+y^2-2xy\cos\theta}}t\phi^\prime
\lf(\frac{\sqrt{x^2+y^2-2xy\cos\theta}}t\r)\r]\,d\theta\noz.
\end{eqnarray*}
Note that \eqref{CR equa:psi} holds if $\psi$ satisfies  that for any $t, x,y$,
$$\prz_t\tlz\phi_t(y)=\prz_x\psi(t, x, y)+\frac{2\lz}x\psi(t, x, y).$$
Thus, we define
\begin{eqnarray}\label{defi of psi}
\quad\quad\,\,\psi(t, x, y):&&=-c_\lz t^{-2\lz-2}x^{-2\lz}\int_0^x \dint_0^\pi (\sin\theta)^{2\lz-1}\lf[(2\lz+1)\phi\lf(\frac{\sqrt{w^2+y^2-2wy\cos\theta}}t\r)\r.\\
&&\quad\lf.+\frac{\sqrt{w^2+y^2-2wy\cos\theta}}t\phi^\prime
\lf(\frac{\sqrt{w^2+y^2-2wy\cos\theta}}t\r)\r]\,d\theta w^{2\lz}dw\noz.
\end{eqnarray}
Then it is easy to see that $\psi$ satisfies the equation \eqref{CR equa:psi}.

Now we prove that \eqref{comp supp psi} holds.
In fact, for all $x,\,y,\,z\in(0, \fz)$, let
$\triangle(x, y, z)$ be the area of a triangle with sides $x$, $y$,
$z$ if such a triangle exists. And then we define
\begin{equation*}
D(x, y, z):=c_\lz 2^{2\lz-2}(xyz)^{-2\lz+1}[\triangle(x, y, z)]^{2\lz-2}
\end{equation*}
if $\triangle(x, y, z)\not=0$, and $D(x, y, z):=0$ otherwise.
By a change of variables argument, we obtain that
\begin{equation*}
\begin{array}[t]{cl}
&c_\lz\dint_0^\pi \phi_t\lf(\sqrt{x^2+y^2-2xy\cos \theta}\r)
(\sin\theta)^{2\lz-1}\,d\theta=\dint_0^\fz \phi_t(z)D(x, y, z)\,\dmz(z).
\end{array}
\end{equation*}
Recall that for all $x,\,z\in (0, \fz)$,
\begin{equation}\label{integ of func D:integ}
\dint_0^\fz D(x, y, z)\,\dmz(y)=1;
\end{equation}
see in \cite[p.\,335,\,(6)]{h}.
%
By change of variables, we write
\begin{eqnarray}\label{repres of psi}
\quad\psi(t, x,y)&=&-t^{-2\lz-2}x^{-2\lz}\int_0^x\inzf D(w, y,z)\lf[(2\lz+1)\phi\lf(\frac zt\r)+\frac zt\phi^\prime\lf(\frac zt\r)\r] z^{2\lz}\,dzw^{2\lz}dw\\
&=&-t^{-1}x^{-2\lz}\int_0^x\inzf D(w, y,z)\prz_z\lf[\lf(\frac zt\r)^{2\lz+1}\phi\lf(\frac zt\r)\r]\,dzw^{2\lz}\,dw\noz\\
&=&-t^{-1}x^{-2\lz}\inzf\int_0^x D(w, y,z)\prz_z\lf[\lf(\frac zt\r)^{2\lz+1}\phi\lf(\frac zt\r)\r]w^{2\lz}\,dw\,dz.\noz
\end{eqnarray}

We first prove $\psi(t, x, y)=0$ if $x>y+t$. To this end, recall that $\supp(\phi)\subset (0, 1)$. Then $\phi(z/t)\not=0$ only if $z\in(0, t)$. Also, by the definition of $D(w, y,z)$,
we see that $D(w, y,z)\not=0$ only if $|y-z|< w<y+z$. Then by \eqref{integ of func D:integ} and
 the fact that $\phi\in C^\fz_c(\R_+)$, we have that
\begin{eqnarray*}
\psi(t, x, y)&=&-t^{-1}x^{-2\lz}\inzf\int_0^\fz D(w, y,z)\prz_z\lf[\lf(\frac zt\r)^{2\lz+1}\phi\lf(\frac zt\r)\r]w^{2\lz}\,dw\,dz\\
&=&-t^{-1}x^{-2\lz}\inzf\int_0^\fz D(w, y,z)w^{2\lz}\,dw\prz_z\lf[\lf(\frac zt\r)^{2\lz+1}\phi\lf(\frac zt\r)\r]\,dz\\
&=&-t^{-1}x^{-2\lz}\inzf\prz_z\lf[\lf(\frac zt\r)^{2\lz+1}\phi\lf(\frac zt\r)\r]\,dz=0.
\end{eqnarray*}

On the other hand, assume that $x<y-t$. Then by the compact support of $\phi$, we see that
$w\le x<y-t\le y-z\le |y-z|$. This together with the definition of $D(w, y, z)$ implies that
$\psi(t, x, y)=0$.

Now we show $\psi$ satisfies ${\rm (K_i)}$. By \eqref{comp supp psi} and the doubling property of $dm_\lz$,  it suffices to show that
for any $t, \,x,\,y$ such that $|x-y|<t$,
\begin{equation}\label{size of psi-2}
|\psi(t, x,y)|\ls \frac1{m_\lz(I(x, t))}.
\end{equation}
From \eqref{repres of psi}, \eqref{integ of func D:integ} and $\phi\in C^\fz_c(\R_+)$, we deduce that
\begin{eqnarray*}
|\psi(t, x,y)|&\ls& t^{-1}x^{-2\lz}\lf|\inzf\int_0^x D(w, y,z)\prz_z\lf[\lf(\frac zt\r)^{2\lz+1}\phi\lf(\frac zt\r)\r]w^{2\lz}\,dw\,dz\r|\noz \ls t^{-1}x^{-2\lz}.
\end{eqnarray*}
Moreover, if $x\le t$, then from \eqref{repres of psi}, \eqref{integ of func D:integ} and $\phi\in C^\fz_c(\R_+)$,
it follows that
\begin{eqnarray*}
|\psi(t, x,y)|&\ls&t^{-2\lz-2}x^{-2\lz}\int_0^x\inzf D(w, y,z)\lf|(2\lz+1)\phi\lf(\frac zt\r)+\frac zt\phi^\prime\lf(\frac zt\r)\r| z^{2\lz}\,dzw^{2\lz}dw\\
&\ls&t^{-2\lz-2}x^{-2\lz}\int_0^x\inzf D(w, y,z) z^{2\lz}\,dzw^{2\lz}dw \ls t^{-2\lz-2}x\ls t^{-2\lz-1}.
\end{eqnarray*}
Thus, ${\rm(K_i)}$ holds.

Now we show $\psi$ satisfies ${\rm (K_{ii})}$.  We first observe that by the assumption that $|y-\wz y|\le (t+|x-y|)/2$
and the doubling property of $dm_\lz$,
$|x-y|+t\sim |x-\wz y|+t$
and
$$m_\lz(I(x, t))+m_\lz(I(y,t))+m_\lz(I(x, |x-y|))\sim m_\lz(I(x, t))+m_\lz(I(\wz y,t))+m_\lz(I(x, |x-\wz y|)).$$
Based on these facts, we may further assume that $y>\wz y$. For otherwise, it is sufficient to show that
\begin{eqnarray}\label{regular of psi-3}\quad\quad
|\psi(t, x, y)-\psi(t, x, \wz y)|\ls \frac1{m_\lz(I(x, t))+m_\lz(I(\wz y,t))+m_\lz(I(x, |x-\wz y|))}\frac {t|y-\wz y|}{(|x-\wz y|+t)^2}.
\end{eqnarray}

Since $y>\wz y$, if $x>y+t$, then $x>\wz y+t$ and by \eqref{comp supp psi}, we see that $\psi(t, x, y)=0$ and $\psi(t, x, \wz y)=0$.
Similarly, if $x<\wz y-t$, then $x<y-t$ and by \eqref{comp supp psi} again, $\psi(t, x, y)=0$ and $\psi(t, x, \wz y)=0$.
Hence, ${\rm (K_{ii})}$ holds trivially if $x>y+t$ or $x<\wz y-t$.
Moreover,  observe that
$$[\wz y-t, y+t]=[\wz y-t, \wz y+t]\cup [y-t, y+t].$$
%
Therefore, by similarity and \eqref{regular of psi-2}, we only need to consider the case that $y-t\le x\le y+t$.
It  then suffices to show that
\begin{equation}\label{regular of psi-2}
|\psi(t, x, y)-\psi(t, x,\wz y)|\ls \frac1{m_\lz(I(y,t))+m_\lz(I(y, |x-y|))}\frac {|y-\wz y|}{t}.
\end{equation}
We write
\begin{eqnarray*}
&&|\psi(t, x, y)-\psi(t, x,\wz y)|\\
&&\quad\ls  t^{-2\lz-2}x^{-2\lz}\dint_0^\pi\int_0^x (\sin\theta)^{2\lz-1}\lf|\phi\lf(\frac{\sqrt{w^2+y^2-2wy\cos\theta}}t\r)\r.\\
&&\hskip6cm\lf.-\phi\lf(\frac{\sqrt{w^2+\wz y^2-2w\wz y\cos\theta}}t\r)\r|w^{2\lz}dwd\theta\\
&&\quad\quad+ t^{-2\lz-2}x^{-2\lz}\dint_0^\pi\int_0^x (\sin\theta)^{2\lz-1}\lf|\frac{\sqrt{w^2+y^2-2wy\cos\theta}}t\phi^\prime\lf(\frac{\sqrt{w^2+y^2-2wy\cos\theta}}t\r)\r.\\
&&\hskip4cm\lf.-\frac{\sqrt{w^2+{\wz y}^2-2w\wz y\cos\theta}}t\phi^\prime\lf(\frac{\sqrt{w^2+{\wz y}^2-2w\wz y\cos\theta}}t\r)\r|w^{2\lz}dwd\theta\\
&&\quad=:{\rm I}+{\rm II}.
\end{eqnarray*}
We only consider the term ${\rm I}$, since for the term ${\rm II}$, we consider the function $\wz\phi(x):=x \phi^\prime (x)$, and then the form of  ${\rm II}$ will be the same as  ${\rm I}$.  We study the following four cases:

For the term ${\rm I}$, we first note that from the mean value theorem,
\begin{eqnarray*}
{\rm I}&& \ls t^{-2\lz-2}x^{-2\lz}\dint_0^\pi\int_0^x (\sin\theta)^{2\lz-1}\lf|\phi^\prime\lf(\frac{\sqrt{w^2+\xi^2-2w\xi\cos\theta}}t\r)\r|\\
&&\quad\times \frac{|\xi-w\cos\theta||y-\wz y|}{t\sqrt{w^2+\xi^2-2w\xi\cos\theta}}w^{2\lz}dwd\theta\\
&&\ls t^{-2\lz-2}x^{-2\lz}\dint_0^\pi\int_0^x (\sin\theta)^{2\lz-1}\lf|\phi^\prime\lf(\frac{\sqrt{w^2+\xi^2-2w\xi\cos\theta}}t\r)\r|
\frac{|y-\wz y|}{t}w^{2\lz}dwd\theta.
\end{eqnarray*}
To continue, we consider the following three cases.

Case (i) $t<y\le 8t$. 
In this case, by $x<y+t$, we have $x<2y\le 16t$. Again, since $|\phi'(x)|\ls 1$, we get that
\begin{eqnarray*}
{\rm I}
&\ls & x^{-2\lz}|y-\wz y|\ t^{-2\lz-3}\dint_0^\pi\int_0^x (\sin\theta)^{2\lz-1}
w^{2\lz}dwd\theta\\
&\ls & x^{-2\lz}|y-\wz y|\ t^{-2\lz-3} x^{2\lz+1}\\
&\ls & |y-\wz y|t^{-2\lz-2}\\
&\ls& \frac1{m_\lz(I(y,t))+m_\lz(I(y, |x-y|))}\frac {|y-y'|}{t}.
\end{eqnarray*}
%

Case (ii) $8t<y$ and $|x-y|<y/2$.  In this case,
$$|y-\xi|<|y-\wz y|\le \frac{t+|x-y|}{2}<\frac5{16}y,$$
which implies that $\xi\sim y\sim \wz y \sim x$. Thus,
we see that
\begin{eqnarray*}
{\rm I}
&\ls & t^{-2\lz-3}y^{-2\lz}|y-\wz y|\dint_0^\fz\int_0^\fz D(\xi, w, z) \lf|\phi^\prime\lf(\frac{z}t\r)\r|
w^{2\lz}dw z^{2\lz}dz\\
&\ls & t^{-2\lz-3}y^{-2\lz}|y-\wz y|\int_0^\fz  \lf|\phi^\prime\lf(\frac{z}t\r)\r|z^{2\lz}dz\\
&\ls & t^{-2}y^{-2\lz}|y-\wz y|\int_0^\fz  \lf|\phi^\prime\lf(z\r)\r|z^{2\lz}dz\\
&\ls & t^{-2}y^{-2\lz}|y-\wz y|\\
&\ls& \frac1{m_\lz(I(y,t))+m_\lz(I(y, |x-y|))}\frac {|y-y'|}{t}.
\end{eqnarray*}

Case (iii)  $8t< y$ and $|x-y|\ge y/2$.  In this case, by $x<y+t$, we have $x\le y/2$. Moreover,
$$\xi>\wz y=y-(y-\wz y)>y-\frac{t+y-x}2=\frac{y+x-t}2>x,$$
which implies that for any $w\in(0, x)$ and $\theta\in(0, \pi)$,
$$\sqrt{w^2+\xi^2-2w\xi\cos\theta}>\xi-w>\xi-x>\frac{y-x-t}2\ge \frac y8\ge t.$$
This together with $\supp(\phi)\subset (0, 1)$ shows that ${\rm I}=0$.

Combining the cases above we conclude that \eqref{regular of psi-2} holds, which implies ${\rm (K_{ii})}$.


Finally, we show that ${\rm (K_{iii})}$ holds.
Indeed, by \eqref{repres of psi} together with \eqref{integ of func D:integ} and $\phi\in C^\fz_c(\R_+)$, we conclude that
\begin{eqnarray*}
&&\inzf \psi(t, x, y)\dmz(y)\\
&&\quad=-t^{-1}x^{-2\lz}\inzf\inzf\int_0^x D(w, y,z)\prz_z\lf[\lf(\frac zt\r)^{2\lz+1}\phi\lf(\frac zt\r)\r]w^{2\lz}\,dw\,dz\dmz(y)\\
&&\quad=-t^{-1}x^{-2\lz}\int_0^x\inzf\prz_z\lf[\lf(\frac zt\r)^{2\lz+1}\phi\lf(\frac zt\r)\r]\,dz w^{2\lz}\,dw\\
&&\quad=-\frac x{(2\lz+1)t}\inzf\prz_z\lf[\lf(\frac zt\r)^{2\lz+1}\phi\lf(\frac zt\r)\r]\,dz =0.
\end{eqnarray*}
This shows ${\rm (K_{iii})}$, and hence finishes the proof of Lemma \ref{lem: constru of func psi}.
\end{proof}

\begin{lem}\label{lem:Merryfield lem}
Let $p\in((2\lz+1)/(2\lz+2), 1]$ and $\phi$ be as in Lemma \ref{lem: constru of func psi}.
Then there exists a positive constant $C$ such that
for any $f,\,g\in\ltz$ with $u(t, x):=\plz f(x)$ satisfying $\sup_{|y-x|<t}|u(t,y)|\in L^p(\R_+,\dmz(x))$,
\begin{eqnarray*}
{\rm I}&&:= \dinrp|\gratx u(t,x)|^2\lf|\phi_t\sharp_\lz g(x)\r|^2 t\dmz(x)\,dt\\
&&\le C \lf[\inrp [f(x)]^2[g(x)]^2\dmz(x)+\dinrp|u(t,x)|^2|Q_t(g)(x)|^2\dlmxt\r],\noz
\end{eqnarray*}
where $Q_t( g)(x) := \big( t \partial_t (\phi_t\sharp_\lz g)(x),\ t \partial_x (\phi_t\sharp_\lz g)(x),\psi( g)(t,x)\big)$ is a vector-valued function, with $\psi( g)(t,x)$ as obtained in Lemma \ref{lem: constru of func psi}.
\end{lem}

\begin{proof}
First, we claim that $u(t, x)\to0$ as $t\to \fz$.
Indeed, observe that for any $x,y,\,t\in \R_+$ with $|y-x|<t$,
$$|u(t, x)|=\lf|\plz f(x)\r|\le \sup_{|y-z|<t}|u(t,z)|.$$
Since $\sup_{|x-y|<t}|u(t,y)|\in\lpz$, we have that
\begin{eqnarray*}
\lf|\plz f(x)\r|^p&\le& \frac1{m_\lz(I(x, t))} \dint_{I(x,\, t)}\bigg|\sup_{|x-y|<t}|u(t,y)|\bigg|^p\dmzy\\
&\le& \frac1{m_\lz(I(x, t))}\lf\|\sup_{|x-y|<t}|u(t,y)|\r\|_\lpz^p.
\end{eqnarray*}
 This means that $u(t, x)\to0$ as $t\to \fz$ and the claim follows.

We now claim that
\begin{equation}\label{equality for gradient}
2|\gratx u(t,x)|^2=\deltx(u^2(t,x)).
\end{equation}
In fact, recall that $u$ satisfies the equation \eqref{bessel laplace equation}.
We then see that
\begin{eqnarray*}
\deltx(u^2(t,x))=\prz_t^2 u^2+\prz_x^2 u^2+\frac{2\lz}x\prz_x u^2=2\lf[(\prz_t u)^2+(\prz_x u)^2\r]=2|\gratx u(t,x)|^2.
\end{eqnarray*}
This implies claim \eqref{equality for gradient}.

From the claim \eqref{equality for gradient} and integration by parts, we deduce that
\begin{eqnarray*}
2{\rm I}&&=\dinrp  \deltx(u^2(t,x))  \lf|\phi_t\sharp_\lz g(x)\r|^2 t\dmz(x)\,dt\\
&&=-\dinrp\mathcal D^\ast \mathcal D(u^2(t,x))(\phi_t\sharp_\lz g(x))^2 t\,\dmz(x)\,dt
\\
&&\quad\quad+\dinrp \prz^2_t (u^2(t,x))\lf[t(\phi_t\sharp_\lz g(x))^2\r]\,\dmz(x)\,dt\\
&&=-\dinrp  \prz_x(u^2(t,x))\ \prz_x (\phi_t\sharp_\lz g(x))^2 t\,\dmz(x)\,dt\\
&&\quad\quad-\dinrp \prz_t (u^2(t,x))\ \prz_t\lf[t(\phi_t\sharp_\lz g(x))^2\r]\,\dmz(x)dt\\
&&=-4\dinrp u(t,x)\ \prz_x u(t,x)\  \phi_t\sharp_\lz g(x)\  \prz_x(\phi_t\sharp_\lz g(x))t\,\dmz(x)\,dt\\
&&\quad\quad -2 \dinrp u(t,x) \prz_t (u(t,x))\ (\phi_t\sharp_\lz g(x))^2\,\dmz(x)dt\\
&&\quad\quad -4 \dinrp u(t,x) \prz_t (u(t,x))\ t (\phi_t\sharp_\lz g(x))\ \prz_t (\phi_t\sharp_\lz g(x))\,\dmz(x)dt\\
&&=-4\dinrp u(t,x)\gratx u(t,x)\cdot t (\phi_t\sharp_\lz g(x))\ \gratx (\phi_t\sharp_\lz g(x))\ \,\dmz(x)\,dt\\
&&\quad\quad - \dinrp  \prz_t (u^2(t,x))\ (\phi_t\sharp_\lz g(x))^2\,\dmz(x)dt\\
&&=:{\rm A}+{\rm B}.
\end{eqnarray*}

For the term ${\rm A}$, using H\"older's inequality and Cauchy's inequality, we obtain that
\begin{eqnarray}\label{esti A merryfield lem}
|{\rm A}|&\le& 4\lf[\dinrp u^2(t,x)|\gratx (\phi_t\sharp_\lz g(x))|^2 t \dmz(x)\,dt\r]^{\frac12} \\
&&\quad\times\lf[\dinrp \lf|\gratx u(t,x)\r|^2\lf|\phi_t\sharp_\lz g(x)\r|^2t \dmz(x)\,dt\r]^{\frac12}\noz\\
&\le & C\dinrp u^2(t,x)|\gratx (\phi_t\sharp_\lz g(x))|^2 t \dmz(x)\,dt\noz\\
&&\quad + {1\over 8}\dinrp \lf|\gratx u(t,x)\r|^2\lf|\phi_t\sharp_\lz g(x)\r|^2 t \dmz(x)\,dt\noz\\
&= & C\dinrp u^2(t,x)|t\gratx (\phi_t\sharp_\lz g(x))|^2 {\dmz(x)\,dt \over t}
 + {{\rm I}\over 8}.\noz
\end{eqnarray}

For term ${\rm B}$, from integration by parts, we have
\begin{eqnarray*}
{\rm B}&=&  -  \int_{\R_+} u^2(t,x)\ (\phi_t\sharp_\lz g(x))^2\bigg|_{t=0}^{t=\infty} \dmz(x) +  \dinrp  u^2(t,x)\ \prz_t (\phi_t\sharp_\lz g(x))^2\,\dmz(x)dt \\
&=:& {\rm B}_1+{\rm B}_2.
\end{eqnarray*}
It is easy to see that
\begin{equation}\label{esti B1 merryfield lem}
{\rm B}_1 \sim\int_{\R_+} f(x)^2g(x)^2 \dmz(x).
\end{equation}
For ${\rm B}_2$, we get that
\begin{eqnarray*}
{\rm B}_2=  2\dinrp  u^2(t,x)\ \phi_t\sharp_\lz g(x)\ \prz_t (\phi_t\sharp_\lz g(x))\,\dmz(x)dt.
\end{eqnarray*}
Then, using  Lemma \ref{lem: constru of func psi}, for the function $\phi$, there exists a function $\psi(t,x,y)$ such that
$\psi$ satisfies the equation \eqref{CR equa:psi} and $\psi$ satisfies all properties listed in (ii)--(v) in Lemma \ref{lem: constru of func psi}. Hence, we get that
\begin{eqnarray*}
{\rm B}_2&=& 2 \dinrp  u^2(t,x)\ \phi_t\sharp_\lz g(x)\ \mathcal D^*\big(\psi(g)(t,x)\big)\dmz(x)dt\\
&=&  2\dinrp  u^2(t,x)\ \phi_t\sharp_\lz g(x)\ \prz_x(\psi(g))(t, x)\ x^{2\lambda}\, dxdt\\
&&\quad+  2\dinrp  u^2(t,x)\ \phi_t\sharp_\lz g(x)\ \frac{2\lz}{x}\psi(g)(t, x)\dmz(x)dt\\
&=:& 2\mathcal{B}_{21}+ 2\mathcal{B}_{22}.
\end{eqnarray*}
For  $\mathcal{B}_{21}$, integration by parts, gives
\begin{eqnarray*}
\mathcal{B}_{21} &=&  -\dinrp  \prz_x \bigg( u^2(t,x)\ \phi_t\sharp_\lz g(x)\ x^{2\lz}\bigg) \psi(g)(t, x)\ \, dxdt\\
&&  +  \int_{\R_+} u^2(t,x)\ \phi_t\sharp_\lz g(x)\ x^{2\lambda}\psi(g)(t, x)\bigg|_{x=0}^{x=\infty}\, dt\\
&=&- 2\dinrp  u(t,x) \prz_x u(t,x)\ \phi_t\sharp_\lz g(x)\ x^{2\lz}\ \psi(g)(t, x)\ \, dxdt\\
&& -\dinrp    u^2(t,x)\  \prz_x \big(\phi_t\sharp_\lz g(x)\big) \ x^{2\lz}\ \psi(g)(t, x)\ \, dxdt\\\\
&& -\dinrp   u^2(t,x)\ \phi_t\sharp_\lz g(x)\ 2\lz\,x^{2\lz -1}\ \psi(g)(t, x)\ \, dxdt\\\\
&&  +  \int_{\R_+} u^2(t,x)\ \phi_t\sharp_\lz g(x)\ x^{2\lambda}\psi(g)(t, x)\bigg|_{x=0}^{x=\infty}\, dt,
\end{eqnarray*}
where the third term on the right-hand side equals $- \mathcal{B}_{22}$. Hence,
\begin{eqnarray*}
{\rm B}_2&=&  - 2\dinrp  u(t,x) \prz_x u(t,x)\ \phi_t\sharp_\lz g(x)\ x^{2\lz}\ \psi(g)(t, x)\ \, dxdt\\
&& -\dinrp    u^2(t,x)\  \prz_x \big(\phi_t\sharp_\lz g(x)\big) \ x^{2\lz}\ \psi(g)(t, x)\ \, dxdt\\
&&  +  \int_{\R_+} u^2(t,x)\ \phi_t\sharp_\lz g(x)\ x^{2\lambda}\psi(g)(t, x)\bigg|_{x=0}^{x=\infty}\, dt,\\
&=:& {\rm B}_{21}+ {\rm B}_{22}+{\rm B}_{23}.
\end{eqnarray*}
For the term ${\rm B}_{23}$, we first note that
\begin{eqnarray*}
{\rm B}_{23} &=&  - \int_{\R_+} \lim_{x\to0} \Big( u^2(t,x)\ \phi_t\sharp_\lz g(x)\ x^{2\lambda}\psi_t \sharp_\lz g(x)\Big) \, dt\\
&&+ \int_{\R_+} \lim_{x\to\infty} \Big( u^2(t,x)\ \phi_t\sharp_\lz g(x)\ x^{2\lambda}\psi_t \sharp_\lz g(x)\Big) \, dt\\
&=:& {\rm B}_{231}+ {\rm B}_{232}.
\end{eqnarray*}
We claim that
\begin{eqnarray} \label{claim boundary}
{\rm B}_{231}={\rm B}_{232}=0.
\end{eqnarray}
We first consider ${\rm B}_{231}$.  Letting $x\to0+$ and applying H\"older's inequality, we see that
\begin{eqnarray*}
\lim_{x\to0}|u(t, x)|&\le & \lim_{x\to0}\lf[\inzf \plz (x, y)^2y^{2\lz}dy\r]^{\frac12}\|f\|_\ltz\\
&=&\|f\|_\ltz\frac{2\lz t}{\pi}\lf\{\inzf\lf[\dint_0^\pi \frac{(sin\theta)^{2\lz-1}}{(y^2+t^2)^{\lz+1}}\,d\theta\r]^2 y^{2\lz}dy\r\}^{\frac12}\\
&\ls&t\|f\|_\ltz\lf[\inzf\frac{y^{2\lz}}{(y^2+t^2)^{2\lz+2}}\,dy\r]^\frac12\\
&\ls& t^{-\lz-\frac12}\|f\|_\ltz,
\end{eqnarray*}
and
\begin{eqnarray*}
&&\lim_{x\to0}|\phi_t\sharp_\lz g(x)|\\
&&\quad\ls\lim_{x\to0}\lf|\inzf \dint_0^\pi\lf[t^{-2\lz-1}\phi\lf(\frac{\sqrt{x^2+y^2-2xycos\theta}}{t}\r)(sin\theta)^{2\lz-1}\,d\theta\r]^2 y^{2\lz}\,dy\r|^\frac12\|g\|_\ltz\\
&&\quad\sim\lf|\inzf \lf[t^{-2\lz-1}\phi\lf(\frac{y}{t}\r)\r]^2 y^{2\lz}\,dy\r|^\frac12\|g\|_\ltz\\
&&\quad\sim t^{-\lz-\frac12}\|\phi\|_\ltz\|g\|_\ltz.
\end{eqnarray*}
Moreover, by \eqref{comp supp psi}, ${\rm (K_{i})}$ and \eqref{defi of psi}, we see that
$$\|\psi\|_\loz\ls1\,\,{\rm and}\,\,|\psi(t, x, y)|\ls\frac{x}{t^{2\lz+2}}.$$
Thus, we have that
\begin{eqnarray*}
\lim_{x\to0}|\psi(g)(t, x)|&\le& \lim_{x\to0}\|\psi\|^\frac12_\loz\lf[\inzf|\psi(t, x,y)|g(y)|^2\dmz(y)\r]^{\frac12}\\
&\ls&\lim_{x\to0} x^\frac12\|g\|_\ltz=0.
\end{eqnarray*}
Therefore, we obtain that
$$\lim_{x\to0} \Big| u^2(t,x)\ \phi_t\sharp_\lz g(x)\ x^{2\lambda}\psi(g)(t, x)\Big|=0,$$
which gives that ${\rm B}_{231}=0$.  Next we verify the term ${\rm B}_{232}$.  Note that
$$\lf|\plz(x, y)\r|\ls \frac1{m_\lz(I(x, |x-y|+t))}\frac t{|x-y|+t}.$$
 Then by H\"older's inequality, we have that
\begin{align*}
|u(t,x)|^2&\ls \|f\|_\ltz^2\sum_{k=0}^\fz 2^{-2k}\dint_{|x-y|<2^kt}\frac1{[m_\lz(I(x, 2^{k-1}t))]^2}\dmz(y)\\
&\ls\|f\|_\ltz^2\frac1{x^{2\lz}t}.
\end{align*}
Moreover, from \eqref{size of psi-2} and H\"older inequality,  we deduce that
\begin{eqnarray*}
|\psi(g)(t, x)|&\le&\|\psi\|^\frac12_\loz\lf[\inzf|\psi(t,x, y)||g(y)|^2\dmz(y)\r]^{\frac12}\ls x^{-\lz}t^{-1/2}\|g\|_\ltz.
\end{eqnarray*}
By these and the fact that
\begin{eqnarray*}
|\phi_t\sharp_\lz g(x)|&\le&\|\phi\|_\ltz\|g\|_\ltz,
\end{eqnarray*}
we obtain that
$$ \lim_{x\to\infty} \Big| u^2(t,x)\ \phi_t\sharp_\lz g(x)\ x^{2\lambda}\psi(g)(t, x)\Big| =0,$$
which implies that
$${\rm B}_{232}=0.$$
Hence, the claim \eqref{claim boundary} holds.



Similar to the estimate for the term ${\rm A}$, as for the term ${\rm B}_{21}$, using H\"older's inequality and Cauchy's inequality, we obtain that
\begin{eqnarray}\label{esti B21 merryfield lem}
{\rm B}_{21}&\leq& {1\over 8} \dinrp   |\gratx u(t,x)|^2  |\phi_t\sharp_\lz g(x)|^2\ t \dmz(x)dt\\
&&\quad +C\dinrp   | u(t,x)|^2  |\psi(g)(t, x)|^2\ {\dmz(x)dt\over t}\noz\\
&=& {{\rm I}\over 8}
 +C\dinrp  | u(t,x)|^2  |\psi(g)(t, x)|^2\ {\dmz(x)dt\over t}.\noz
\end{eqnarray}
Again, for the term ${\rm B}_{22}$ we have
\begin{eqnarray}\label{esti B22 merryfield lem}
{\rm B}_{22}&\ls&  \dinrp  |u(t,x)|^2  \Big| t\prz_x\big(\phi_t\sharp_\lz g(x)\big)\Big|^2\  {\dmz(x)dt\over t} +\dinrp  | u(t,x)|^2  |\psi(g)(t, x)|^2\ {\dmz(x)dt\over t}.\noz
\end{eqnarray}

Combining the estimates of ${\rm A}$ and ${\rm B}$, \eqref{esti A merryfield lem}, \eqref{esti B1 merryfield lem}, \eqref{esti B21 merryfield lem},
 \eqref{esti B22 merryfield lem} and \eqref{claim boundary}, and
 by moving the two terms ${{\rm I}\over 8}$ to the left-hand side,  we obtain that
\begin{eqnarray*}
{\rm I}\ls&&  \int_{\R_+} f(x)^2g(x)^2 \dmz(x) +\dinrp |u(t,x)|^2|t\partial_x (\phi_t\sharp_\lz g(x))|^2 {\dmz(x)\,dt \over t}\\
&&\quad+\dinrp |u(t,x)|^2|t\partial_t (\phi_t\sharp_\lz g(x))|^2 {\dmz(x)\,dt \over t}+\dinrp  | u(t,x)|^2  |\psi( g)(t,x)|^2\ {\dmz(x)dt\over t}.
\end{eqnarray*}
We now define
$$ Q_t( g)(x) := \big( t \partial_t (\phi_t\sharp_\lz g(x)),\ t \partial_x (\phi_t\sharp_\lz g(x)),\psi( g)(t,x)\big).  $$
Then we have
\begin{eqnarray*}
{\rm I}&\ls&  \int_{\R_+} f(x)^2g(x)^2 \dmz(x)
+\dinrp |u(t,x)|^2 |Q_t( g)(x)|^2 {\dmz(x)\,dt \over t}.
\end{eqnarray*}
This finishes the proof of Lemma \ref{lem:Merryfield lem}.
\end{proof}

Next we have the following result for the product case, which follows from the iteration of Lemma \ref{lem:Merryfield lem}.
Before stating our next Lemma, we introduce the notation $\phi_{t_1}\sharp_{\lz,\,1} g(x_1, x_2)$, $\phi_{t_2}\sharp_{\lz,\,2} g(x_1, x_2)$ and  $\phi_{t_1}\phi_{t_2}\sharp_{\lz,\,1,\,2} g(x_1, x_2)$
to denote the convolution with respect to the first, second and both variables, respectively.

\begin{lem}\label{lem: product Merryfield lem}
Let $u(t_1, t_2, x_1, x_2):=\plzo\plzt f(x_1, x_2)$ and $\phi$ be a smooth function as in Lemma \ref{lem: constru of func psi}. 
Then for $f,\,g\in L^2(\rlz)$, there exists a positive constant $C$  such that
\begin{eqnarray*}
{\rm \tilde I}&:=&\dinrp\dinrp\lf|\gratxo\gratxt u(t_1, t_2, x_1, x_2)\r|^2 \\
&&\quad\quad\times\lf|\phi_{t_1}\phi_{t_2}\sharp_{\lz,\,1,\,2} g(x_1, x_2)\r|^2 t_1 t_2\dmzdt\\
&\le& C\bigg\{\dinrp[f(x_1, x_2)]^2 [g(x_1, x_2)]^2\dmzd\\
&&\quad+\inrp\dinrp\lf|\plzt f(x_1, x_2)\r|^2\lf[ Q^{(2)}_{t_2}(g)(x_1,x_2) \r]^2\dlmxtt\dmzo\\
&&\quad+\inrp\dinrp \lf|\plzo f(x_1, x_2)\r|^2 \lf[ Q^{(1)}_{t_1}(g)(x_1,x_2) \r]^2\,\dlmxto\dmzt\\
&&\quad+\dinrp\dinrp | u(t_1,t_2,x_1, x_2)|^2  \Big| Q^{(1)}_{t_1} Q^{(2)}_{t_2}(g)(x_1,x_2) \Big|^2 {\dmzdt\over t_1t_2}\bigg\}.
\end{eqnarray*}

Here the operator $Q^{(1)}_{t_1}$ is defined as
$$Q^{(1)}_{t_1}(g)(x_1,x_2) := \big( t_1 \partial_{t_1} (\phi_{t_1}\sharp_{\lz,\,1} g)(x_1,x_2),\ t_1 \partial_{x_1} (\phi_{t_1}\sharp_{\lz,\,1} g)(x_1,x_2),  \psi( g(\cdot,x_2) )(t_1,x_1)\big),  $$
where $$\psi( g(\cdot,x_2) )(t_1,x_1) = \int_0^\infty \psi(t_1,x_1,y_1)g(y_1,x_2) \dmz(y_1)$$ is obtained from Lemma  \ref{lem: constru of func psi}.  The definition of $Q^{(2)}_{t_2}(g)(x_1,x_2) $ is similar.
\end{lem}

\begin{proof}
By Lemma \ref{lem:Merryfield lem} for $t_1$ and $x_1$ and the conservation property of Poisson semigroup, we have that
\begin{eqnarray*}
{\rm \tilde{I}}&&=\dinrp\dinrp\lf|\gratxo \plzo\Big(\gratxt \plzt f(\cdot, t_2, \cdot, x_2)  \Big) (t_1,x_1)\r|^2\\
&&\quad\times\lf|  \phi_{t_1}  \sharp_{\lz,\,1} \Big(\phi_{t_2}\sharp_{\lz,\,2} g(\cdot, x_2) \Big)(x_1)  \r|^2 t_1 \ t_2 \dmzdt\\
&&\ls \dinrp\inrp\lf|\gratxt\lf(\plzt  f\r)(x_1, x_2)\r|^2\lf|\lf(\phi_{t_2}\sharp_{\lz,\,2}  g\r)(x_1, x_2)\r|^2\dmzo\ t_2\dmzt\,dt_2\\
&&\quad+\dinrp\dinrp\lf|\gratxt\lf(\plzo\plzt  f\r)(x_1, x_2)\r|^2\\
&&\hskip2cm\times \lf[ Q^{(1)}_{t_1} ( \phi_{t_2}\sharp_{\lz,\,2} g(\cdot, x_2) )(x_1) \r]^2\dlmxto\,t_2\dmzt\,dt_2\\
&&= \inrp\ \dinrp\lf|\gratxt\lf(\plzt  f(x_1, \cdot)\r)(x_2)\r|^2\lf|\lf(\phi_{t_2}\sharp_{\lz,\,2}  g(x_1, \cdot)\r)(x_2)\r|^2\ t_2\dmzt\,dt_2\ \dmzo\\
&&\quad+\dinrp\dinrp\lf|\gratxt \plzt \lf(\plzo  f(x_1, \cdot) \r)(x_1, x_2)\r|^2\\
&&\hskip2cm\times \lf[ \phi_{t_2}\sharp_{\lz,\,2}\big( Q^{(1)}_{t_1} (  g(\cdot, \cdot) )(x_1)\big)(x_2) \r]^2\dlmxto\,t_2\dmzt\,dt_2\\
&&=:{\rm {\tilde I_1}+{\tilde I_2}},
\end{eqnarray*}
where in the last but second equality, we use the fact that the order of $\sharp_{\lz,\,2}$ and $Q^{(1)}_{t_1}$ can be changed since they are acting on different variables and $Q^{(1)}_{t_1}$ is a linear integral operator.

Now we apply Lemma \ref{lem:Merryfield lem} to ${\rm {\tilde I_1}}$ and see that
\begin{eqnarray*}
{\rm {\tilde I_1}}&&\ls \dinrp |f(x_1, x_2)|^2 |g(x_1, x_2)|^2\dmzo\dmzt\\
&&\quad+\inrp\dinrp\lf|\plzt f(x_1, x_2)\r|^2\lf[ Q^{(2)}_{t_2}(g)(x_1,x_2)  \r]^2\dlmxtt\dmzo.
\end{eqnarray*}
Similarly, another application of Lemma \ref{lem:Merryfield lem} yields that
\begin{eqnarray*}
{\rm {\tilde I_2}}
&&\ls \inrp\dinrp \lf|\plzo f(x_1, x_2)\r|^2 \lf[ Q^{(1)}_{t_1}(g)(x_1,x_2)  \r]^2\dmzt\,\dlmxto\\
&&\quad+\dinrp\dinrp | u(t_1,t_2,x_1, x_2)|^2  \Big| Q^{(1)}_{t_1} Q^{(2)}_{t_2}(g)(x_1,x_2) \Big|^2 {\dmzdt\over t_1t_2}.
\end{eqnarray*}
This finishes the proof of Lemma \ref{lem: product Merryfield lem}.
\end{proof}

\begin{proof}[\bf Proof of $\|f\|_{H^p_{S_u}( \rlz)} \leq C
\|f\|_{H^p_{\cn_P}(\rlz)}$]

\hskip.2cm

For any $\alpha>0$ and  $f\in\ltzd$ satisfying
$\cn_P(f) \in\lpzd $, 
we define
\begin{eqnarray*}
{\mathcal A}(\alpha):=\lf\{(x,y)\in \R_+\times\R_+:\
\cms\big(\chi_{\cn_P(f)>\alpha\}}\big)(x,y)<\frac1{200}\frac1{2^{4\lz+2}}\r\},
\end{eqnarray*}
where $\cms$ is the strong maximal function defined in \eqref{stro max func-defn}.
We first claim that
%
\begin{eqnarray}\label{step1}
&&\iint_{{\mathcal A}(\az)}S_u^2(f)(x_1,x_2)\dmzd\\
&&\quad\leq\iiiint_{R^{*}}\big| t_1t_2\nabla_{t_1,\,y_1}\nabla_{t_2,\,y_2} u(t_1,t_2,y_1,y_2) \big|^2\frac{d\mu_\lz(y_1,y_2)\,dt_1dt_2}{t_1t_2},\noz
\end{eqnarray}
where for $t_1,\,t_2,\,y_1,\,y_2\in\R_+$, $R(y_1, y_2, t_1, t_2):=I(y_1,t_1)\times I(y_2,t_2)$ and
$$R^{*}:=\Big\{(y_1,y_2,t_1,t_2):\ \dfrac{\mu_\lz(\{ \cn_P(f)>\alpha\}\cap
R(y_1,y_2,t_1,t_2))} {\ \mu_\lz(R(y_1,y_2,t_1,t_2))}<\frac1{200}\frac1{2^{4\lz+2}}\Big\}.$$

Indeed,  observe that
\begin{eqnarray*}
\lefteqn{\iint_{{\mathcal A}(\az)}S_u^2(f)(x_1,x_2)\dmzd}\\
&& \quad\leq\iint_{\R_+\times \R_+}\iint_{\R_+\times \R_+}\iint_{{\mathcal A}(\az)\cap R(y_1, y_2, t_1, t_2)}\\
&&\hskip1cm\big| \nabla_{t_1,\,y_1}\nabla_{t_2,\,y_2} u(t_1,y_1,t_2,y_2) \big|^2
t_1t_2\frac{\dmzd\ d\mu_\lambda(y_1,y_2) dt_1dt_2}{m_\lz(I(x_1,t_1))m_\lz(I(x_2,t_2))}.
\end{eqnarray*}
For any $(y_1,y_2,t_1,t_2)\in \R_+\times\R_+\times\R_+\times\R_+$ with $R(y_1,y_2,t_1,t_2) \bigcap
{\mathcal A}(\az)\not=\emptyset$,
there exists some $(x_1,x_2)\in R(y_1,y_2,t_1,t_2) \bigcap
{\mathcal A}(\az)$ such that
\begin{eqnarray*}
\cms\big(\chi_{\cn_P(f)>\alpha\}}\big)(x_1,x_2)<\frac1{200}\frac1{2^{4\lz+2}}.
\end{eqnarray*}
Hence we have
\begin{eqnarray*}
\dfrac{\mu_\lz\lf(\{ \cn_P(f)>\alpha\}\bigcap R(y_1,y_2,t_1,t_2)\r)}{\mu_\lz(R(y_1,y_2,t_1,t_2))}<\frac1{200}\frac1{2^{4\lz+2}}.
\end{eqnarray*}
Then by the fact that for any $y_1\in I(x_1, t_1)$ and $y_2\in I(x_2, t_2)$,
$$m_\lz(I(x_1,t_1))\sim m_\lz(I(y_1,t_1))\quad {\rm and}\quad m_\lz(I(x_2,t_2))\sim m_\lz(I(y_2,t_2)),$$
we have \eqref{step1} and the claim holds.
%

%
%
%
%
%
%
%
%
%
%

Let $ g(x,y):=\chi_{\{ \cn_P(f)\leq \alpha\}}(x,y)$ and $\phi\in C^\fz_c(\R_+)$ such that
$\supp(\phi)\subset (0, 1)$, $\phi\equiv 1$ on $(0,1/2]$ and $0\le \phi(x)\le 1$ for all $x\in\R_+$.
Then for
$(x_1,x_2,t_1,t_2)\in R^{*}$, we have
\begin{align}\label{step2}
 &\phi_{t_1}\phi_{t_2}\sharp_{\lz,\,1,\,2} g(x_1, x_2)\\
 &\quad=\iint_{\R_+\times\R_+} \tau_{x_1}^{[\lambda]}\tau_{x_2}^{[\lambda]}(\phi_{t_1}\phi_{t_2})(y_1,y_2)\chi_{\{ \cn_P(f)\leq
\alpha\}}(y_1,y_2) dm_\lz(y_1)dm_\lz(y_2) \nonumber\\
 &\quad=\iint_{  \{ \cn_P(f)\leq
\alpha\} \cap R(x_1,\,x_2,\,t_1,\,t_2) } \tau_{x_1}^{[\lambda]}\tau_{x_2}^{[\lambda]}(\phi_{t_1}\phi_{t_2})(y_1,y_2)  dm_\lz(y_1)dm_\lz(y_2) \nonumber\\
&\quad\ge \iint_{  \{ \cn_P(f)\leq
\alpha\} \cap R(x_1,\,x_2,\,t_1/2,\,t_2/2) }  dm_\lz(y_1)dm_\lz(y_2) \gs1,\nonumber
\end{align}
where the last inequality follows from the fact that
\begin{eqnarray*}
&& \mu_\lz(\{(y_1,y_2)\in\R_+\times \R_+:\,\, \cn_P(f)(y_1,y_2)\leq
\alpha\} \cap R(x_1,x_2,t_1/2,t_2/2) ) \\
&&\quad\ge \mu_\lz(\{(y_1,y_2)\in\R_+\times \R_+:\,\, \cn_P(f)(y_1,y_2)\leq
\alpha\} \cap R(x_1,x_2,t_1,t_2))\\
&&\quad\quad- \mu_\lz(R(x_1,x_2,t_1,t_2)\setminus R(x_1,x_2,t_1/2,t_2/2))\\
&&\quad\ge \mu_\lz (R(x_1,x_2,t_1/2,t_2/2))- \frac1{200}\frac1{2^{4\lz+2}}\mu_\lz( R(x_1,x_2,t_1,t_2) ) \\
&& \quad\gs \mu_\lz( R(x_1,x_2,t_1,t_2) ).
\end{eqnarray*}


Combining (\ref{step1}) and (\ref{step2}), and then using Lemma \ref{lem: product Merryfield lem}, we have
\begin{eqnarray*}
&&\iint_{{\mathcal A}(\az)}S_u^2(f)(x_1,x_2)\dmzd\\
&&\quad \ls \iiiint_{R^{*}}\big| t_1t_2\nabla_{t_1,\,y_1}\nabla_{t_2,\,y_2} u(t_1,t_2,y_1,y_2) \big|^2
\lf| \phi_{t_1}\phi_{t_2}\sharp_{\lz,\,1,\,2} g(y_1, y_2) \r|^2\frac{d\mu_\lz(y_1,y_2)\,dt_1dt_2}{t_1t_2}\nonumber\\
&&\quad\ls\bigg\{\dinrp[f(x_1, x_2)]^2 [g(x_1, x_2)]^2\dmzd\\
&&\quad\quad+\inrp\dinrp\lf|\plzt f(x_1, x_2)\r|^2\lf| Q^{(2)}_{t_2}(g)(x_1,x_2) \r|^2\dlmxtt\dmzo\\
&&\quad\quad+\inrp\dinrp \lf|\plzo f(x_1, x_2)\r|^2 \lf| Q^{(1)}_{t_1}(g)(x_1,x_2) \r|^2\,\dlmxto\dmzt\\
&&\quad\quad+\dinrp\dinrp | u(t_1,t_2,x_1, x_2)|^2  \Big| Q^{(1)}_{t_1} Q^{(2)}_{t_2}(g)(x_1,x_2) \Big|^2 {\dmzdt\over t_1t_2}\bigg\}\\
&&\quad=: {\rm I+II+III+IV}.\nonumber
\end{eqnarray*}

For the term ${\rm I}$, we have
\begin{eqnarray*}
{\rm I}&=&\iint_{\{\cn_P(f)\leq \alpha\}}f^2(y_1, y_2)d\mu_\lz(y_1,y_2)\leq\iint_{\{\cn_P(f)\leq \alpha\}}|\cn_P(f)(y_1, y_2)|^2d\mu_\lz(y_1,y_2).
\end{eqnarray*}

For the term ${\rm II}$, we claim that: if $|Q^{(2)}_{t_2}(g)(x_1,x_2)|\not=0$, then there exists some $w_2$
such that $(x_1,w_2)\in \{\cn_P(f)\leq \alpha\}$, and satisfies
$|x_2-w_2|< t_2$.
To see this, recall that
$$Q^{(2)}_{t_2}(g)(x_1,x_2) := \big( t_2 \partial_{t_2} (\phi_{t_2}\sharp_{\lz,\,2} g(x_1,x_2)),\ t_2 \partial_{x_2} (\phi_{t_2}\sharp_{\lz,\,2} g(x_1,x_2)),  \psi( g(x_1, \cdot) )(t_2,x_2)\big),  $$
where $$\psi( g(x_1, \cdot) )(t_2,x_2) = \int_0^\infty \psi(t_2,x_2,y_2)g(x_1,y_2) \dmz(y_2)$$ is obtained from Lemma  \ref{lem: constru of func psi}. Hence, if $|Q^{(2)}_{t_2}(g)(x_1,x_2)|\not=0$, then we have that one of the three terms $\partial_{t_2} (\phi_{t_2}\sharp_{\lz,\,2} g)(x_1,x_2)$, $\partial_{x_2} (\phi_{t_2}\sharp_{\lz,\,2} g)(x_1,x_2)$, $\psi( g(x_1, \cdot) )(t_2,x_2)$ must be non-zero. Hence, there must be some $w_2$
such that $(x_1,w_2)$ is in the support of the function $g$, and satisfies
$|x_2-w_2|< t_2$. This implies that the claim holds.

Then we get that
$|\plzt f(x_1, x_2)|\leq \alpha$. As a consequence,
\begin{eqnarray}\label{term II}
{\rm II}&\le& \alpha^2 \inrp\dinrp \lf| Q^{(2)}_{t_2}(g)(x_1,x_2) \r|^2\dlmxtt\dmzo \\
&=& \alpha^2 \inrp\dinrp \lf| Q^{(2)}_{t_2}(1-g)(x_1,x_2) \r|^2\dlmxtt\dmzo \nonumber\\
&\ls& \alpha^2 \dinrp |1-
g(x_1,x_2)|^2 d\mu_\lz(x_1,x_2) \nonumber\\
&\ls &  \alpha^2 \mu_\lz(\{(x_1,x_2)\in\R_+\times \R_+:\,\,\cn_P(f)(x_1,x_2)>\alpha\}), \nonumber
\end{eqnarray}
where the second inequality follows from the $L^2$ boundedness of the Littlewood--Paley square function estimate, i.e.,  \eqref{Littlewood--Paley} of Theorem \ref{thm product Littlewood--Paley}.  For the term ${\rm III}$, symmetrically, we can obtain the same estimate as term ${\rm II}$.

For the term ${\rm IV}$, if $Q^{(1)}_{t_1} Q^{(2)}_{t_2}(g)(x_1,x_2)\not=0 $, then  there exist some $(w_1,w_2)$ such that $(w_1,w_2)\in
\{\cn_P(f)\leq \alpha\}$ and $|x_1-w_1|<t_1$ and $|x_2-w_2|<t_2$. Hence
$|u(t_1,t_2,x_1, x_2)|\leq \alpha$. Following the same routine of
(\ref{term II}), and using the $L^2$ boundedness of the product Littlewood--Paley square function estimate, i.e.,  \eqref{product Littlewood--Paley} of Theorem \ref{thm product Littlewood--Paley}, we have
\begin{eqnarray}\label{term IV}
{\rm IV}\ls \alpha^2\mu_\lz\lf(\lf\{(x_1,x_2)\in\R_+\times \R_+:\,\,\cn_P(f)(x_1,x_2)>\alpha\r\}\r). \nonumber
\end{eqnarray}

Combining the four terms above, we have
\begin{eqnarray}\label{part I}
&&\iint_{\mathcal A(\alpha)}S_u^2(f)(x_1,x_2)\,d\mu_\lz(x_1,x_2)\\
&&\quad\ls \alpha^2 \mu_\lz(\{(x_1,x_2)\in\R_+\times \R_+:\,\,\cn_P(f)(x_1,x_2)>\alpha\})\noz\\
&&\quad\quad+ \iint_{\{\cn_P(f)\leq
\alpha\}}|\cn_P(f)(x_1,\,x_2)|^2\dmzd.\noz
\end{eqnarray}

By the $L^2(\rlz)$-boundedness of the strong maximal function $\cm_S$, we have
\begin{align} \label{part II}
\mu_\lz\big(\R_+\times\R_+\setminus \mathcal A (\alpha)\big)
&\ls \iint_{\R_+\times\R_+} \lf[\cm_S(\chi_{\{\cn_P(f)>\alpha\}})(x_1,x_2)\r]^2  \,d\mu_\lz(x_1,x_2)\\
&\ls \iint_{\R_+\times\R_+}\lf[\chi_{\{\cn_P(f)>\alpha\}}(x_1,x_2)\r]^2 \,d\mu_\lz(x_1,x_2)\nonumber\\
&= \mu_\lz\big(\{(x_1,\,x_2)\in\R_+\times \R_+:\,\,\cn_P(f)(x_1,\,x_2)>\alpha\}\big).\nonumber
\end{align}

Combining (\ref{part I}) and (\ref{part II}), we have
\begin{eqnarray*}
&&\mu_\lz\Big(\{(x_1,x_2)\in\R_+\times \R_+:\ S_u(f)(x_1,x_2)>\alpha\}\Big)\\
&& \quad\leq \mu_\lz\Big( \{(x_1,x_2)\in\R_+\times \R_+:\
S_u(f)(x_1,x_2)>\alpha\} \cap  \mathcal A (\alpha)\Big) \\
&&\quad\quad+ \mu_\lz\Big( \{(x_1,x_2)\in\R_+\times \R_+:\
S_u(f)(x_1,x_2)>\alpha\} \setminus  \mathcal A (\alpha)\Big)\nonumber\\
& &\quad\ls  \mu_\lz\big(\{(x_1,\,x_2)\in\R_+\times \R_+:\,\,\cn_P(f)(x_1,\,x_2)>\alpha\}\big)\\
&&\quad\quad+ \frac1{\az^2}\iint_{\{\cn_P(f)\leq
\alpha\}}|\cn_P(f)(x_1,x_2)|^2\,\dmzd,
\end{eqnarray*}
which via a standard argument shows that $\|S_u(f)\|_\lpzd \ls \|\cn_P(f)\|_\lpzd$.
\end{proof}

\bigskip

{\bf Step 3: \ $\|f\|_{H^p_{\cn_P}( \rlz)}\ls
\|f\|_{H^p_{\mathcal R_P}(\rlz)}$} for $f\in H^p_{\mathcal R_P}(\rlz) \cap L^2(\rlz)$.

We now define the product grand maximal functions, borrowing an idea from \cite{YZ} in the one-parameter setting (see also \cite{gly1,gly2}).
\begin{defn}
Let $\beta_1,\beta_2,\gamma_1,\gamma_2\in(0,1]$. For any $f\in \big(\GGpone(\beta_1,\beta_2;\gamma_1,\gamma_2)\big)'$, we define the product grand maximal function as follows: For $(x_1,x_2)\in\R_+\times\R_+$,
\begin{align}
G_{\beta_1,\,\beta_2,\,\gamma_1,\,\gamma_2}(f)(x_1,x_2):= \sup\big\{ \langle f,\vz_1\vz_2\rangle: \ \|\vz_i\|_{\G(x_i,\,r_i,\,\beta_i,\,\gamma_i)}\leq 1, r_i>0,\ i=1,2 \big\}.
\end{align}
\end{defn}

By the definition of $\cn_{P}f$, we have
\begin{align*}
\cn_{P}f(x_1, x_2)=\sup_{\gfz{|y_1-x_1|< t_1}{|y_2-x_2|< t_2}}\!\lf| \int_{\R_+}\int_{\R_+} P_{t_1}^{[\lz]}(y_1, z_1)P_{t_2}^{[\lz]}(y_2, z_2)  f(z_1, z_2) d\mu_\lz(z_1,z_2)\r|.
\end{align*}
Next, for $i=1,2$, for any fixed $y_i,t_i\in\R_+$, the Poisson kernel $P_{t_i}^{[\lz]}(y_i, z_i)$, as a function of $z_i$, satisfies the conditions (${\rm K_{i}}$) and (${\rm K_{ii}}$) as in Section 2.3 (see \cite{yy}), and hence, it
is a test function of the type $(y_i,\,t_i,\,1,\,1)$, with  the norm $\big\|P_{t_i}^{[\lz]}(y_i, \cdot)\big\|_{\G(y_i,\,t_i,\,1,\,1)}=: C_\lz$, where $C_\lz$ is a positive constant depending only on $\lz$ (see Definition \ref{def-of-test-func-space} for the test function and its norm).
Hence, it is a test function of the type
$(y_i,\,t_i,\,\beta_i,\,\gamma_i)$ for every $\beta_i,\gamma_i\in(0,1]$ with the norm $ C_\lz$. 
Moreover, for any $x_i$ with $|x_i-y_i|<t_i$, we have that $\big\|P_{t_i}^{[\lz]}(y_i, \cdot)\big\|_{\G(x_i,\,t_i,\,\beta_i,\,\gamma_i)}\ls C_\lz$, where the implicit constant is independent of $x_i,\,t_i,\,\beta_i$ and $\gamma_i$.

Then, there exists a positive constant $\wz C_\lz$ such that
$$ \sup_{|y_1-x_1|< t_1} \lf\|P_{t_1}^{[\lz]}(y_1, \cdot)\r\|_{\G(x_1,\,t_1,\,\beta_1,\,\gamma_1)}=
\sup_{|y_2-x_2|< t_2}\lf\|P_{t_2}^{[\lz]}(y_2, \cdot)\r\|_{\G(x_2,\,t_2,\,\beta_2,\,\gamma_2)}=\wz C_\lz.$$
We then obtain that
\begin{align*}
\cn_{P}f(x_1, x_2)\ls G_{\beta_1,\,\beta_2,\,\gamma_1,\,\gamma_2}(f)(x_1,x_2).
\end{align*}


Next we claim that
\begin{align}\label{moser}
G_{\beta_1,\,\beta_2,\,\gamma_1,\,\gamma_2}(f)(x_1,x_2)\ls\big\{ \mathcal{M}_1\mathcal{M}_2( |\mathcal{R}_P(f)|^r )(x_1,x_2) \big\}^{1\over r}
\end{align}
for  any $r\in({2\lz+1\over 2\lz+2}, p)$ and $f\in H^p_{\mathcal R_P}(\rlz) \cap L^2(\rlz)$, where $\cm_1$ and $\cm_2$ are as in Section \ref{s2}.
This implies our Step 3.

To prove \eqref{moser}, we first prove the following inequality:
\begin{align}\label{moser 1}
 |\langle f,\psi_1\psi_2\rangle| &\ls \Bigg[ \mathcal{M}_1\bigg(  \mathcal{M}_2\Big(   \lf|\mathcal{R}_P(f)\r|^r  \Big)  \bigg)(x_1,x_2) \Bigg]^{1\over r}
 \end{align}
for any $r\in( {2\lz+1\over 2\lz+2}, p)$,  $f\in H^p_{\mathcal R_P}(\rlz) \cap L^2(\rlz)$,
 $\psi_1\in \GG(\beta_1,\gamma_1)$ with $\|\psi_1\|_{\G(x_1,\,2^{-\wz k_1},\,\beta_1,\,\gamma_1)}\leq 1$
 and $\psi_2\in \GG(\beta_2,\gamma_2)$ with $\|\psi_2\|_{\G(x_2,\,2^{-\wz k_2},\,\beta_2,\,\gamma_2)}\leq 1$.

To see this,  consider the following approximations to the identity: For each $k\in\mathbb Z$, define the operator
\begin{align}\label{Pk}
P_k:= P^{[\lz]}_{2^{-k}}
\end{align}
with the kernel
$
P_k(x,y):= P^{[\lz]}_{2^{-k}}(x,y).
$ 
Then, it is easy to see that
$$\lim_{k\to\infty} P_k = \lim_{k\to\infty}P^{[\lz]}_{2^{-k}}= Id\quad {\rm and }\quad \lim_{k\to-\infty} P_k = \lim_{k\to-\infty}P^{[\lz]}_{2^{-k}}= 0$$
in the sense of $L^2(\R_+,\dmz)$. Moreover, based on size and smoothness conditions of the Poisson kernel $P^{[\lz]}_t(x,y)$, it is direct that $P_k(x,y)$
satisfies the size and smoothness conditions as in $({\rm A_i})$, $({\rm A_{ii}})$ and $({\rm A_{iii}})$ in Definition \ref{def-ati} for $x,y$  with a certain positive constant $\overline C_\lz$.

Also, from ${\rm (S_{iii})}$ in Lemma \ref{l-seimgroup prop}, we have that for any $k\in\zz$ and $x\in\R_+$,
$$\int_{\R_+}P_k(x,y) \dmz(y)=\int_{\R_+}P_k(x,y) \dmz(x)=1.$$
Hence, $\{P_k\}_{k\in\mathbb Z}$ is an approximation to the identity as in Definition \ref{def-ati}. Then we set $Q_k:=P_k-P_{k-1}$ as the difference operator, and it is obvious that the kernel $Q_k(x,y)$ of $Q_k$ satisfies the same size and smoothness  conditions as $P_k(x,y)$ does, and
$$\int_{\R_+}Q_k(x,y) \dmz(y)=\int_{\R_+}Q_k(x,y) \dmz(x)=0.$$

Now to classify the action on different variables, for $i=1,2$, we let $\left\{P_{k_i}^{(i)}\right\}_{k_i\in\mathbb Z}$ be the approximation to the identity on the $i$th variable as defined above, and similarly let  $Q_{k_i}^{(i)}$ be the corresponding difference operator.

Then, following Theorem 2.9 in \cite{HLL2}, we now have the following Calder\'on's reproducing formula:
\begin{align}\label{reproducing Poisson}
f(x_1,x_2)&=   \sum_{k_1}\sum_{I_1 \in\mathscr{X}^{k_1+N_1}} \sum_{k_2}\sum_{I_2 \in\mathscr{X}^{k_2+N_2}}  m_\lambda(I_1)m_\lambda(I_2) \\
&\quad\quad\times\tilde{Q}^{(1)}_{k_1}(x_1,x_{I_1})\tilde{Q}^{(2)}_{k_2}(x_2,x_{I_2})Q^{(1)}_{k_1}Q^{(2)}_{k_2}(f)(x_{I_1},x_{I_2}),\noz
\end{align}
where the series converges in the sense of $ \big(\GGpone(\beta_1,\beta_2;\gamma_1,\gamma_2)\big)'$, and for $i:=1,2$, $\tilde{Q}^{(i)}_{k_i}$  satisfies the same size, smoothness and cancellation conditions as $Q^{(i)}_{k_i}$ does, $\mathscr{X}^{k}$ is as in Section \ref{s2},
and $x_{I_1}$, $x_{I_2}$ are arbitrary points in the dyadic intervals $I_1$ and $I_2$, respectively.

We now prove \eqref{moser 1}. To begin with, for any $f\in H^p_{\mathcal R_P}(\rlz) \cap L^2(\rlz)$ and  $\psi_1\in \GG(\beta_1,\gamma_1)$ and with $\|\psi_1\|_{\G(x_1,\,2^{-\wz k_1},\,\beta_1,\,\gamma_1)}\leq 1$ and $\psi_2\in \GG(\beta_2,\gamma_2)$ with $\|\psi_2\|_{\G(x_2,\,2^{-\wz k_1},\,\beta_2,\,\gamma_2)}\leq 1$, from \eqref{reproducing Poisson} we obtain that
\begin{align*}
\langle f, \psi_1\psi_2\rangle &=   \sum_{k_1}\sum_{I_1 \in\mathscr{X}^{k_1+N_1}} \sum_{k_2}\sum_{I_2 \in\mathscr{X}^{k_2+N_2}}  m_\lambda(I_1)m_\lambda(I_2) \\
&\quad\quad\times\tilde{Q}^{(1)}_{k_1}(\psi_1)(x_{I_1})\tilde{Q}^{(2)}_{k_2}(\psi_2)(x_{I_2})Q^{(1)}_{k_1}Q^{(2)}_{k_2}(f)(x_{I_1},x_{I_2}).\noz
\end{align*}
Next, recall again the following almost orthogonality estimate (see \eqref{eee2} in Section 2, and see also Lemma 2.11 in \cite{HLL2}):  For $\epsilon \in(0, 1)$,
\begin{align}\label{eeeeee2}
&\big|\tilde{Q}^{(1)}_{k_1}(\psi_1)(x_{I_1})\tilde{Q}^{(2)}_{k_2}(\psi_2)(x_{I_2})\big|\\
&\quad\ls \prod_{i=1}^2  2^{-|k_i-\wz k_i|\epsilon}\bigg(\frac {2^{-k_i}+2^{-\wz k_i}}{|x_i-x_{I_i}|+2^{-k_i}+2^{-\wz k_i}}\bigg)^\epsilon  \nonumber\\
&\quad\quad\times   \frac1{m_\lz(I(x_i, 2^{-k_i}+2^{-\wz k_i} ))+m_\lz(I(x_{I_i}, 2^{-k_i}+2^{-\wz k_i}))+m_\lz(I(x_i, |x_i-x_{I_i}|))}.\nonumber
\end{align}
For arbitrary dyadic intervals $I_1$ and $I_2$, we choose $x_{I_1}\in I_1$ and  $x_{I_2}\in I_2$ such that
$$ \lf|Q^{(1)}_{k_1}Q^{(2)}_{k_2}(g)(x_{I_1},x_{I_2})\r| \leq 2\inf_{z_1\in I_1,\, z_2\in I_2}\lf|Q^{(1)}_{k_1}Q^{(2)}_{k_2}(g)(z_{1},z_{2})\r|,$$
 which implies that
\begin{align*}
 \lf|Q^{(1)}_{k_1}Q^{(2)}_{k_2}(g)(x_{I_1},x_{I_2})\r| &\leq 2\inf_{z_1\in I_1,\, z_2\in I_2}\Big(\big|P^{(1)}_{k_1}P^{(2)}_{k_2}(g)(z_{1},z_{2})\big|+\big|P^{(1)}_{k_1-1}P^{(2)}_{k_2}(g)(z_{1},z_{2})\big|\\
& \quad\quad\quad\quad\quad\quad\quad\quad+\big|P^{(1)}_{k_1}P^{(2)}_{k_2-1}(g)(z_{1},z_{2})\big|+\big|P^{(1)}_{k_1-1}P^{(2)}_{k_2-1}(g)(z_{1},z_{2})\big|\Big)\\
&\leq 8\inf_{z_1\in I_1,\, z_2\in I_2}\mathcal{R}_P(f)(z_1,z_2).
\end{align*}

 Then, based on the estimates in the proof of Theorem 2.10 in \cite[pp.\, 335--336]{HLL2}, see also the estimates we had in Section 2.3 for  $\mathbb L$ in \eqref{eee1}, we have the following estimate:
\begin{align*}
|\langle f, \psi\rangle|
& \ls \sum_{k_1} \sum_{k_2} 2^{-|k_1-\wz k_1|\epsilon}2^{-|k_2-\wz k_2|\epsilon} 2^{[(k_1\wedge \wz k_1)-k_1](2\lz+1)(1-{1\over r})}2^{[(k_2\wedge \wz k_2)-k_2](2\lz+1)(1-{1\over r})}\\
&\quad
 \times\Bigg[ \mathcal{M}_1\bigg( \sum_{I_1 \in\mathscr{X}^{k_1+N_1}}  \mathcal{M}_2\Big( \sum_{I_2 \in\mathscr{X}^{k_2+N_2}}   \inf_{\gfz{z_1\in I_1}{z_2\in I_2}}\lf|\mathcal{R}_P(f)(z_1,z_2)\r|^r  \chi_{I_2}(\cdot)\Big)(x_2) \chi_{I_1}(\cdot) \bigg)(x_1) \Bigg]^{1\over r}\\
 &\ls
 \Bigg[ \mathcal{M}_1\bigg(  \mathcal{M}_2\Big(   \lf|\mathcal{R}_P(f)\r|^r  \Big)  \bigg)(x_1,x_2) \Bigg]^{1\over r},
\end{align*}
where ${2\lz+1\over2\lz+2}<r<p$ and $a\wedge b:= \min\{a,b\}$, which shows that  \eqref{moser 1} holds.

We now prove \eqref{moser}. For every $\vz:=\vz_1\vz_2$ with $\|\vz_i\|_{\G(x_i,\,t_i,\,\beta_i,\,\gamma_i)}\leq 1,\ i=1,2$,
let
$$\sigma_1:=\int_{\R_+}\vz_1(x_1)\dmz(x_1),\quad \sigma_2:=\int_{\R_+}\vz_2(x_2)\dmz(x_2).$$
It is obvious that $|\sigma_1|, |\sigma_2|\ls 1$ since $\vz_i\in\G(x_i,\,t_i,\,\beta_i,\,\gamma_i)$ for $i=1,2$.
We set
$$\psi_1(y_1):={1\over 1+\sigma_1\wz C_\lz}\lf[ \vz (y_1)- \sigma_1 P^{(1)}_{\wz k_1}(x_1,y_1) \r], \quad \psi_2(y_2):={1\over 1+\sigma_2\wz C_\lz}\lf[ \vz (y_2)- \sigma_2 P^{(2)}_{\wz k_2}(x_2,y_2) \r],$$
where $\wz k_i := \lfloor\log_2 t_i \rfloor+1$ for $i=1,2$. Then, we see that $\psi_1\in \G(x_1,\, 2^{-\wz k_1},\,\beta_1,\,\gamma_1)$ and $\psi_2\in \G(x_2,\, 2^{-\wz k_2},\,\beta_2,\,\gamma_2)$. Based on the normalisation
factor ${1\over 1+\sigma_1\wz C_\lz}$, we obtain that
$$ \|\psi_1\|_{\G(x_1,\, 2^{-\wz k_1},\,\beta_1,\,\gamma_1)}\leq1 \quad{\rm and}\quad \|\psi_2\|_{\G(x_2,\, 2^{-\wz k_2},\,\beta_2,\,\gamma_2)}\leq1.  $$

Moreover, we point out that
$$\int_{\R_+}\psi_1(y_1)\dmz(y_1)=0 \quad {\rm for\ all\ }\ x_1\in\R_+ $$
since $\int_{\R_+} P^{(1)}_{\wz k_1}(x_1,y_1) \dmz(y_1)=1 \quad {\rm for\ all\ } x_1\in\R_+$.  Similarly we have the cancellation property for $\psi_2(y_2)$. Hence, we further obtain that  $\psi_1\in \GG(\beta_1,\gamma_1)$ and $\psi_2\in \GG(\beta_2,\gamma_2)$.

Based on the definition of $\psi_1$ and $\psi_2$, we have
\begin{align*}
|\langle f, \vz\rangle|&=\bigg|\bigg\langle f, \lf[\sigma_1 P_{\wz k_1}(x_1,\cdot)+(1+\sigma_1\,\wz C_\lz)\psi_1(\cdot)\r]\lf[\sigma_2P_{\wz k_2}(x_2,\cdot)+(1+\sigma_2\,\wz C_\lz)\psi_2(\cdot)\r] \bigg\rangle\bigg|\\
&\leq \bigg|\bigg\langle f, \sigma_1\sigma_2P_{\wz k_1}(x_1,\cdot)P_{\wz k_2}(x_2,\cdot) \bigg\rangle\bigg|+\bigg|\bigg\langle f, (1+\sigma_1\,\wz C_\lz)\psi_1(\cdot)\sigma_2P_{\wz k_2}(x_2,\cdot) \bigg\rangle\bigg|\\
&\quad+\bigg|\bigg\langle f, \sigma_1P_{\wz k_1}(x_1,\cdot)(1+\sigma_2,\wz C_\lz)\psi_2(\cdot) \bigg\rangle\bigg|+\bigg|\bigg\langle f, (1+\sigma_1\,\wz C_\lz)\psi_1(\cdot)(1+\sigma_2\,\wz C_\lz)\psi_2(\cdot) \bigg\rangle\bigg|\\
&=:A_1+A_2+A_3+A_4.
\end{align*}
For the term $A_1$, from the definition of $\mathcal R_P(f)$ in Section 1, we get that
$$A_1\ls \mathcal R_P(f)(x_1,x_2) =\big\{ |\mathcal R_P(f)(x_1,x_2)|^r\big\}^{1\over r}\leq\Bigg[ \mathcal{M}_1\bigg(  \mathcal{M}_2\Big(   \lf|\mathcal{R}_P(f)\r|^r  \Big)  \bigg)(x_1,x_2) \Bigg]^{1\over r}$$
for any $r\in (0,1]$.
For the term $A_4$,
from \eqref{moser 1} we  obtain that
$$A_4\ls  \Bigg[ \mathcal{M}_1\bigg(  \mathcal{M}_2\Big(   \lf|\mathcal{R}_P(f)\r|^r  \Big)  \bigg)(x_1,x_2) \Bigg]^{1\over r}$$
for ${2\lz+1\over2\lz+2}<r<p$.  As for $A_2$, let $F_{x_2}(\cdot):=\langle f, P_{\wz k_2}(x_2,\cdot) \rangle$. Then we have
\begin{align*}
A_2&\sim\bigg|\Big\langle F_{x_2}(\cdot), (1+\sigma\,\wz C_\lz)\psi_1(\cdot) \Big\rangle\bigg|.
\end{align*}
Then, following the same approach above, by using the reproducing formula in terms of $Q^{(2)}_{k_2}$, the almost orthogonality estimate, we obtain that
\begin{align*}
 A_2 &\ls \Bigg[ \mathcal{M}_1\bigg(  \Big [\sup_{t_1>0}\lf|P_{t_1}^{[\lz]}(F_{x_2}(\cdot))\r|\Big]^r    \bigg)(x_1) \Bigg]^{1\over r}=
 \Bigg[ \mathcal{M}_1\bigg(  \Big [\sup_{t_1>0}\lf|P_{t_1}^{[\lz]}P_{2^{-\wz k_2}}^{[\lz]}(f)(\cdot,x_2)\r|\Big]^r    \bigg)(x_1) \Bigg]^{1\over r},
 \end{align*}
which is further bounded by
$$\Bigg[ \mathcal{M}_1\bigg(  \mathcal{M}_2\Big(   \lf|\mathcal{R}_P(f)\r|^r  \Big)  \bigg)(x_1,x_2) \Bigg]^{1\over r}.$$
Similarly, we obtain that $A_3$ satisfies the same estimates. Combining the estimates of $A_1$,  $A_2$, $A_3$ and $A_4$,
we obtain that \eqref{moser} holds.

\medskip
{\bf Step 4:}  $\|f\|_\hrp\ls \|f\|_\hrz$ for $f\in \hrz\cap L^2(\rlz)$.

\medskip
Indeed,  we recall the well-known subordination formula that for all $f\in L^2(\rlz)$,
\begin{equation*}
\plzo\plzt f(x_1, x_2)=\frac1{\pi}\dinzf
\frac {e^{-u_1}}{\sqrt{u_1}}\frac {e^{-u_2}}{\sqrt{u_2}}e^{-\frac{t_1^2}{4u_1} \Delta_\lz}e^{-\frac{t_2^2}{4u_2}\Delta_\lz}f(x_1, x_2)\,du_1du_2.
\end{equation*}
From this, it follows that
\begin{eqnarray*}
\crz_Pf(x_1, x_2)
&&\ls \supd\dinzf \frac {e^{-u_1}}{\sqrt{u_1}}\frac{e^{-u_2}}{\sqrt{u_2}}
\lf|e^{-\frac{t_1^2}{4u_1}\Delta_\lz}e^{-\frac{t_2^2}{4u_2}\Delta_\lz}f(x_1, x_2)\r|\,du_1du_2\\
&&\ls \crz_hf(x_1,x_2)\dinzf  \frac {e^{-u_1}}{\sqrt{u_1}}\frac {e^{-u_2}}{\sqrt{u_2}}\,du_1du_2\\
&&\ls \crz_hf(x_1,x_2),
\end{eqnarray*}
which further implies  that for all $f\in \hrz$,
$\|f\|_\hrp\ls \|f\|_\hrz.$

\smallskip
{\bf Step 5:} $\|f\|_\hrz\le \|f\|_\hnz$ for $f\in \hnz\cap L^2(\rlz)$.
\medskip

 Observe that for all $f\in \ltzd$, $\crz_h f\le \cn_h f$. Then we see that for all
$f\in \hnz$,
$$\|f\|_\hrz\le \|f\|_\hnz.$$

\smallskip
{\bf Step 6:}
$
\|f\|_\hnz\ls \|f\|_\hap$ 
  for $f\in \hap\cap L^2(\rlz)$.

We claim that it suffices to prove that there exists a constant $\delta>0$ such that for every $H^p$ rectangular atom $\alpha_R$ as in Definition \ref{def-of-p q atom}, and $\gz_1,\gz_2\ge 2$,
\begin{equation}\label{eqn:bdd of nontang maxi func on atom-1}
\int_{x_1\notin \gz_1 I}\inzf |\cn_h (\alpha_R)(x_1, x_2)|^p\,d\mu_\lz(x_1,x_2)\ls   [\mu_\lz(R)]^{1-\frac{p}{2}} \|\alpha_R\|_{L^2(\rlz)}^p \gz_1^{-p}
\end{equation}
and
\begin{equation}\label{eqn:bdd of nontang maxi func on atom-2}
\inzf\int_{x_2\notin \gz_2 J} |\cn_h (\alpha_R)(x_1, x_2)|^p\,d\mu_\lz(x_1,x_2)\ls  [\mu_\lz(R)]^{1-\frac{p}{2}} \|\alpha_R\|_{L^2(\rlz)}^p \gz_2^{-p}.
\end{equation}
In fact, if \eqref{eqn:bdd of nontang maxi func on atom-1} and \eqref{eqn:bdd of nontang maxi func on atom-2}
hold, then we can obtain that for every $H^p$ atom $a$, we have
\begin{equation}\label{eqn:bdd of nontang maxi func on atom}
\inzf\inzf |\cn_h(a)(x_1, x_2)|^p\,d\mu_\lz(x_1,x_2)\ls  1.
\end{equation}

To see this, suppose $a$ is supported in an open set $\Omega\subset \R_+\times \R_+$ with finite measure and
$a:=\sum_{R\in m(\Omega)}\az_R. $
We now define
\begin{align*}
&\widetilde{\Omega}:=\{ (x_1,x_2) \in  \R_+\times \R_+: \cms(\chi_\Omega)(x_1,x_2)>1/2 \}\\
&\widetilde{\widetilde{\Omega}}:=\{ (x_1,x_2) \in  \R_+\times \R_+: \cms(\chi_{\widetilde{\Omega}})(x_1,x_2)>1/2 \}.
\end{align*}
Moreover, for every $R:=I\times J \in m_1(\Omega)$ (see Section 2.1 for the definition of $m_1(\Omega)$), let $\tilde{I}$ be the largest dyadic interval containing $I$ such that $\tilde{R}:=\tilde{I}\times J\subset \widetilde{\Omega}$
and let $\tilde{J}$ be the largest dyadic interval containing $J$ such that $\tilde{\tilde{R}}:=\tilde{I}\times \tilde{J}\subset \widetilde{\widetilde{\Omega}}$. We now let $\gamma_1:={|\wz I|\over |I|}$ and $\gamma_2:={|\wz J|\over |J|}$.

Then we have
\begin{align*}
&\inzf\inzf |\cn_h(a)(x_1, x_2)|^p\,d\mu_\lz(x_1,x_2)\\
&\quad=\Bigg(\iint_{\cup_{R\in m(\Omega)} 10 \tilde{\tilde{R}} } + \iint_{ \big(\cup_{R\in m(\Omega)} 10 \tilde{\tilde{R}}\big)^c }\Bigg)  |\cn_h(a)(x_1, x_2)|^p\,d\mu_\lz(x_1,x_2)\\
&\quad:={\rm A_1}+{\rm A}_2.
\end{align*}
For the term ${\rm A}_1$, using H\"older's inequality and the $\ltzd$-boundedness of $\cn_h$, we have that
\begin{align*}
{\rm A}_1&\leq \Big[\mu_\lambda\big( \bigcup_{R\in m(\Omega)} 10 \tilde{\tilde{R}} \big)\Big]^{1-{p\over 2}} \Bigg(\iint_{\cup_{R\in m(\Omega)} 10 \tilde{\tilde{R}} }  |\cn_h(a)(x_1, x_2)|^2\,d\mu_\lz(x_1,x_2)\Bigg)^{p\over2}\\
&\ls \Big[\mu_\lambda\big( \Omega \big)\Big]^{1-{p\over 2}} \|a\|_{L^2(\rlz)}^p\ls \Big[\mu_\lambda\big( \Omega \big)\Big]^{1-{p\over 2}} \Big[ \mu_\lambda\big( \Omega \big)^{{1\over2}-{1\over p}}\Big]^p \ls 1.
\end{align*}
For the term ${\rm A}_2$, we have
\begin{align*}
{\rm A}_2&\leq  \sum_{R\in m(\Omega)} \iint_{ \big(10 \tilde{\tilde{R}}\big)^c }  |\cn_h(\az_R)(x_1, x_2)|^2\,d\mu_\lz(x_1,x_2)\\
&\leq \sum_{R\in m(\Omega)} \iint_{ (10\tilde{I})^c\times \R_+ }  |\cn_h(\az_R)(x_1, x_2)|^p\,d\mu_\lz(x_1,x_2)\\
&\quad+ \sum_{R\in m(\Omega)} \iint_{ \R_+\times (10\tilde{I})^c   }  |\cn_h(\az_R)(x_1, x_2)|^p\,d\mu_\lz(x_1,x_2)\\
&=: {\rm A}_{21}+{\rm A}_{22}.
\end{align*}
From \eqref{eqn:bdd of nontang maxi func on atom-1} and \eqref{eqn:bdd of nontang maxi func on atom-2}, we have
\begin{align*}
{\rm A}_{21}&\ls  \sum_{R\in m(\Omega)} \|\az_R\|_{L^2(\rlz)}^p\ \mu_\lambda(R)^{1-{p\over 2}} \left( { |I|\over |\tilde{I}|} \right)^{-p}
\end{align*}
and that
\begin{align*}
{\rm A}_{22}&\ls   \sum_{R\in m(\Omega)}\|\az_R\|_{L^2(\rlz)}^p\ \mu_\lz(R)^{1-{p\over 2}} \left( { |J|\over |\tilde{J}|} \right)^{-p}.
\end{align*}
As a consequence, using H\"older's inequality and Journ\'e's covering lemma (\cite{Jo}, \cite{P}, see also the version on spaces of homogeneous type in \cite{HLLin}) we get that
\begin{align*}
{\rm A}_2&\ls  \bigg(\sum_{R\in m(\Omega)} \|\az_R\|_{L^2(\rlz)}^2\bigg)^{p\over 2}  \bigg(\sum_{R\in m(\Omega)} \mu_\lz(R) \left( { |I|\over |\tilde{I}|} \right)^{-2p}  \bigg)^{1-{p\over 2}}\\
&\quad+  \bigg(\sum_{R\in m(\Omega)} \|\az_R\|_{L^2(\rlz)}^2\bigg)^{p\over 2}   \bigg(\sum_{R\in m(\Omega)} \mu_\lz(R) \left( { |J|\over |\tilde{J}|} \right)^{-2p}  \bigg)^{1-{p\over 2}}\\
&\ls \mu_\lz(\Omega)^{{p\over 2}-1}\ \mu_\lz(\Omega)^{1-{p\over 2}}\ls1.
\end{align*}
Combining the estimates of the two terms ${\rm A}_1$ and ${\rm A}_2$, we get that \eqref{eqn:bdd of nontang maxi func on atom} holds.

Now, based on \eqref{eqn:bdd of nontang maxi func on atom}, for every $f\in \hap$, we have that
$ f=\sum_{j}\lambda_j a_j$ with $\sum_j|\lambda_j|^p \sim \|f\|_{\hap}^p$.
Hence,
\begin{align*}
\iint_{\R_+\times \R_+} |\cn_h(f)(x_1, x_2)|^p\,d\mu_\lz(x_1,x_2) &\leq \sum_j|\lz_j|^p \iint_{\R_+\times \R_+} |\cn_h (a_j)(x_1, x_2)|^p\,d\mu_\lz(x_1,x_2)\\
&\ls\sum_j|\lz_j|^p \ls\|f\|_{\hap}^p.
\end{align*}

Thus, to prove {\bf Step 6}, it suffices to prove \eqref{eqn:bdd of nontang maxi func on atom-1} and \eqref{eqn:bdd of nontang maxi func on atom-2}.  By symmetry, we only prove \eqref{eqn:bdd of nontang maxi func on atom-1}.  To this end, we write
\begin{align*}
&\int_{x_1\notin \gz I}\inzf |\cn_h\az_R(x_1, x_2)|^p\,d\mu_\lz(x_1,x_2) \\
&\quad= \lf[\sum_{k=0}^\fz\int_{2^{k+1}\gz I\setminus 2^k\gz I}\int_{8J}
+\sum_{k=0}^\fz\int_{2^{k+1}\gz I\setminus 2^k\gz I}\int_{\R_+\setminus 8J}\r] |\cn_h\az_R(x_1, x_2)|^p\,\dmzd\\
&\quad=:{\rm F}_1+{\rm F}_2.
\end{align*}

For ${\rm F}_1$, since
$$\cn_h\az_R(x_1, x_2)\leq\sup_{|x_2-y_2|<t_2}W^{[\lz]}_{t_2} \Big(\sup_{|x_1-y_1|<t_1}|W^{[\lz]}_{t_1} \az_R(y_1,\cdot)|\Big)(y_2)$$  by H\"older's inequality and the $\ltz$-boundedness of $\sup_{|x_2-y_2|<t_2}|W^{[\lz]}_{t_2} f(y)|$, we have that
\begin{align*}
{\rm F}_1& \le  \sum_{k=0}^\fz[m_\lz(8J)]^{1-\frac{p}{2}} \int_{2^{k+1}\gz I\setminus 2^k\gz I}\lf[\int_{8J}\lf[\cn_h\az_R(x_1, x_2)\r]^2\dmzt\r]^{\frac p 2}\dmzo \\
& \lesssim  [m_\lz(J)]^{1-\frac{p}{2}} \sum_{k=0}^\fz\int_{2^{k+1}\gz I\setminus 2^k\gz I}\lf[\int_J\lf[\sup_{|x_1-y_1|<t_1}\lf|W^{[\lz]}_{t_1}\az_R(\cdot, x_2)(x_1)\r|\r]^2\dmzt\r]^{\frac p 2}\dmzo .
\end{align*}
Since for any fixed $x_2$, $\inzf \az_R(z_1, x_2)\dmz(z_1)=0$,  we conclude that
\begin{eqnarray*}
|W^{[\lz]}_{t_1}\az_R(\cdot, x_2)(x_1)|&=&\lf|\int_{I}\lf[\wlzo(y_1, z_1)-\wlzo(y_1, x^1_0)\r]\az_R(z_1, x_2)\dmz(z_1)\r|\\
&\ls&\int_{I}\frac{t_1 |z_1-x_0^1| }{m_\lz(I(x^1_0, |x_1-x^1_0|))(|x^1_0-y_1|+t_1)^2}|\az_R(z_1, x_2)|\dmz(z_1),
\end{eqnarray*}
where $x^1_0$ is the center of $I$, and the last inequality follows from the fact that $\wlzo(y_1, z_1)$ as a function of $z_1$
satisfies $({\rm K_{ii}})$ in Section 2.3 (see \cite{yy}). Thus, observing that
$|x_1-x^1_0|\le|x^1_0-y_1|+t_1,$
we have
\begin{eqnarray*}
\sup_{|x_1-y_1|<t_1}|W^{[\lz]}_{t_1}\az_R(\cdot, x_2)(x_1)|\ls\int_{I}\frac{|I|}{m_\lz(I(x^1_0, |x_1-x^1_0|))|x_1-x^1_0|}|\az_R(z_1, x_2)|\dmz(z_1),
\end{eqnarray*}
%
%
%
%
%
%
%
As a consequence, we obtain that
\begin{align*}
{\rm F}_1& \ls  [m_\lz(J)]^{1-\frac{p}{2}} [m_\lz(I)]^{\frac{p}{2}}\sum_{k=0}^\fz\int_{2^{k+1}\gz I\setminus 2^k\gz I}\frac{|I|^p}{|x_1-x^1_0|^p}\frac1{m_\lz(I(x^1_0, |x_1-x^1_0|))^p}\\
&\quad\times\lf[\int_{J}\int_{I}\lf[\az_R(z_1, x_2)\r]^2\dmz(z_1)\dmzt\r]^{\frac p 2}\dmzo \\
& \ls  [m_\lz(J)]^{1-\frac{p}{2}} [m_\lz(I)]^{\frac{p}{2}}\gamma^{-p}\sum_{k=0}^\fz\frac{m_\lz(2^{k+1}I)^{1-p}}{2^{kp}} \|\alpha_R\|_{L^2(\rlz)}^p
\\
& \ls  [m_\lz(J)]^{1-\frac{p}{2}} [m_\lz(I)]^{\frac{p}{2}}\gamma^{-p}\|\alpha_R\|_{L^2(\rlz)}^p m_\lz(I)^{1-p} \sum_{k=0}^\fz\frac{2^{(k+1)(2\lz+1)(1-p)}}{2^{kp}}
\\
&\ls [\mu_\lz(R)]^{1-\frac{p}{2}} \|\alpha_R\|_{L^2(\rlz)}^p \gz^{-p},
\end{align*}
where the last inequality follows from the condition that $p\in ({2\lz+1\over 2\lz+2},1]$.

For ${\rm F}_2$, let $x^2_0$ be the center of $J$. By the cancellation of $\az_R$ and the property  $({\rm K_{ii}})$ for  $\wlzo(y_1, z_1)$ and $\wlzt(y_2, z_2)$,  we also have
\begin{eqnarray*}
&&\cn_{h}\az_R(x_1, x_2)\\
&&\le\!\sup_{\gfz{|y_1-x_1|<t_1}{|y_2-x_2|<t_2}}\!
\int_{I}\!\int_{J}\lf|\wlzo(y_1, z_1)-\wlzo(y_1, x^1_0)\r|\lf|\wlzt(y_2, z_2)-\wlzt(y_2, x^2_0)\r||\az_R(z_1, z_2)|d\mu_\lz(z_1,z_2)\\
&&\ls\sup_{\gfz{|y_1-x_1|<t_1}{|y_2-x_2|<t_2}}\int_{I}\int_{J}\frac1{m_\lz(I(x^1_0, |x^1_0-x_1|))}\frac{t_1|I|}{(|x^1_0-y_1|+t_1)^2}\\
&&\quad\quad\quad\quad\quad\quad\quad\quad\times \frac1{m_\lz(I(x^2_0, |x^2_0-x_2|))}\frac{t_2|J|}{(|x^2_0-y_2|+t_2)^2}|\az_R(z_1, z_2)|d\mu_\lz(z_1,z_2)\\
&&\ls [\mu_\lz(R)]^{\frac{1}{2}}\frac{|I|}{|x_1-x^1_0|}\frac1{m_\lz(I(x^1_0, |x_1-x^1_0|))}
\frac{|J|}{|x_2-x^2_0|} \frac1{m_\lz(I(x^2_0, |x_2-x^2_0|))}\|\az_R\|_{L^2(\rlz)}\\
&&\ls\frac{|I|}{|x_1-x^1_0|}\frac1{m_\lz(I(x^1_0, |x_1-x^1_0|))}
\frac{|J|}{|x_2-x^2_0|}\frac1{m_\lz(I(x^2_0, |x_2-x^2_0|))} [\mu_\lz(R)]^{\frac{1}{2}}\|\az_R\|_{L^2(\rlz)}.
\end{eqnarray*}
Therefore,
\begin{align*}
{\rm F}_2&\ls [\mu_\lz(R)]^{\frac{p}{2}}\|\az_R\|_{L^2(\rlz)}^p\sum_{k=0}^\fz\sum_{l=3}^\fz\int_{2^{k+1}\gz I\setminus 2^k\gz I}\int_{2^{l+1}J\setminus 2^lJ}\frac{|I|^p}{|x_1-x^1_0|^p}\frac1{m_\lz(I(x^1_0, |x_1-x^1_0|))^p}\\
&\quad\times\frac{|J|^p}{|x_2-x^2_0|^p}\frac1{m_\lz(I(x^2_0, |x_2-x^2_0|))^p}\dmzt\dmzo \\
&\ls [\mu_\lz(R)]^{1-\frac{p}{2}} \|\alpha_R\|_{L^2(\rlz)}^p \gz^{-p}.
\end{align*}
Combining the estimates of ${\rm F}_1$ and ${\rm F}_2$, we obtain \eqref{eqn:bdd of nontang maxi func on atom-1}.
%
We finish the proof of Theorem \ref{thm: char of Hardy spacs by max func}.

\section{Proofs of Theorems \ref{thm H1 riesz characterization} and \ref{thm Hp Riesz characterization}}
\label{sec:third theorem}

In this section, we present the proofs of Theorems \ref{thm H1 riesz characterization} and \ref{thm Hp Riesz characterization}.
We  first note that $\rizo\rizt$ is a product Calder\'on--Zygmund operator on space of homogeneous type $\rlz$ (see the definition in Section 1 of \cite{HLLin}). And we consider
$\rizo $ as $\rizo\otimes Id_2$ and $\rizt$ as $Id_1\otimes\rizt$, where we use $Id_1$ and $Id_2$ to denote the identity operator on $L^2(\R_+,dm_\lz)$. Then we can also understand $\rizo$ and $\rizt$ as product Calder\'on--Zygmund operators on $\rlz$. We recall that the product Calder\'on--Zygmund operators $T$
are bounded on $\lrzd$ for $r\in (1, \fz)$, on $H^p(\rlz)$ (\cite[Section 3.1]{HLLin}) and from $H^p(\rlz)$ to $L^p(\rlz)$ for $p\in((2\lz+1)/(2\lz+2),1]$ (\cite[Section 2.1.3]{HLLin}). Hence, for any
$f\in H^p_{\Delta_\lz}(\rlz)$, we have
\begin{eqnarray}\label{restr infty hp lp norm}
&&\lf\|\rizo\lf(f\r)\r\|_\lpzd+\lf\|\rizt\lf(f\r)\r\|_\lpzd+\lf\|\rizo\rizt\lf(f\r)\r\|_\lpzd\ls \|f\|_{H^p_{\Delta_\lz}(\rlz)}.
\end{eqnarray}

Before we give the proof of Theorem \ref{thm H1 riesz characterization}, we first recall the following result from Lemma 11 in \cite{ms}.

\begin{lem}\label{lem: subharmonic func}
Suppose that
\begin{itemize}
\item[(i)] $u(t, x)$ is continuous in $t\in[0, \fz)$, $x\in\rr$ and even in $x$;

\item[(ii)] In the region where $u(t, x)>0$, $u$ is of class $C^2$ and satisfies
$\partial_t^2u+\partial_x^2u
+2\lz x^{-1}\partial_x u\ge0$;

\item[(iii)] $u(0, x)=0;$

\item[(iv)] For some $r\in[1, \fz)$, there exists a positive constant $\wz C$ such that
\begin{equation*}
\sup_{0<t<\fz}\dint_0^\fz|u(t, x)|^r\,dm_\lz(x)\le \wz C<\fz.
\end{equation*}
\end{itemize}
Then $u(t, x)\le0$.
\end{lem}

\medskip

%

For $f\in\lpzd$ with $p\in [1, \fz)$, and $t_1,\,t_2,\,x_1,\,x_2\in\R_+$, let
\begin{equation}\label{real and imaginary parts-1}
u(t_1,\,t_2,\,x_1,\,x_2):=\plzo\plzt f(x_1,x_2),\,\,\,\,v(t_1,\,t_2,\,x_1,\,x_2):=\qlzo\plzt f(x_1,x_2),
\end{equation}
and
\begin{equation}\label{real and imaginary parts-2}
w(t_1,\,t_2,\,x_1,\,x_2):=\plzo\qlzt f(x_1,x_2),\,\,\,\,z(t_1,\,t_2,\,x_1,\,x_2):=\qlzo\qlzt f(x_1,x_2),
\end{equation}
where $\qlzo$ and $\qlzt$ are defined as in \eqref{conj poi defn}.
Moreover, define
\begin{equation}\label{radi maxi defn}
u^\ast(x_1, x_2):=\crz_{P}f(x_1, x_2)
\end{equation}
%
and
\begin{align}\label{sum of real and imagi parts}
F(t_1,\,t_2,\,x_1,\,x_2)&:=\lf\{[u(t_1,\,t_2,\,x_1,\,x_2)]^2+[v(t_1,\,t_2,\,x_1,\,x_2)]^2\r.\\
&\quad\quad+\lf.[w(t_1,\,t_2,\,x_1,\,x_2)]^2+[z(t_1,\,t_2,\,x_1,\,x_2)]^2\r\}^{\frac12}.\noz
\end{align}
We first establish the following lemma.

\begin{lem}\label{lem: F controled by riesz}
Let $f\in \horiz$, $u$,  $v$, $w$, $z$ and $F$ be, respectively, as in \eqref{real and imaginary parts-1},
\eqref{real and imaginary parts-2} and \eqref{sum of real and imagi parts}. Then there exists a positive constant $C$
independent of $f$, $u$,  $v$, $w$, $z$ and $F$, such that
\begin{eqnarray*}
&&\supd\dinrp F(t_1,\,t_2,\,x_1,\,x_2)\dmzd\le C  \|f\|_\horiz.
\end{eqnarray*}
\end{lem}

\begin{proof}
It suffices to show that
\begin{eqnarray}\label{F controled by riesz-u}
\supd\dinrp|u(t_1,\,t_2,\,x_1,\,x_2)|\dmzd\le \|f\|_\lozd,
\end{eqnarray}
\begin{eqnarray}\label{F controled by riesz-v}
\supd\dinrp|v(t_1,\,t_2,\,x_1,\,x_2)|\dmzd\ls \|\rizo f\|_\lozd,
\end{eqnarray}
\begin{eqnarray}\label{F controled by riesz-w}
\supd\dinrp|w(t_1,\,t_2,\,x_1,\,x_2)|\dmzd\ls \|\rizt f\|_\lozd,
\end{eqnarray}
and
\begin{eqnarray}\label{F controled by riesz-z}
\supd\dinrp|z(t_1,\,t_2,\,x_1,\,x_2)|\dmzd\ls \|\rizo\rizt f\|_\lozd.
\end{eqnarray}

To this end, we first note that \eqref{F controled by riesz-u} follows from ${\rm (S_{i})}$ in Lemma \ref{l-seimgroup prop}.
Moreover, 
by the fact that for any $t,\,y\in\R_+$,
\begin{equation}\label{upp bdd for poiss inte}
\inzf xy P^{[\lz+1]}_t(x, y)\dmz(x)\,\ls 1,
\end{equation}
we obtain that for every $f\in\loz$ and $t\in \R_+$,
\begin{eqnarray}\label{conjuga for poiss inte}
\lf\|\qlz f\r\|_\loz&=&\lf\|\inzf \cdot yP^{[\lz+1]}_t(\cdot, y)\riz (f)(y)\,\dmz(y)\r\|_\loz\\
&\ls&\|\riz f\|_\loz;\noz
\end{eqnarray}
see \cite[p.\, 208]{bdt}. Therefore, by the uniform $\loz$-boundedness of $\left\{\plz\right\}_{t>0}$(Lemma \ref{l-seimgroup prop} ${\rm (S_i)}$),
we see that for any $t_1$, $t_2\in R_+$,
\begin{eqnarray*}
\dinrp|v(t_1,\,t_2,\,x_1,\,x_2)|\dmzd\!
\le\! \dinrp\lf|\qlzo f(x_1,\,x_2)\r|\dmzd\ls  \|\rizo f\|_\lozd.
\end{eqnarray*}
This implies \eqref{F controled by riesz-v}. Similarly, we have \eqref{F controled by riesz-w}.

Finally, from \eqref{conjuga for poiss inte}, we deduce that
\begin{eqnarray*}
z(t_1,\,t_2,\,x_1,\,x_2)=\inzf\inzf x_1y_1P^{[\lz+1]}_{t_1}(x_1, y_1)x_2y_2 P^{[\lz+1]}_{t_2}(x_2, y_2)\rizo\rizt f(y_1, y_2)\,d\mu_\lz(y_1,y_2).
\end{eqnarray*}
By this and \eqref{upp bdd for poiss inte}, we show \eqref{F controled by riesz-z} immediately.
This finishes the proof of Lemma \ref{lem: F controled by riesz}.
\end{proof}

\smallskip
\begin{proof}[{\bf Proof of Theorem \ref{thm H1 riesz characterization}}]

We first show that for any $f\in H^1_{\Delta_\lz}(\rlz)$,
\begin{align*}
\|f\|_{H^1_{Riesz}(\rlz)}\ls \|f\|_{H^1_{\Delta_\lz}(\rlz)}.
\end{align*}
To see this, based on Definition \ref{def of Hardy space via riesz} and Theorem \ref{thm H1 and classical H1}, it suffices to prove that
\begin{align*}
\|f\|_\lozd+\|\rizo f\|_\lozd+ \|\rizt f\|_\lozd+ \|\rizo\rizt f\|_\lozd \ls \|f\|_{H^1_{\Delta_\lz}(\rlz)},
\end{align*}
which follows from \eqref{restr infty hp lp norm} with $p:=1$ and the fact that $H^1_{\Delta_\lz}(\rlz)$ is a subspace of $L^1(\rlz)$.

Conversely, assume that $f\in H^1_{Riesz}(\rlz)$.
By Theorem \ref{thm: char of Hardy spacs by max func}, it suffices to show that
\begin{align*}
\|f\|_{H^1_{\mathcal{R}_{P}}(\rlz)}\ls \|f\|_{H^1_{Riesz}(\rlz)}.
\end{align*}
To this end, based on Lemma \ref{lem: F controled by riesz}, it remains to prove that
\begin{equation}\label{radia max cont by holomo fun}
\|f\|_{H^1_{\mathcal{R}_{P}}(\rlz)}=\|u^\ast\|_\lozd\ls{\displaystyle\supd\dinrp} F(t_1,\,t_2,\,x_1,\,x_2)\,\dmzo\dmzt,
\end{equation}
where $u^\ast$ and $F$ are as in \eqref{radi maxi defn} and \eqref{sum of real and imagi parts}. 
We first claim that  we only need to show that for $p\in \big(\frac{2\lz+1}{2\lz+2}, 1\big)$ and $\ez_1,\,t_1,\,\ez_2,\, t_2,\,x_1,\,x_2\in\R_+$,
\begin{equation}\label{F subharmonic}
F^p(\ez_1+t_1,\,\ez_2+ t_2,\,x_1,\,x_2)\ls \plzo\plzt\lf(F^p(\ez_1,\,\ez_2,\,\cdot,\,\cdot)\r)(x_1,x_2).
\end{equation}
Indeed, by Lemma \ref{lem: F controled by riesz}, we see that $F\in \lozd$. If \eqref{F subharmonic} holds, then
   the uniform $\lrz$-boundedness of $\left\{\plz\right\}_{t>0}$, with $r:=1/p$, implies that $\{F^p(\ez_1,\,\ez_2,\,\cdot,\,\cdot)\}_{\ez_1,\,\ez_2>0}$ is bounded in $L^r(\rlz)$. Since $\lrzd$  is reflexive, there exist two  sequences $\{\ez_{1,\,k}\}$, $\{\ez_{2,\,j}\}\downarrow0$ and
 $h\in L^r(\rlz)$ such that $\{F^p(\ez_{1,\,k},\,\ez_{2,\,j},\,\cdot,\,\cdot)\}_{\ez_{1,\,k},\,\ez_{2,\,j}>0}$ converges weakly to $h$
 in $L^r(\rlz)$ as $k,\,j\to\fz$. Moreover, by H\"older's inequality, we see that
\begin{align}\label{norm of weak limit}
\|h\|_{L^r(\rlz)}^r
&=\bigg\{\dsup_{\|g\|_{L^{r'}(\rlz)}\le1}\lf|\dinrp g(x_1,x_2)h(x_1,x_2)\,d\mu_\lz(x_1,x_2)\r|\bigg\}^r\\
&=\bigg\{\dsup_{\|g\|_{L^{r'}(\rlz)}\le1}\lim_{\gfz{k\to\fz}{j\to\fz}}
\lf|{\displaystyle\dinrp} g(x_1, x_2)F^p(\ez_{1,\,k},\,\ez_{2,\,j},\,\cdot,\,\cdot)\,d\mu_\lz(x_1,x_2)\r|\bigg\}^r\noz\\
&\le \dlimsup_{\gfz{k\to\fz}{j\to\fz}}\big\|F^p(\ez_{1,\,k},\,\ez_{2,\,j},\,\cdot,\,\cdot)\big\|_{L^r(\rlz)}^r\noz\\
&\le\sup_{t_1>0,\,t_2>0}\dinrp F(t_1,\,t_2,\,x_1,\,x_2)\,d\mu_\lz(x_1,x_2).\noz
\end{align}

Since $F$ is continuous in $t_1$ and $t_2$, for any $x_1,\,x_2\in\R_+$,
$$F^p(t_1+\ez_{1,\,k}, t_2+\ez_{2,\,j}, x_1, x_2)\to F^p(t_1, t_2, x_1, x_2)$$ as $k,\,j\to\fz$.
Observe that for each $x_1,\,x_2\in (0, \fz)$,
$$\plzo\plzt(F^p(\ez_{1,\,k},\,\ez_{2,\,j},\,\cdot,\,\cdot))(x_1, x_2)\to \plzo\plzt(h)(x_1, x_2)$$
as $k,\,j\to\fz$.
Thus, by these facts and \eqref{F subharmonic}, we have that
for any $t_1,\,t_2,\,x_1,\,x_2\in\R_+$,
\begin{eqnarray*}
F^p(t_1, t_2, x_1, x_2)&=&\dlim_{\gfz{k\to\fz}{j\to \fz}}F^p(t_1+\ez_{1,\,k}, t_2+\ez_{2,\,j}, x_1, x_2)\\
&\ls&\dlim_{\gfz{k\to\fz}{j\to\fz}}\plzo\plzt(F^p(\ez_{1,\,k},\,\ez_{2,\,j},\,\cdot,\,\cdot))(x_1, x_2)\\
&=&\plzo\plzt(h)(x_1, x_2).
\end{eqnarray*}
Therefore,
\begin{equation*}
[u^\ast(x_1, x_2)]^p\le\sup_{t_1>0,\,t_2>0}F^p(t_1,t_2, x_1, x_2)\ls\mathcal R_P(h)(x_1, x_2).
\end{equation*}
By this together with $r:=1/p$, the $\lrzd$-boundedness of $\mathcal R_P$ and \eqref{norm of weak limit}, we then have
\begin{align*}
\|u^\ast\|_\lozd&\ls \Big\|\mathcal R_P(h)\Big\|^r_{L^r(\rlz)}
\ls{\sup_{t_1>0,\,t_2>0}\dinrp} F(t_1,\,t_2,\,x_1,\,x_2)\,d\mu_\lz(x_1,x_2),
\end{align*}
which implies that \eqref{radia max cont by holomo fun}. Thus the claim holds.


Now we prove \eqref{F subharmonic}.
 Observe that for any fixed $t_2,\,x_2\in\R_+$, $u,\,v$ and $w,\,z$ respectively satisfy the Cauchy-Riemann equations for $t_1$ and
$x_1$,
and for any fixed $t_1, \,x_1\in \R_+$, $u,\,w$ and $v,\,z$ respectively satisfy the Cauchy-Riemann equations for $t_2$ and
$x_2$. That is,
\begin{equation}\label{CR equ-1}
\displaystyle\left\{
    \begin{array}{ll}
      \prz_{x_1} u+\prz_{t_1} v=0, \\
      \prz_{t_1} u-\prz_{x_1} v=\frac{2\lambda}{x_1}v;
    \end{array}
  \right.
\hspace{0.8cm}
\left\{
    \begin{array}{ll}
    \prz_{x_1} w+\prz_{t_1} z=0, \\
      \prz_{t_1} w-\prz_{x_1} z=\frac{2\lambda}{x_1}z;
   \end{array}
  \right.
\end{equation}
and
\begin{equation}\label{CR equ-2}
\left\{
    \begin{array}{ll}
      \prz_{x_2} u+\prz_{t_2} w=0, \\
      \prz_{t_2} u-\prz_{x_2} w=\frac{2\lambda}{x_2}w;
    \end{array}
  \right.
\hspace{0.8cm}
\left\{
    \begin{array}{ll}
    \prz_{x_2} v+\prz_{t_2} z=0, \\
      \prz_{t_2} v-\prz_{x_2} z=\frac{2\lambda}{x_2}z.
   \end{array}
  \right.
\end{equation}

For fixed $t_2,\, x_2\in\R_+$, let
$$ F_1(t_1,\,t_2,\,x_1,\,x_2):=\lf\{[u(t_1,\,t_2,\,x_1,\,x_2)]^2+[v(t_1,\,t_2,\,x_1,\,x_2)]^2\r\}^{\frac12},$$
where $t_1,\,x_1\in \R_+$. For the moment, we fix $t_2$, $x_2$ and regard $F_1$ as a function of $t_1$ and $x_1$.
By \eqref{CR equ-1},  for $p\in (\frac{2\lz+1}{2\lz+2}, 1)$,
\begin{equation}\label{subharmonic for Fi}
\prz_{t_1}^2 F^p_1(t_1,\,t_2,\,x_1,\,x_2)+\prz_{x_1}^2 F^p_1(t_1,\,t_2,\,x_1,\,x_2)
+\frac{2\lz}{x_1}\prz_{x_1} F^p_1(t_1,\,t_2,\,x_1,\,x_2)\ge0;
\end{equation}
see \cite[Lemma 5]{ms} or \cite[p.\,206]{bdt}.
We also observe that for any $\ez_1,\,t_2,\,x_1, x_2\in\R_+$,
\begin{equation}\label{subharm fun lim bound valu}
\lim_{t_1\to 0}P^{[\lz]}_{t_1}(F^p_1(\ez_1,\,t_2,\,\cdot,\,x_2))(x_1)=F^p_1(\ez_1,\,t_2,\,x_1,\,x_2),
\end{equation}
and  by Lemma \ref{lem: F controled by riesz}, for all $t_2\in \R_+$ and almost $x_2\in \R_+$,
\begin{equation}\label{subharm fun  norm bdd}
\sup_{t_1>0}\inzf \lf[F^p_1(t_1,\,t_2,\,x_1,\,x_2)\r]^r\dmzo\le \sup_{t_1>0}\inzf F(t_1,\,t_2,\,x_1,\,x_2)\dmzo<\fz.
\end{equation}
Now we  claim that for any $\ez_1, t_1,\,t_2,\,x_1, x_2\in\R_+$,
\begin{align}\label{harmo contr-F1}
F^p_1(\ez_1+t_1,\,t_2,\,x_1,\,x_2)\le \plzo\lf(F^p_1(\ez_1,\,t_2,\,\cdot,\, x_2)\r)(x_1).
\end{align}
Indeed, as in \cite{bdt}, for any $\ez_1,\, t_1,\,t_2,\,x_1, x_2\in\R_+$,
let
$$\wz F_{1,\,t_2,\,x_2}(t_1,\,x_1):=\lf\{[\wz u(t_1,\,t_2,\,x_1,\,x_2)]^2+[\wz v(t_1,\,t_2,\,x_1,\,x_2)]^2\r\}^{\frac12}$$
and
$$U_{\ez_1,\,t_2,\,x_2}(t_1,\,x_1):=\wz F_{1,\,t_2,\,x_2}^p(\ez_1+t_1,\,x_1)-\wz \plzo[\wz F_{1,\,t_2,\,x_2}^p(\ez_1,\,\cdot)](x_1),$$
where  for fixed $t_2$ and $x_2$, $\wz \plzo(\wz F_{1,\,t_2,\,x_2}^p), \wz u$ are even extensions of $\plzo(\wz F_{1,\,t_2,\,x_2}^p)$ and $u$, and $\wz v$  is the odd extension of
$v$ with respect to $x_1$ to $\R_+\times \rr$, respectively.
By \eqref{bessel laplace equation}, \eqref{subharmonic for Fi}, \eqref{subharm fun lim bound valu} and \eqref{subharm fun  norm bdd}, it is not difficult to check that
$U_{\ez_1,\,t_2,\,x_2}$ satisfies (i)-(iv) of Lemma \ref{lem: subharmonic func}.
Then an application of Lemma \ref{lem: subharmonic func} shows that $U_{\ez_1,\,t_2,\,x_2}(t_1,x_1)\le 0$. Thus, \eqref{harmo contr-F1} holds.

Similarly, let
$$F_2(t_1,\,t_2,\,x_1,\,x_2):=\lf\{[w(t_1,\,t_2,\,x_1,\,x_2)]^2+[z(t_1,\,t_2,\,x_1,\,x_2)]^2\r\}^{\frac12}.$$
Since  \eqref{subharmonic for Fi}, \eqref{subharm fun lim bound valu} and \eqref{subharm fun  norm bdd} all hold with $F_1$ replaced with $F_2$, 
we have that  for any $\ez_1, t_1,\,t_2$, $x_1$, $x_2\in\R_+$,
\begin{align}\label{harmo contr-F2}
F^p_2(\ez_1+t_1,\,t_2,\,x_1,\,x_2)\le \plzo\lf(F^p_2(\ez_1,\,t_2,\,\cdot,\, x_2)\r)(x_1).
\end{align}
Observe that for any $t_1,\,t_2,\,x_1,\,x_2\in \R_+$,
\begin{equation*}
F(t_1,\,t_2,\,x_1,\,x_2)\sim \sum_{i=1}^2F_i(t_1,\,t_2,\,x_1,\,x_2).
\end{equation*}
By this fact, \eqref{harmo contr-F1}, \eqref{harmo contr-F2} and Lemma \ref{l-seimgroup prop} ${\rm (S_{ii})}$, we have that
\begin{equation}\label{harmo contr-F}
F^p(\ez_1+t_1,\,t_2,\,x_1,\,x_2)\ls \plzo\lf(F^p(\ez_1,\,t_2,\,\cdot,\,\cdot)\r)(x_1,x_2).
\end{equation}
Moreover, 
%
%
from \eqref{CR equ-2}, Lemma \ref{lem: F controled by riesz} and Lemma \ref{lem: subharmonic func},
we also deduce that
\begin{equation*}
F^p(t_1,\,\ez_2+ t_2,\,x_1,\,x_2)\ls \plzt\lf(F^p(t_1,\,\ez_2,\,\cdot,\,\cdot)\r)(x_1,x_2).
\end{equation*}
Now by this and \eqref{harmo contr-F}, we conclude that
\begin{eqnarray*}
F^p(\ez_1+t_1,\,\ez_2+ t_2,\,x_1,\,x_2)&&\ls \plzo\lf(F^p(\ez_1,\,\ez_2+ t_2,\,\cdot,\,\cdot)\r)(x_1,x_2)\\
&&\ls \plzo\plzt\lf(F^p(\ez_1,\,\ez_2,\,\cdot,\,\cdot)\r)(x_1,x_2).\noz
\end{eqnarray*}
This implies \eqref{F subharmonic}, and hence finishes the proof of Theorem \ref{thm H1 riesz characterization}.
\end{proof}

\smallskip
We next present the proof of Theorem \ref{thm Hp Riesz characterization}.

\begin{proof}[{\bf Proof of Theorem \ref{thm Hp Riesz characterization}}]  We first assume that $f\in\hap$.
By Theorem \ref{thm: char of Hardy spacs by max func}, we see that for any
$f\in\hap$,  $\plzo\plzt f\in\lpzd$ with
\begin{equation*}
\lf\|\plzo\plzt f\r\|_\lpzd\le \|\crz_{P} f\|_\lpzd\ls \|f\|_\hap,
\end{equation*}
where the implicit constant is independent of $t_1, t_2$ and $f$.
Moreover, by the semigroup property of $\left\{\plz\right\}_{t>0}$ (Lemma \ref{l-seimgroup prop})
and Theorem \ref{thm: char of Hardy spacs by max func},  we obtain that
\begin{eqnarray*}
\lf\|\plzo\plzt f\r\|_\hap\sim\lf\|\crz_{P}\lf(\plzo\plzt f\r)\r\|_\lpzd
&&=\lf\|\crz_{P}\lf(f\r)\r\|_\lpzd\ls \|f\|_\hap.
\end{eqnarray*}
This implies $\plzo\plzt f\in\hap$ with the norm independent of $t_1$ and $t_2$.
Then an application of \eqref{restr infty hp lp norm} shows that \eqref{eqn:restr infty lp norm; Riesz trans char Hp} holds.

%
%

Conversely, we assume that $f$ satisfies \eqref{eqn:restr infty lp norm; Riesz trans char Hp} and show $f\in \hap$.
The proof  is similar to that of Theorem \ref{thm H1 riesz characterization}. Precisely, we first claim that
\begin{equation}\label{lp norm tilde G}
\sup_{\gfz{t_1>0}{t_2>0}}\iint_{\R_+\times\R_+}|F(t_1, t_2, x_1, x_2)|^p\,\dmzd\ls1,
\end{equation}
where $F(t_1, t_2, x_1, x_2)$ is as in \eqref{sum of real and imagi parts}.
From Proposition \ref{pp-aoti}, we see that the definition of $F$ makes sense.
Indeed, for $\dz_1,\,\dz_2,\,t_1,\,t_2,\, x_1,\,x_2\in \R_+$,
let
\begin{equation*}\label{real and imag parts p-u}
u(\dz_1,\,\dz_2,\,t_1,\,t_2,\,x_1,\,x_2):=\plzo\plzt\lf(\plzzo\plzzt f\r)(x_1,x_2),
\end{equation*}
\begin{equation*}\label{real and imag parts p-v}
\,\,\,\,v(\dz_1,\,\dz_2,\,t_1,\,t_2,\,x_1,\,x_2):=\qlzo\plzt \lf(\plzzo\plzzt f\r)(x_1,x_2),
\end{equation*}
\begin{equation*}\label{real and imag parts p-w}
w(\dz_1,\,\dz_2,\,t_1,\,t_2,\,x_1,\,x_2):=\plzo\qlzt\lf(\plzzo\plzzt f\r)(x_1,x_2),
\end{equation*}
and
\begin{equation*}\label{real and imag parts p-z}
z(\dz_1,\,\dz_2,\,t_1,\,t_2,\,x_1,\,x_2):=\qlzo\qlzt \lf(\plzzo\plzzt f\r)(x_1,x_2).
\end{equation*}
%
%
%
Moreover, as functions of variables $(\dz_1,\,\dz_2,\,t_1,\,t_2,\,x_1,\,x_2)$, we  define
$$G_1:=\lf\{u^2+v^2\r\}^{\frac12},\ G_2:=\lf\{w^2+z^2\r\}^{\frac12},\ G_3:=\lf\{u^2+w^2\r\}^{\frac12},\ G_4:=\lf\{v^2+z^2\r\}^{\frac12}$$
and
\begin{eqnarray*}
G&&:=\{u^2+v^2+w^2+z^2\}^{\frac12}.
\end{eqnarray*}

We now prove that for  all $\dz_1,\,\dz_2,\,t_1,\,t_2,\,x_1,\,x_2\in \R_+$, when $i:=1,\,2,$
\begin{equation}\label{subhar G-1 p}
G^p_i(\dz_1,\,\dz_2,\,t_1,\,t_2,\,x_1,\,x_2)\le \plzo\lf(G^p_i(\dz_1,\,\dz_2,\,0,\,t_2,\cdot,\, x_2)\r)(x_1),
\end{equation}
and when $i:=3,\,4$,
\begin{equation}\label{subhar G-2 p}
G^p_i(\dz_1,\,\dz_2,\,t_1,\,t_2,\,x_1,\,x_2)\le \plzt\lf(G^p_i(\dz_1,\,\dz_2,\,t_1,\,0,\, x_1,\,\cdot)\r)(x_2).
\end{equation}
Here we mention that for any function $g\in\lrz$ with $r\in[1, \fz)$,
$$P^{[\lz]}_0g:=\lim_{t\to 0}\plz g= g\quad{\rm and}\quad Q^{[\lz]}_0g:=\lim_{t\to0}Q_tg=\riz g;$$
see \cite{bdt,bfbmt}.  Since the assumption that
$f$ is restricted at infinity implies that $\plzo\plzt f\in \lszd$
for all $t_1, t_2\in \R_+$ and $s\in[p, \fz]$,
$G_i(\dz_1,\,\dz_2,\,t_1,\,0,\, x_1,\,x_2)$ and  $G_i(\dz_1,\,\dz_2,\,0,\,t_2,\, x_1,\,x_2)$ make sense for $i\in\{1,2,3,4\}$.

By similarity, we only show that \eqref{subhar G-1 p} holds for $G_2$. We fix $\dz_1,\dz_2,t_2,x_2\in \R_+$ and regard $G_2$ as a function of $t_1$ and $x_1$ for the moment and the argument is analogous to that for \eqref{harmo contr-F1}.  Indeed,
let
$$V_{\dz_1,\,\dz_2,\,t_2,\,x_2}(t_1,\,x_1)
:=\wz G^p_{2,\,\dz_1,\,\dz_2,\,t_2,\,x_2}(t_1,\,x_1)- \wz\plzo\lf(\wz G^p_{2,\,\dz_1,\,\dz_2,\,t_2,\,x_2}(0,\,\cdot)\r)(x_1),$$
where
$\wz G^p_{2,\,\dz_1,\,\dz_2,\,t_2,\,x_2}(t_1,\,x_1):=[\wz w^2+\wz z^2]^\frac12$ with $\wz w,\,\wz z$ are even and odd extensions of $w$ and $z$ with respect to $x_1$ to $\R$, respectively,
, and $\wz\plzo(\wz G^p_{2,\,\dz_1,\,\dz_2,\,t_2,\,x_2}(0,\,\cdot))(x_1)$
is the even extension of $\plzo(\wz G^p_{2,\,\dz_1,\,\dz_2,\,t_2,\,x_2}(0,\,\cdot))(x_1)$ to $\R$ with respect to $x_1$.

We now show that $V_{\dz_1,\,\dz_2,\,t_2,\,x_2}$ satisfies (i)-(iv) of Lemma \ref{lem: subharmonic func}.
In fact, since $V_{\dz_1,\,\dz_2,\,t_2,\,x_2}$ is an even function with respect to $x_1$, we only need to consider $x_1\in\R_+$.
Since the assumption that
$f$ is restricted at infinity implies that $\plzo\plzt f\in \lszd$
for all $t_1, t_2\in \R_+$ and $s\in[p, \fz]$,
by Lemma \ref{l-seimgroup prop} ${\rm (S_i)}$, the uniform boundedness of $\qlz$ on $\ltz$ (see \cite[p.\,87]{ms}),
we further obtain that for fixed $\dz_1,\,\dz_2,\,t_1,\,t_2\in\R_+$,
\begin{equation}\label{3.20}
\begin{array}[t]{cl}
&\displaystyle\dinrp G^2_2(\dz_1,\,\dz_2,\,t_1,\,t_2,\,x_1,\,x_2)\,d\mu_\lz(x_1,x_2)\\
&\quad=
\displaystyle\dinrp \lf[\lf|\plzo\qlzt\lf(\plzzo\plzzt f\r)\r|^2
+\lf|\qlzo\qlzt \lf(\plzzo\plzzt f\r)\r|^2\r]\,d\mu_\lz(x_1,x_2)\\
&\quad\ls\displaystyle\dinrp\lf[\plzzo\plzzt f(x_1, x_2)\r]^2\,d\mu_\lz(x_1,x_2)<\fz.
\end{array}
\end{equation}

Observe that for all $\dz_1,\,\dz_2,\,t_2$ and almost every $x_1,\,x_2\in (0, \fz)$,
$$[G_2(\dz_1,\,\dz_2,\,0,\,t_2,\,x_1,\,x_2)]^2=\lf|\qlzt\lf(\plzzo\plzzt f\r)(x_1,x_2)\r|^2+\lf|\rizo\qlzt \lf(\plzzo\plzzt f\r)(x_1,x_2)\r|^2.$$
This fact together with Lemma \ref{l-seimgroup prop} ${\rm (S_i)}$, the uniform boundedness of  $\qlzt$ on $L^2(\R_+, \dmzt)$ and the boundedness of $\rizo$ on $L^2(\R_+, \dmzo)$ implies that
\begin{equation*}
\begin{array}[t]{cl}
&\dint_0^\fz\dint_0^\fz\lf[\plzo(G^p_2(\dz_1,\,\dz_2,\,0,\,t_2,\,\cdot,\,x_2))(x_1)\r]^{2/p}\,d\mu_\lz(x_1,x_2)\\
&\quad\ls\dint_0^\fz\dint_0^\fz\lf|G_2(\dz_1,\,\dz_2,\,0,\,t_2,\,x_1,\,x_2)\r|^2\,d\mu_\lz(x_1,x_2)\\
&\quad\ls \dint_0^\fz\dint_0^\fz\lf[\lf|\qlzt\lf(\plzzo\plzzt f\r)(x_1,x_2)\r|^2+\lf|\rizo\qlzt \lf(\plzzo\plzzt f\r)(x_1,x_2)\r|^2\r]\,d\mu_\lz(x_1,x_2)\\
&\quad\ls\dint_0^\fz\dint_0^\fz\lf|\plzzo\plzzt f(x_1,x_2)\r|^2\,d\mu_\lz(x_1,x_2),
\end{array}
\end{equation*}
which, together with \eqref{3.20}, yields that for all $\dz_1,\,\dz_2,\,t_2$ and almost all $x_2$,
\begin{equation*}
\begin{array}[t]{cl}
&\dsup_{0<t_1<\fz}\dint_0^\fz|V_{\dz_1,\,\dz_2,\,t_2,\,x_2}(t_1,\,x_1)|^{2/p}\,dm_\lz(x_1)
\ls\dint_0^\fz\lf|\plzzo\plzzt f(x_1,x_2)\r|^2\,dm_\lz(x_1)<\fz.
\end{array}
\end{equation*}
Therefore, for fixed $\dz_1,\,\dz_2,\,t_2,\,x_2$, $V_{\dz_1,\,\dz_2,\,t_2,\,x_2}(t_1, x_1)$ satisfies the assumptions of Lemma \ref{lem: subharmonic func} and hence \eqref{subhar G-1 p} for $G_2$
follows from Lemma \ref{lem: subharmonic func} immediately.

By an argument involving \eqref{subhar G-1 p} and \eqref{subhar G-2 p}, we further see that for almost all $x_1$ and $x_2$,
$$G^p(\dz_1,\,\dz_2,\,t_1,\,t_2,\,x_1,\,x_2)\ls \plzt\plzo\lf(G^p(\dz_1,\,\dz_2,\,0,\,0,\,\cdot,\,\cdot)\r)(x_1, x_2).
$$
Moreover, observe that
\begin{eqnarray*}
G(\dz_1,\,\dz_2,\,0,\,0,\,x_1,\,x_2)&&\sim \lf|P^{[\lz]}_{\dz_1}P^{[\lz]}_{\dz_2}(f)(x_1, x_2)\r|+\lf|\rizo\lf(P^{[\lz]}_{\dz_1}P^{[\lz]}_{\dz_2} f\r)(x_1, x_2)\r|\\
&&\quad+\lf|\rizt\lf(P^{[\lz]}_{\dz_1}P^{[\lz]}_{\dz_2} f\r)(x_1, x_2)\r|+\lf|\rizo\rizt\lf(P^{[\lz]}_{\dz_1}P^{[\lz]}_{\dz_2} f\r)(x_1, x_2)\r|.
\end{eqnarray*}
Using these facts, \eqref{eqn:restr infty lp norm; Riesz trans char Hp} and Lemma \ref{l-seimgroup prop} ${\rm (S_i)}$,
we see that
\begin{align}\label{lp norm F_1}
&\dinzf[G(\dz_1,\,\dz_2,\,t_1,\,t_2,\,x_1,\,x_2)]^p\,\dmzd\\
&\quad\ls \dinzf\plzt\plzo([G(\dz_1,\,\dz_2,\,0,\,0,\,\cdot,\,\cdot)]^p)(x_1,x_2)\,\dmzd\noz\\
&\quad\ls\dinzf|G(\dz_1,\,\dz_2,\,0,\,0,\,x_1,\,x_2)|^p\,\dmzd\noz\\
&\quad\sim\dint_0^\fz\dint_0^\fz\lf[\lf|P^{[\lz]}_{\dz_1}P^{[\lz]}_{\dz_2}(f)(x_1, x_2)\r|+\lf|\rizo\lf(P^{[\lz]}_{\dz_1}P^{[\lz]}_{\dz_2} f\r)(x_1, x_2)\r|\r.\noz\\
&\quad\quad\quad\quad\quad\quad
+\lf|\rizo\lf(P^{[\lz]}_{\dz_1}P^{[\lz]}_{\dz_2} f\r)(x_1, x_2)\r|\noz\\
&\quad\quad\quad\quad\quad\quad\quad\quad\lf.+\lf|\rizo\rizt\lf(P^{[\lz]}_{\dz_1}P^{[\lz]}_{\dz_2} f\r)(x_1, x_2)\r|\r]^p\,\dmzd\noz \ls1,\noz
\end{align}
where the implicit constant is independent of $t_1$, $t_2$, $\dz_1$ and $\dz_2$.

Observe that for each $t_1,\,t_2$, $x_1,\,x_2\in\R_+$,
$$G(\dz_1,\,\dz_2,\,t_1,\,t_2,\,x_1,\,x_2)\to F(t_1, t_2, x_1, x_2)\quad {\rm as}\quad \dz_1,\,\dz_2\to0,$$
Indeed, observe that
$P^{[\lz]}_{\dz_1}P^{[\lz]}_{\dz_2} f\to f$ in $G(1,1;1,1)'$ as $\dz_1,\,\dz_2\to0$.
By these facts together with \eqref{lp norm F_1} and
the Fatou lemma, we further have \eqref{lp norm tilde G}.

Let $q\in({2\lz+1\over 2\lz+2}, p)$ and $r:=p/q$. As in the proof of \eqref{subharmonic for Fi},
we see that
\begin{equation}\label{harmo cont G}
F^q(\dz_1+t_1,\,\dz_2+ t_2,\,x_1,\,x_2)\ls \plzo\plzt\lf(F^q(\dz_1,\,\dz_2,\,\cdot,\,\cdot)\r)(x_1,x_2).
\end{equation}
Then by \eqref{lp norm tilde G}, there exist a subsequence  $\{F^q(\dz_{1,\,k},\,\dz_{2,\,j},\,\cdot,\,\cdot)\}_{\dz_{1,\,k};\,\dz_{2,\,j}>0}$ 
of $\{F^q(\dz_1,\,\dz_2,\,\cdot,\,\cdot)\}_{\dz_1;\,\dz_2>0} $
 and  $h\in L^r(\rlz)$ such that $\{F^q(\dz_{1,\,k},\,\dz_{2,\,j},\,\cdot,\,\cdot)\}_{\dz_{1,\,k};\,\dz_{2,\,j}>0}$ converges weakly to $h$
 in $L^r(\rlz)$ as $k,\,j\to\fz$, which further implies that
\begin{align}\label{lp norm weak limit}
\|h\|^r_\lrzd
&=\bigg\{\dsup_{\|g\|_{L^{r'}(\rlz)}\le1}\lim_{\gfz{k\to\fz}{j\to\fz}}
\bigg|\dint_0^\fz\dint_0^\fz g(x_1,x_2)F^q(\dz_{1,\,k},\,\dz_{2,\,j},\,x_1,x_2)\,\,d\mu_\lz(x_1,x_2)\bigg|\bigg\}^r\\
&\le \dlimsup_{\gfz{k\to\fz}{j\to\fz}}\big\|F^q(\dz_{1,\,k},\,\dz_{2,\,j},\,\cdot,\,\cdot)\big\|_\lrzd^r\noz\\
&\le\dlimsup_{\gfz{k\to\fz}{j\to\fz}}\big\|F(\dz_{1,\,k},\,\dz_{2,\,j},\,\cdot,\,\cdot)\big\|^p_\lpzd\noz\\
&\ls1.\noz
\end{align}
Moreover, by \eqref{harmo cont G}, we then have that
for any $t_1,\,t_2,\,x_1,\,x_2\in \R_+$,
\begin{align*}
F^q(t_1, t_2, x_1, x_2)
&\ls\plzo\plzt(h)(x_1, x_2).
\end{align*}
Therefore,
\begin{equation*}
[u^\ast(x_1, x_2)]^q\le\supd F^q(t_1,t_2, x_1, x_2)\ls\mathcal R_P(h)(x_1, x_2).
\end{equation*}
By this together with the $\lrzd$-boundedness of $\mathcal R_P$ and
\eqref{lp norm weak limit}, we then have
\begin{equation*}
\|u^\ast\|^p_\lpzd=\|(u^\ast)^q\|_\lrzd^r\ls\Big\|\mathcal R_P(h)\big|\Big\|^r_\lrzd\ls \|h\|^r_\lrzd\ls1.
\end{equation*}
This finishes the proof of Theorem \ref{thm Hp Riesz characterization}.
\end{proof}

%
%
%
%
%
%
%

\bigskip
\bigskip
{\bf Acknowledgments:}
X. T. Duong is supported by ARC DP 140100649.
J. Li is supported by ARC DP 160100153 and Macquarie University New Staff Grant.
B. D. Wick's research supported in part by National Science Foundation
DMS grants \#1603246 and \#1560955.
D. Yang is supported by the NNSF of China (Grant No. 11571289) and the State Scholarship Fund of China (No. 201406315078).

\end{document}